\documentclass[11pt,reqno]{amsart}

\usepackage[margin=1.2in]{geometry}
\usepackage{}
\usepackage{amssymb}

\usepackage{bbm}
\usepackage{cases}
\usepackage{amsmath}
\usepackage{graphicx}
\usepackage{mathrsfs}
\usepackage{stmaryrd}
\usepackage{color}
\usepackage{soul}
\usepackage[dvipsnames]{xcolor}
\usepackage{amsfonts}
\usepackage{cite}
\usepackage{enumerate,amsmath,amssymb,amsthm}

\numberwithin{equation}{section}

\newcommand{\be}{\begin{eqnarray}}
\newcommand{\mE}{\end{eqnarray}}
\newcommand{\ce}{\begin{eqnarray*}}
\newcommand{\de}{\end{eqnarray*}}
\newtheorem{theorem}{Theorem}[section]
\newtheorem{lemma}[theorem]{Lemma}
\newtheorem{remark}[theorem]{Remark}
\newtheorem{definition}[theorem]{Definition}
\newtheorem{proposition}[theorem]{Proposition}
\newtheorem{example}[theorem]{Example}
\newtheorem{corollary}[theorem]{Corollary}

\def\e{{\mathrm{e}}}
\def\eps{\varepsilon}

\def\a{\alpha}

\def\p{\partial}

\def\<{{\langle}}
\def\>{{\rangle}}

\def\bx{{\mathbf{x}}}
\def\tr{{\rm tr}}

\def\dif{{\mathord{{\rm d}}}}

\def\no{\nonumber}
\def\={&\!\!=\!\!&}
\def\bt{\begin{theorem}}
\def\et{\end{theorem}}
\def\bl{\begin{lemma}}
\def\el{\end{lemma}}
\def\br{\begin{remark}}
\def\er{\end{remark}}

\def\bd{\begin{definition}}
\def\ed{\end{definition}}
\def\bp{\begin{proposition}}
\def\ep{\end{proposition}}
\def\bc{\begin{corollary}}
\def\ec{\end{corollary}}
\def\bx{\begin{example}}
\def\ex{\end{example}}

\def\cB{{\mathcal B}}
\def\cC{{\mathcal C}}

\def\cL{{\mathcal L}}

\def\mE{{\mathbb E}}

\def\mL{{\mathbb L}}

\def\mR{{\mathbb R}}

\def\mW{{\mathbb W}}

\def\sB{{\mathscr B}}

\def\sD{{\mathscr D}}

\def\sG{{\mathscr G}}

\def\sK{{\mathscr K}}
\def\sL{{\mathscr L}}
\def\sM{{\mathscr M}}

\def\sP{{\mathscr P}}

\def\sR{{\mathscr R}}

\def\sT{{\mathscr T}}

\def\geq{\geqslant}
\def\leq{\leqslant}

\allowdisplaybreaks

\begin{document}

\title{Long-time behavior of McKean--Vlasov stochastic systems with singular coefficients}

\date{}

\author{Shangyun Hua, Tao Wang and Longjie Xie}

\address{Shangyun Hua:
	School of Mathematics and Statistics, Jiangsu Normal University,
	Xuzhou, Jiangsu 221000, P.R. China\\
	Email: syhua@jsnu.edu.cn
}

	\address{Tao Wang:
		School of Mathematics and Statistics, Jiangsu Normal University,
		Xuzhou, Jiangsu 221000, P.R. China\\
		Email:  taowang@jsnu.edu.cn}

\address{Longjie Xie:
	School of Mathematics and Statistics, Jiangsu Normal University,
	Xuzhou, Jiangsu 221000, P.R. China\\
	Email: longjiexie@jsnu.edu.cn}

\begin{abstract}
We develop a quantitative  framework for the long-time behavior of McKean--Vlasov stochastic differential equations  with singular coefficients and possible phase transitions. The framework separates the existence of invariant measures from their uniqueness.
For existence, we introduce a generalized Lyapunov condition with a one-level trapping mechanism, which requires dissipativity  at  only one admissible moment level rather than global contraction. This permits distribution dependence with supercritical growth  and thus  is more compatible with phase-transition models. For uniqueness and convergence, we establish an anchored uniqueness and quantitative
ergodicity principle. An explicit \(L^1\)-smallness condition, expressed through a convolution kernel built from  derivative estimates of the anchored frozen semigroup,  yields uniqueness and transfers
exponential or polynomial mixing rates of the frozen dynamics to the nonlinear
McKean--Vlasov system.
Moreover,  the principle is  sensitive to the choice of topology: different distances  lead to different perturbation kernels and hence different stability thresholds, reflecting the effective structure of the law dependence of a given model. A local version gives local uniqueness and quantitative attraction near a prescribed equilibrium, which is useful when several invariant
measures coexist.
We apply the framework to two representative models. For non-symmetric granular media dynamics, we derive two explicit uniqueness thresholds that reveal the topology-sensitive nature
of the theory: the total variation criterion exploits noise-induced  regularization  and applies to rough $L^p$-interaction kernels, whereas the Wasserstein-1 criterion captures the exact
dissipativity and mean-field interaction balance and gives sharper thresholds. For the dynamical Curie--Weiss model,   the Wasserstein-1 criterion recovers the sharp bifurcation threshold up to the critical equality. We identify
all invariant measures, establish basin-dependent exponential convergence in the
 phase-transition regime, and show that the loss of  anchored smallness at criticality
leads to polynomial slowing down.

\bigskip

  \noindent {{\bf AMS 2020 Mathematics Subject Classification:} 60H10; 35Q84; 60J60}

  \noindent{{\bf Keywords:} McKean--Vlasov SDEs; singular coefficients; invariant measures; quantitative  ergodicity;
 phase transitions.}
\end{abstract}

\maketitle

\tableofcontents

\section{Introduction}

Consider the  McKean--Vlasov stochastic differential equation (SDE for short) on $\mR^d$ $(d\geq 1)$:
\begin{align}\label{sde0}
X_t = \xi+\int_0^tb(X_s,\cL_{X_s})\dif s+\int_0^t\sigma(X_s,\cL_{X_s})\dif W_s,\quad  t\geq 0,
\end{align}
where   $(W_t)_{t\geq 0}$  is a  standard Brownian motion, $\xi$ is an initial random variable,  and the coefficients
$$
b: \mR^d\times\sP(\mR^d)\to\mR^d,\quad \sigma: \mR^d\times\sP(\mR^d)\to\mR^d\otimes\mR^d
$$
depend on both the state  and the current law of the solution. Here and throughout, $\sP(\mR^d)$ denotes the space of all probability measures on $\mR^d$, and $\cL_{\xi}$ denotes the distribution of a random variable $\xi$.
Such equations
arise naturally as mean-field limits of
interacting particle systems (the so-called propagation of chaos, see e.g., \cite{Hao2024,GLM,JW,Sz91}) and provide probabilistic representations of nonlinear
Fokker--Planck equations. They play a central role in nonlinear
probability,  statistical
mechanics,  stochastic control and the analysis of collective phenomena, we refer the reader  to \cite{Barbu2020,CD18,Daw83,WangRen2025,RZ2021} and the references therein.

\smallskip
A central problem  in the theory of McKean--Vlasov SDEs is to
understand the long-time behavior of solutions. In particular, the existence,
uniqueness, stability, and basin structure of invariant   measures. These questions have been studied from a variety of perspectives,
see, among others, \cite{BRS19,Cormier2025,Eberle2019,Liu2021,Guillin2022, Hammersley2021,Huang2025,Huang2026,HuangWang2025, Ren2021,Liang2021,Gvalani2020,MR}.
They become substantially more difficult
when the coefficients are singular or
when the nonlinear system exhibits phase transitions, see \cite{C2020,Delgadino2021,Delgadino2023,T,Zhang2025}.
For  McKean--Vlasov SDEs with regular coefficients and
sufficient contractivity, the existence and uniqueness of invariant measures
can often be obtained simultaneously by a contraction argument for the frozen
invariant-measure map; see, e.g.,
\cite{W18,Wang2023a,Zhang2023,T,Zhang2025}. Such frameworks, however, are not
suitable in phase-transition regimes, where several invariant
measures may coexist.
Moreover, most existing results impose restrictive assumptions: the coefficients are typically required to be  locally Lipschitz or satisfy a strong dissipativity condition uniformly over all probability measures, and the distribution dependence is often limited to subcritical or critical growth. These assumptions exclude a wide range of physically and biologically relevant models with   singular, non-symmetric, or  supercritical interactions.

\smallskip
Even when existence is established, the uniqueness and quantitative convergence to equilibrium remain challenging. Standard approaches based on coupling--synchronous or asynchronous--often require strong regularity of the coefficients or special structure properties of the system (e.g., one-sided monotonicity  of the drift, or the diffusion coefficient independent of the distribution), and yield only partial information on convergence rates. Moreover, the choice of the metric, whether weighted total variation, Wasserstein, or a weaker moment-based distance, is typically dictated by   available estimates rather than by the effective structure of the
law dependence of the model, and the resulting uniqueness conditions are often implicit and far from sharp.

\smallskip
{\bf Main contributions.}
The purpose of this paper is to develop a   framework that addresses these issues within a common perturbative approach.  The resulting theory provides existence without global
contraction, topology-sensitive anchored uniqueness criteria, quantitative
exponential or polynomial ergodicity, and local attraction in regimes with
multiple equilibria. The principal results  are summarized below.

\smallskip
{\it (i) Existence through one-level Lyapunov trapping.}
In Theorem \ref{main1}, we establish the existence of invariant   measures under weak  regularity  assumptions and a generalized law-dependent Lyapunov-type condition. The key innovation is a one-level trapping mechanism: instead of requiring the frozen invariant-measure map to send every prescribed moment ball into itself, it suffices to construct one admissible trapping level
$M_0$. This allows us to treat distribution dependence of supercritical growth and is therefore  more
flexible than a global
large-moment domination condition. Moreover, the framework accommodates coefficients with only local integrability (for the drift), local ellipticity and local H\"older regularity (for the  diffusion), and allows the diffusion coefficient to depend on the distribution--a feature that is often excluded in singular coefficient settings. For models with additional singular drift, we incorporate a Zvonkin transform to extend the existence result to the
$L^p$-class of drifts, see Corollary \ref{sing}.  All existence results are obtained without requiring uniqueness, this is essential in phase-transition regimes where existence  remains valid although uniqueness is false.

\smallskip
{\it (ii) A topology-sensitive anchored uniqueness and quantitative ergodicity principle.}
We  develop a perturbative principle relative to a fixed invariant measure $\mu_\ast$ of the frozen equation. Rather than imposing or verifying a uniform contraction property for all frozen equations, we  measure  the strength of the distribution dependence of the coefficients only relative to this reference equilibrium.
Under an explicit anchored $L^1$-smallness condition,  we prove in Theorem \ref{main2} that $\mu_\ast$ is the unique invariant
measure  and attracts all admissible initial distributions.
Furthermore, we show that the convergence rate is inherited from the ergodic behavior of the frozen dynamics, covering both exponential and polynomial mixing.

\smallskip
The smallness condition is  expressed through a convolution kernel built from derivative estimates of the frozen semigroup. A distinctive feature of our approach is that the criterion   is sensitive to the distance used to measure perturbations. In the drift-only case, weighted total variation estimates are sufficient; when the diffusion coefficient depends on the distribution, a weaker weighted H\"older distance becomes natural because the perturbation analysis involves Hessian bounds of the frozen semigroup, see Corollary \ref{tu}. In models where the law dependence acts through low-order moments, the Wasserstein-1 formulation can give a  sharper condition, see Corollary \ref{cor:W1-unique}, as demonstrated in the granular media and Curie--Weiss examples below. Thus the choice of topology is part of the stability mechanism: our framework does not merely prove convergence in different distances, it identifies the topology in which the perturbative structure of the model is most faithfully represented.

\smallskip
We also establish in Theorem~\ref{local-anchored} a local version of the
anchored principle. If the coefficient perturbation estimates and  the anchored smallness
condition  hold only in
a neighborhood of an invariant measure,  the theorem yields local
uniqueness and quantitative local attraction.  This local result is designed for phase-transition
regimes in which several invariant measures coexist and
no global uniqueness criterion can hold. Combined with additional information on
the basin geometry, such as a closed order-parameter equation, the local
principle provides a general
route to quantitative basin-dependent  ergodicity.

\smallskip
{\it (iii) Explicit criteria for non-symmetric granular media dynamics.}
We apply the abstract framework to the non-symmetric granular media equation
\[
\dif X_t = \left[-\alpha X_t + \int_{\mathbb R^d}F(X_t,y)\,\cL_{X_t}(\dif y)\right]\dif t + \sigma\,\dif W_t,
\]
where no symmetry, convolution structure or gradient form
 is imposed on $F$.
Such equations  and related models have been extensively studied in the literature, see, e.g. \cite{BGG13,Carrillo2006,CGM08,W18}. In Theorem \ref{coe}, we provide existence criteria covering
$L^p$-interactions and polynomial-growth interactions. For uniqueness and quantitative convergence, we derive two explicit criteria in Theorem \ref{ou} that reveal the topology-sensitive nature of the theory. The first, based on the  total variation framework, applies to rough $L^p$-kernels and exploits the regularization induced by non-degenerate noise. Its threshold depends explicitly on the $L^p$-norm of $F$ and the noise intensity: stronger noise permits  larger interactions. The second, based on the Wasserstein-1 framework, requires Lipschitz or one-sided Lipschitz structure  but yields a sharper dissipative condition
which is independent of the additive noise and reflects the exact competition between confinement and mean-field interaction. This contrast highlights the role of the distance choice and illustrates a robustness--sharpness trade-off, see Remark \ref{23}.

\smallskip
{\it (iv) Phase-transition, critical slowing down, and basin-dependent
ergodicity in the dynamical Curie--Weiss model.}
To further demonstrate the power of our framework in the presence of  non-uniqueness, we analyze the dynamical Curie--Weiss model:
\[
\dif X_t = -(\gamma+c)X_t\dif t + \sqrt c\,\tanh(\sqrt c\,\mathbb E X_t)\dif t + \dif W_t.
\]
In Theorem \ref{tan1}, we identify all invariant  measures and
prove that the bifurcation occurs at $\gamma=0$: uniqueness holds for $\gamma\ge 0$, whereas for $-c<\gamma<0$ exactly three invariant measures coexist.
Importantly, the Wasserstein-1 criterion developed above recovers this bifurcation threshold  up to the critical equality. This demonstrates that the abstract smallness condition is not merely sufficient, but can be essentially sharp in certain problems.

\smallskip

Theorem~\ref{tan2} describes the complete convergence picture in the uniqueness
and critical regimes. For \(\gamma>0\), the nonlinear dynamics converges
exponentially to the unique equilibrium   with  rate  \(\gamma\)
for non-zero initial mean. At the critical value \(\gamma=0\), the global anchored gap
vanishes. Initial data with zero
mean remain on the centered Ornstein--Uhlenbeck dynamics and converge
exponentially, whereas non-zero initial means relax at the optimal
polynomial rate \(t^{-1/2}\).
Thus the loss of strict anchored gap
corresponds  to critical slowing down.

\smallskip
In the phase-transition regime, Theorem~\ref{tan3} identifies the basins of
attraction through the sign of the initial mean. Positive and negative initial means converge to the corresponding non-zero equilibria, whereas  zero mean remains on  the  centered branch. We establish exponential convergence to each stable equilibrium with an explicit rate given by the slower of the Ornstein--Uhlenbeck relaxation and the stability exponent of the nonlinear mean dynamics.
At a non-zero equilibrium, the local anchored gap coincides with the linear stability
exponent of the nonlinear mean equation. This provides a unified
interpretation of global exponential stability, critical polynomial
relaxation, and symmetry breaking.

\smallskip
{\it (v) A unified analytical backbone based on  global SDE perturbation estimates.}
All the above results--existence, uniqueness, and basin-dependent ergodicity--rest on the quantitative global stability estimates for weak solutions of SDEs under perturbations of drift and diffusion coefficients, developed in Section~3. The  estimates are formulated in an abstract test-function framework and then specialized to weighted total variation and Wasserstein-1 settings. The stability estimate provides the essential link between the frozen reference dynamics and the nonlinear McKean--Vlasov flow, and its flexibility across different topologies is precisely what enables the sharp phase-transition analysis.  This unified approach contrasts with previous techniques and we believe it offers a versatile toolbox for the long-time analysis of a broad class of singular mean-field systems, well beyond the two concrete examples studied here. Moreover, although local ellipticity is assumed for technical convenience, the method extends naturally to degenerate and kinetic models whenever suitable Kolmogorov regularity and generalized It\^o formulas are available; see Remark~\ref{deg}.

\smallskip
The remainder of the paper is organized as follows. Section~2 introduces the
assumptions and states the main results. Section~3 develops global stability
estimates for weak solutions of SDEs under drift and diffusion perturbations.
Section~4 proves the existence of invariant measures through the one-level
Lyapunov trapping mechanism. Section~5 establishes the global and local
anchored uniqueness and ergodicity principles. Sections~6 and 7 apply the
theory to non-symmetric granular media dynamics and the dynamical
Curie--Weiss model, respectively.

\medskip
\noindent{\bf Notations.}
We collect  the notations used throughout the paper. We write \(C\) for a positive constant whose value may change from line to
line. When its dependence is relevant, it is indicated by subscripts.

For \(x\in\mathbb R^d\) and \(R>0\), let
\[
B_R(x):=\{y\in\mathbb R^d:|y-x|<R\},
\qquad
B_R:=B_R(0).
\]
For $p\in[1,\infty]$,  $L^p(B_R)$ denotes the usual space consisting of all $p$-integrable functions over $B_R$. When $R=\infty$, we write $L^p=L^p(\mR^d)$.
Let $Q=I\times\mathcal{O}\subset\mathbb{R}\times\mathbb{R}^{d}$, where $I\subset\mathbb{R}$ is an open interval and $\mathcal{O}\subset\mathbb{R}^d$ is an open domain. For $q,p\in[1,\infty]$, we define the anisotropic Lebesgue norm
$$
\|f\|_{\mathbb{L}^q_p(Q)}:=\|f\|_{L^q(I;L^p(\mathcal{O}))}
:=\left(\int_I\left(\int_{\mathcal{O}} |f(t,x)|^p\,\dif x\right)^{q/p}\dif t\right)^{1/q},
$$
with the usual modifications when $q=\infty$ or $p=\infty$.
Let $\mW^{1,2}_{q,p}(Q)$ denote the anisotropic Sobolev space consisting of all measurable functions $f$ on $Q$ such that $\partial_t f$ and $\nabla_x^2 f$ exist in the weak sense and
$$
\|f\|_{\mW^{1,2}_{q,p}(Q)}:=\|f\|_{\mathbb{L}^q_p(Q)}+\|\partial_t f\|_{\mathbb{L}^q_p(Q)}+\|\nabla_x^2 f\|_{\mathbb{L}^q_p(Q)}<\infty.
$$
We also introduce the local Sobolev space
$$
\mW^{1,2}_{q,p;\mathrm{loc}}(Q):=\big\{u:\ u\phi\in \mW^{1,2}_{q,p}(Q)\ \text{for all }\phi\in \mathbf{C}^\infty_c(Q)\big\}.
$$

Given a non-negative measurable function $V:\mR^d\to[0,\infty)$,
define
$$
\cB_V(\mR^d)
:=
\left\{
f:\mR^d\to\mR:\
\|f\|_{\cB_V}:=\sup_{x\in\mR^d}\frac{|f(x)|}{V(x)}<\infty
\right\}.
$$
For $\alpha\in(0,1)$, define the weighted H\"older space
$$
{\bf C}^\alpha_V(\mR^d)
:=
\left\{
f\in {\bf C}^\alpha_{\rm loc}(\mR^d):\
\|f\|_{{\bf C}^\alpha_V}<\infty
\right\},
$$
where
$$
\|f\|_{{\bf C}^\alpha_V}
:=
\sup_{x\in\mR^d}\frac{\|f\|_{{\bf C}^\alpha(B_1(x))}}{V(x)}.
$$
Recall that $\sP(\mR^d)$ denotes the space of all probability measures on $\mR^d$. We write
\begin{equation*}
	\mu(V):=\int_{\mR^d}V(x)\mu(\dif x).
\end{equation*}
and define
\begin{equation*}
	\mathscr P_V(\mR^d)
	:=
	\left\{\mu\in\mathscr P(\mR^d):\mu(V)<\infty\right\}.
\end{equation*}
For $M>0$, we denote the bounded subsets of $
\sP_V(\mR^d)$   by
$$
\sP_V^M(\mR^d):=\big\{\mu\in\sP_V(\mR^d): {\mu(V)}\leq M\big\}.
$$
When $V(x)=1+|x|^r$ with \(r>0\), we write
\[
\mathscr P_r(\mathbb R^d)
:=\mathscr P_{1+|x|^r}(\mathbb R^d),\qquad
\|\mu\|_r
:=
\left(
\int_{\mathbb R^d}|x|^r\,\mu(\dif x)
\right)^{1/r}.
\]
Weak convergence of probability measures is
written as
\[
\mu_n\Rightarrow\mu,
\]
meaning that
\(
\mu_n(f)\rightarrow \mu(f)
\)
for every bounded continuous function \(f\in {\bf C}_b(\mR^d)\). For two probability measures $\mu_1,\mu_2$ on $\mR^d$,
 the weighted
total variation distance is defined by
\begin{align*}
{\bf d}_V(\mu_1,\mu_2)=\sup_{\|f\|_{\mathcal B_V}\leq1}
\left|
\int_{\mathbb R^d} f(x)(\mu_1-\mu_2)(\dif x)
\right|
\end{align*}
Equivalently, whenever the right-hand side is finite,
\begin{align*}
{\bf d}_V(\mu_1,\mu_2)
=
\int_{\mR^d}V(x)|\mu_1-\mu_2|(\dif x).
\end{align*}
In particular, when \(V\equiv1\), \({\bf d}_V\) reduces to the total
variation distance. For $\alpha\in(0,1)$, define also the weaker distance
\begin{align}\label{alphaV}
{\bf d}_{\alpha,V}(\mu_1,\mu_2)
:=
\sup_{\|f\|_{\mathbf C^\alpha_V}\leq1}
\left|
\int_{\mathbb R^d}f(x)(\mu_1-\mu_2)(\dif x)
\right|.
\end{align}
Note that if $V(x)\geq 1+|x|$, then
$$
W_1(\mu_1,\mu_2)\leq C_{\a,V}\,{\bf d}_{\alpha,V}(\mu_1,\mu_2)
\leq
C_{\a,V}\,{\bf d}_{V}(\mu_1,\mu_2).
$$

Throughout the paper, an admissible initial value means a
random variable \(\xi\) for which the corresponding solution of the
stochastic equation is well defined and all assumptions used in the stated
result, such as the required moment bounds or ergodic estimates, are finite.
Its law \(\mathcal L_\xi\) is called an admissible initial distribution. We shall write $X_t(\xi)$ to stress the dependence on the initial value of the solution, and when $\xi=x\in\mR^d$, we also write $X_t(x)$. If $X_t$ is a Markov process and $P_t$ is the associated semigroup, then for
a probability $\nu$ we use the standard notation
\begin{align}\label{up}
\nu P_t(f):=\int_{\mathbb R^d}P_t f(x)\,\nu(\dif x).
\end{align}
Equivalently, \(\nu P_t\) denotes the law at time \(t\) of the
process $X_t(\xi)$ with initial distribution $\cL_{\xi}=\nu$.

\section{Assumptions  and  main results}

\subsection{Existence of invariant measures}
A probability measure $\mu$ is called an {\it invariant  measure} of the
McKean--Vlasov SDE \eqref{sde0} if, whenever $\cL_{\xi}={\mu}$, the solution satisfies
$$
\cL_{X_t(\xi)}\equiv\mu,\quad\forall t\geq0.
$$
Our objective is to develop an existence theory under assumptions that are weak
enough to accommodate singular coefficients, law-dependent diffusions, and
the coexistence of multiple invariant measures.

\subsubsection*{Abstract one-level Lyapunov structure}
We first impose the following local regularity and ellipticity condition.

\begin{enumerate}[$(\mathbf{H_0})$]
\item  There exist two functions $0<\lambda(x,\mu)\leq\Lambda(x,\mu)<\infty$ such that
\begin{align*}
\lambda(x,\mu)|\xi|^{2} \leq |\sigma(x,\mu) \xi|^2 \leq \Lambda(x,\mu) |\xi|^{2},\quad \forall x,\xi\in \mathbb{R}^{d},\mu\in\sP(\mR^d).
\end{align*}
Moreover, one of the following alternatives holds:
\begin{enumerate}[(i)]
		\item There exists $\alpha\in(0,1)$ such that for every
$\mu\in\mathscr P(\mathbb R^d)$,
\[
\sigma(\cdot,\mu),\, b(\cdot,\mu)
\in
{\bf C}^\alpha_{\rm loc}(\mathbb R^d).
\]

\item The coefficient $\sigma(x,\mu)\equiv\sigma(x)$ is independent of $\mu$, and there exist $\alpha\in(0,1]$ and $p\in(d,\infty]$ such that for every $\mu\in\sP(\mR^d)$,
 $$
 \sigma(\cdot)\in \mathbf C^\alpha_{\rm loc}(\mathbb R^d),\quad
 b(\cdot,\mu)\in L^p_{\rm loc}(\mathbb R^d).
 $$
\end{enumerate}

\end{enumerate}
\vspace{2mm}

	For every fixed  $\mu\in\mathscr P(\mR^d)$, consider the frozen autonomous SDE
\begin{equation}\label{000}
\dif X_t^{\mu}
	=
	b(X_t^{\mu},\mu)\dif t
	+
	\sigma(X_t^{\mu},\mu)\dif W_t,
\end{equation}
and denote its generator by
\begin{equation*}\label{fg}
\mathscr L_\mu f(x)
:=
b(x,\mu)\cdot\nabla_x f(x)
+
\operatorname{tr}
\left(
a(x,\mu)\cdot\nabla^2_x f(x)
\right),
\end{equation*}
where	
$$
a(x,\mu):=\tfrac12\sigma(x,\mu)\sigma(x,\mu)^*.
$$
Under $(\mathbf{H_0})$ and for every $\mu\in\sP(\mR^d)$, the
autonomous SDE (\ref{000}) admits a unique weak solution up
to its explosion time by the standard local theory,  see e.g. \cite{SV,XXZ26}.
The next assumption is the key Lyapunov input for the existence theorem.

\vspace{1mm}
\begin{enumerate}[$(\mathbf{H_1})$]
\item There exist a function $U\in \mathbf C^2(\mathbb R^d;[0,\infty))$  with compact level sets, i.e.,
\[
\{x\in\mathbb R^d: U(x)\leq R\}
\quad\text{is compact for every }R>0,
\] and a
lower semicontinuous function
$
V:\mathbb R^d\rightarrow(0,\infty)
$
with compact level sets,
constants $\kappa_0>0$, $c_0\geq0$, $M_0>0$, and finitely many
constants
\[
c_i>0,\qquad \alpha_i\in[0,1),\qquad \beta_i\geq0,
\qquad 1\leq i\leq N,
\]
such that, for every $x\in\mathbb R^d$ and every
$\mu\in\mathscr P_V^{M_0}(\mathbb R^d)$,
\begin{align}\label{HL-general}
\mathscr L_\mu U(x)
\leq
-\kappa_0V(x)
+
\sum_{i=1}^N
c_iV(x)^{\alpha_i}\mu(V)^{\beta_i}
+c_0,
\end{align}
and
\begin{align}\label{trapping-general}
c_0+\sum_{i=1}^Nc_iM_0^{\alpha_i+\beta_i}
\leq \kappa_0M_0.
\end{align}
\end{enumerate}

\begin{remark}
The inequality \eqref{HL-general} should be viewed as  a generalized Lyapunov-type condition  with nonlinear feedback through the law.
The
key structural requirement is the scalar trapping condition
\eqref{trapping-general}.  It
says that the dissipative term $-\kappa_0V$ dominates the positive distribution
dependent contributions  at a single admissible moment level $M_0$.  This is
substantially weaker than requiring such domination uniformly over all moment balls, which is  often needed in
standard fixed-point approaches and naturally leads to subcritical or critical growth restrictions, see Remark \ref{sup} below for a more detailed explanation.
\end{remark}

The next assumption concerns the continuity of the coefficients with respect to the weak topology on the trapping moment ball, which is formulated
in a rather weak localized form.

\vspace{1mm}
\begin{enumerate}[$(\mathbf{H_2})$]
\item There exists   $p\in(d,\infty]$ such that for every $R>0$ and
$
\mu_n,\mu\in \mathscr P_V^{M_0}(\mathbb R^d)$ with
$
\mu_n\Rightarrow\mu$,
one has
\begin{align*}
\lim_{n\to\infty}\Big(\|b(\cdot,\mu_n)-b(\cdot,\mu)\|_{L^p(B_R)} +\|a(\cdot,\mu_n)-a(\cdot,\mu)\|_{L^\infty(B_R)}\Big)=0.
\end{align*}
\end{enumerate}

\begin{remark}
Let $V$ be as in $(\mathbf H_1)$.
If $\mu_n,\mu\in\sP_V^{M_0}$ and $\mu_n\Rightarrow\mu$, then
\[
\mu_n(f)\longrightarrow\mu(f)
\]
for every continuous function $f:\mathbb R^d\to\mathbb R$
such that $|f(x)|/V(x)\to0$ as $|x|\to\infty$.
Indeed,
\[
\sup_{\nu\in\sP_V^{M_0}}
\int_{|x|>R}|f(x)|\,\nu(\dif x)
\leq
M_0\sup_{|x|>R}\frac{|f(x)|}{V(x)}
\longrightarrow0,
\]
and the assertion follows by continuous truncation and weak convergence.

In particular, if $q\geq1$ and $|x|^q=o(V(x))$, then
\[
\mu_n,\mu\in\sP_V^{M_0},\qquad
\mu_n\Rightarrow\mu
\quad\Longrightarrow\quad
W_q(\mu_n,\mu)\longrightarrow0.
\]
Thus continuity of the coefficients with respect to $W_q$
in those local norms is sufficient for $(\mathbf H_2)$.
\end{remark}

We now state the first main result.

\bt\label{main1}
Assume $(\mathbf{H_0})$-$(\mathbf{H_2})$ hold. Then the McKean--Vlasov SDE (\ref{sde0}) admits at least one  invariant   measure \(\mu\in\mathscr P_V^{M_0}(\mathbb R^d)\).
\et

\br
The significance of Theorem \ref{main1} lies not only in the weak assumptions on the coefficients, but also in the fact that the nonlinear
system (\ref{sde0}) is allowed to have multiple invariant  measures.  The proof
does not rely on global contraction property of the frozen invariant-measure map,
thus no uniqueness of equilibria for the nonlinear McKean--Vlasov dynamics can be guaranteed.  This separation between existence and uniqueness makes
the result particularly suitable for models exhibiting phase transitions, where
non-uniqueness of invariant measures is a genuine feature of the dynamics.
\er

\subsubsection*{Explicit subcritical, critical, and supercritical regimes}
To make the one-level trapping mechanism in $(\mathbf{H_1})$ more transparent, we single out a two-term
   version.
This formulation already captures the essential balance between confinement and distribution-dependent growth, while making explicit the distinction between subcritical, critical, and supercritical regimes.

\vspace{1mm}

\begin{enumerate}[$(\mathbf{\hat H_1})$]
\item There exist
$
U\in\mathbf C^2(\mathbb R^d;[0,\infty))
$
and lower semicontinuous function
$
V:\mathbb R^d\rightarrow(0,\infty)
$
both with compact level sets, and constants
$
\vartheta_1\in[0,1),
\vartheta_2\geq0,
\kappa_1>0,
\kappa_2,\kappa_3,\kappa_4\geq0$ and $M_0>0,
$
such that for any $x\in\mR^d$ and $\mu\in \mathscr P_V^{M_0}(\mathbb R^d)$,
\begin{align}\label{HL-two-term}
\mathscr L_\mu U(x)
\leq
-\kappa_1V(x)
+
\kappa_2V(x)^{\vartheta_1}
\left(
\mu(V)^{\vartheta_2}+\kappa_3
\right)
+\kappa_4,
\end{align}
and
\begin{align}\label{two-term-exact-trapping}
\kappa_2M_0^{\vartheta_1+\vartheta_2}
+
\kappa_2\kappa_3M_0^{\vartheta_1}
+
\kappa_4
\leq
\kappa_1M_0.
\end{align}
\end{enumerate}

\vspace{2mm}
Obviously, \eqref{HL-two-term} is a particular case of \eqref{HL-general}, and condition \eqref{two-term-exact-trapping} is exactly
\eqref{trapping-general} in this case. Thus we have the following immediate reduction.

\begin{proposition}
Assume $(\mathbf{\hat H_1})$. Then $(\mathbf H_1)$ holds with
$N=2$.
\end{proposition}

\begin{remark}\label{sup}
Define $\vartheta=\vartheta_1+\vartheta_2$, and
\[
\Phi(M)
:=
\kappa_1M
-\kappa_2M^{\vartheta}
-\kappa_2\kappa_3M^{\vartheta_1}
-\kappa_4.
\]
Then \eqref{two-term-exact-trapping} is equivalent to
\[
\Phi(M_0)\geq0
\]
for some admissible $M_0>0$. In particular:
\begin{enumerate}[(i)]
\item  If $\vartheta<1$ (subcritical regime), the distribution dependent part in (\ref{HL-two-term}) is
sublinear relative to the Lyapunov moment and is automatically dominated
by the dissipative term at large levels since $\Phi(M)\to+\infty$ as $M\to\infty$.

\item  If $\vartheta=1$ and $\kappa_2<\kappa_1$ (critical regime), domination still holds
since  again
      $\Phi(M)\to+\infty$ as $M\to\infty$.
\item If $\vartheta>1$ (supercritical regime), the positive term in (\ref{HL-two-term}) is supercritical and therefore
domination cannot hold for all large levels.  This is exactly where standard global fixed-point arguments break down. Nevertheless, our approach still applies whenever a trapping region  exists, i.e., whenever
      \[
      \sup_{M>0}\Phi(M)\geq0.
      \]
      For example, we can choose 	
      $$
      M_0
      =
      \left(\frac{\kappa_1}{\vartheta\kappa_2}\right)^{1/(\vartheta-1)},
      $$
      and assume that
      \begin{equation}\label{trapping-2}
       \kappa_4\leq(1-\vartheta^{-1})\kappa_1M_0-\kappa_2\kappa_3M_0^{\vartheta_1}.
      \end{equation}
   Then
      $\kappa_4\leq \kappa_1M_0-\kappa_2M_0^{\vartheta}-\kappa_2\kappa_3M_0^{\vartheta_1},$
      and $\Phi(M_0)\geq0$.
\end{enumerate}
Thus the existence theory developed here genuinely extends beyond the subcritical/critical threshold usually imposed in the literature.
\end{remark}

An explicit and sufficient coefficient-level criterion to ensure the assumption $(\mathbf{\hat H_1})$ is as follows.
	
	\begin{enumerate}[$(\mathbf{\tilde H_1})$]
		\item  There exist constants $r_3\geq r_1>0$,  $0\leq r_2<r_1$, $r_4\geq 0$, ${c}_1,{c}_2>0$, ${c}_3\geq 0$ and  $M_0>0$ such that for any $\mu\in\mathscr{P}_{1+|x|^{r_3}}^{M_0}(\mR^d)$,
		\begin{align*}
			2\<x,b(x,\mu)\>&+(1+r_3-r_1)\|\sigma(x,\mu)\|_{\mathrm{HS}}^2\\
			&\qquad\qquad\leq -{c}_1|x|^{r_1}+{c}_2|x|^{r_2}\cdot\|\mu\|_{r_3}^{r_4}+{c}_3.
		\end{align*}
		Moreover, one of the following alternatives holds.
		
		\begin{enumerate}[(i)]
			
			\item (subcritical case) $r_2+r_4<r_1$;
			
			\item (critical case) $r_2+r_4=r_1$ and ${c}_2<{c}_1$;
			
			\item (supercritical case) $r_2+r_4>r_1$,  and
			\begin{equation}\label{trapping-r}
				c_1+c_3M_0^{1+(r_2-r_1)/r_3}
				<
				\frac{\vartheta-1}{\vartheta}c_1M_0,
			\end{equation}		
 where
\[
\vartheta
:=
1+\frac{r_2+r_4-r_1}{r_3},
\qquad
M_0
:=
\left(
\frac{c_1}{\vartheta c_2}
\right)^{\frac{1}{\vartheta-1}}.
\]
		\end{enumerate}
	\end{enumerate}
	
\vspace{2mm}
We shall show  that under  $(\mathbf{\tilde H_1})$, the assumption $(\mathbf{\hat H_1})$ holds with
$$
U(x)=|x|^{2+r_3-r_1},\qquad V(x)=1+|x|^{r_3}.
 $$
 In this case, the continuity assumption $(\mathbf{H_2})$ reduces to the following explicit form.

\vspace{1mm}	
	\begin{enumerate}[$(\mathbf{\tilde H_2})$]
		\item For every $R>0$ and
$
\mu_n,\mu
\in
\sP^{M_0}_{1+|x|^{r_3}}(\mathbb{R}^d)
$
with
$
\mu_n\Rightarrow\mu,
$
one has
\begin{align*}
\lim_{n\to\infty}
\bigg[
&
\|b(\cdot,\mu_n)-b(\cdot,\mu)\|_{L^\infty(B_R)}
+
\|a(\cdot,\mu_n)-a(\cdot,\mu)\|_{L^\infty(B_R)}
\bigg]
=0.
\end{align*}
	\end{enumerate}
	
\vspace{2mm}
	We then obtain the following result.
	
	\begin{corollary}\label{cor1}
		Assume that $(\mathbf{H_0})$, $(\mathbf{\tilde H_1})$ and $(\mathbf{\tilde H_2})$ hold.
		Then the McKean--Vlasov SDE (\ref{sde0}) admits at least one invariant measure.
	\end{corollary}

The following example illustrates a supercritical situation in
which large-moment domination fails, but a finite trapping level can still
be identified.
\begin{example}
	Consider the  McKean--Vlasov SDE
	\begin{equation}\label{SDE}
		\dif X_t=
		\left[
		-2X_t+
		\frac14
		\frac{X_t}{\sqrt{1+X_t^2}}
		\mE(X_t^2)
		\right]\dif t
		+\sqrt{2}\dif W_t .
	\end{equation}
Note that
	\begin{align*}
		&2\<x,b(x,\mu)\>
		+ \|\sigma(x,\mu)\|_{\mathrm{HS}}^2\leq
		-4|x|^2
		+
		\frac12 |x|\|\mu\|_2^2
		+2 .
	\end{align*}
Therefore, we may choose
$$r_1=2,\quad
		r_2=1,\quad
		r_3=2,\quad
		r_4=2,\quad \theta=\frac{3}{2},$$
	and
$$
		c_1=4,\qquad
		c_2=\frac12,\qquad
		c_3=2.$$
	Since $
	r_2+r_4 > r_1,$
this falls into the supercritical regime of $(\mathbf{\tilde H_1})$. Here
\[
M_0=\left(\frac{4}{(3/2)(1/2)}\right)^2,
\]
and
\[
4+2M_0^{1/2}
<
\frac13\,4M_0.
\]
Hence \eqref{trapping-r} holds, and \eqref{SDE} admits at least one invariant measure.
\end{example}

\subsubsection*{Singular  drifts extension}
We next establish the existence of invariant   measures for McKean--Vlasov systems with an additional singular drift  by combining the generalized Lyapunov condition with the Zvonkin transformation. Namely, consider the following McKean--Vlasov SDE:
\begin{equation}\label{sde02}
X_t = \xi+\int_0^t\Big[b_0(X_s,\cL_{X_s})+b_1(X_s,\cL_{X_s})\Big]\dif s+\int_0^t\sigma(X_s,\cL_{X_s})\dif W_s,
\end{equation}
where $b_0, b_1: \mR^d\times\sP(\mR^d)\to\mR^d$ are measurable. We make the following assumption.

\vspace{2mm}

	\begin{enumerate}[$(\mathbf{\tilde H}^{\rm sing}_1)$]
\item
Assume that \(\sigma\) is bounded, uniformly elliptic and for some \(\alpha\in(0,1)\),
\[
\sup_{\mu\in\mathscr P_{r_3}}
\|a(\cdot,\mu)\|_{\mathbf C_b^\alpha}<\infty.
\]
The pair $(\sigma,b_0)$ satisfies
\((\mathbf H_0)\), \((\mathbf{\widetilde H}_1)\) and
\((\mathbf{\widetilde H}_2)\) with
\(
r_1>1\)
in either the subcritical case
or the critical case.
Moreover,
for every \(M>0\), there exists \(C_M>0\) such that
\begin{align}\label{sing-b0-growth}
\sup_{\|\mu\|_{r_3}\leq M}|b_0(x,\mu)|
\leq C_M(1+|x|^{r_1-1}),
\qquad x\in\mathbb R^d.
\end{align}
\end{enumerate}

\begin{enumerate}[$(\mathbf{\tilde H}^{\rm sing}_2)$]
\item
There exists \(p>d\) such that
\begin{align}
\sup_{\mu\in\mathscr P_{r_3}(\mR^d)}
\|b_1(\cdot,\mu)\|_{L^p}
<\infty ,
\label{sing-H1-2}
\end{align}
and
  for every \(R>0\) and $\mu_n, \mu\in \sP_{r_3}(\mR^d)$ with $\mu_n\Rightarrow\mu$, we have
\begin{align}
\lim_{n\to\infty}\|b_1(\cdot,\mu_n)-b_1(\cdot,\mu)\|_{L^p(B_R)}
=0 .
\label{sing-H2}
\end{align}
\end{enumerate}

\begin{corollary}
\label{sing}
Assume that $(\mathbf{\tilde H_1^{\rm sing}})$ and
$(\mathbf{\tilde H_2^{\rm sing}})$ hold.
Then the McKean--Vlasov SDE \eqref{sde02} admits at least one invariant   measure.
\end{corollary}

\begin{remark}
A central feature of Corollary \ref{sing} is that the Lyapunov
condition is imposed only on the pair $(\sigma, b_0)$, rather than on the full drift $b_0+b_1$. The singular term $b_1$ may not even be defined pointwise in a regular sense and may destroy the   Lyapunov assumption. It is handled additionally through the Zvonkin transform. This enlarges the range of admissible models and shows that the existence theory developed here is robust under singular lower-order perturbations.
\end{remark}

\subsection{Anchored uniqueness and quantitative ergodicity principle}
We now turn to uniqueness and convergence to equilibrium.
Let $\mu_\ast$ be an invariant  measure of the McKean--Vlasov SDE
\eqref{sde0}, whose existence follows from Theorem \ref{main1}. We fix this measure and use the frozen equation at
$\mu_\ast$ as a reference dynamics:
\begin{align}\label{ref-mustar}
	\dif \bar X_t
	=
	b(\bar X_t,\mu_\ast)\dif t
	+
	\sigma(\bar X_t,\mu_\ast)\dif W_t.
\end{align}
The corresponding Markov semigroup is denoted by
$$
P_t^\ast f(x):=\mathbb E f(\bar X_t(x)), \quad\forall t\geq 0.
$$
Our mechanism is to compare the original McKean--Vlasov dynamics (\ref{sde0}) with the anchored system (\ref{ref-mustar}).  The resulting anchored principle simultaneously controls uniqueness,
global or local attraction, and the transfer of convergence rates from the
reference dynamics (\ref{ref-mustar}) to the nonlinear McKean--Vlasov system (\ref{sde0}).

\subsubsection{Global anchored uniqueness and ergodicity}
We first formulate the global anchored principle. We shall show that, if the distributional dependence of the coefficients of the original
nonlinear equation is sufficiently small around $\mu_\ast$, then
$\mu_\ast$ is not only invariant but also unique and  attract all the  admissible initial distributions.
Here ``global'' means
that the coefficient perturbation estimates relative to \(\mu_\ast\) are
assumed for every admissible distribution, ``anchored" means that the
smallness is measured only relative to the reference equilibrium $\mu_\ast$, rather
than uniformly over all probability measures.

To formulate the result in a flexible and topology-sensitive way, let $\mathfrak F$ be a class of test functions equipped with a norm
$\|\cdot\|_{\mathfrak F}$, and define the associated dual distance on $\sP(\mR^d)$ by
\begin{align}\label{dF-new}
	{\bf d}_{\mathfrak F}(\mu,\nu)
	:=
	\sup_{\|f\|_{\mathfrak F}\leq 1}
	\left|
	\int_{\mathbb R^d} f(x)\,(\mu-\nu)(\dif x)
	\right|.
\end{align}
We always assume that ${\bf d}_{\mathfrak F}$ separates the probability measures in the admissible class, i.e., ${\bf d}_{\mathfrak F}(\mu,\nu)=0$ implies $\mu=\nu$. The role of \(\mathfrak F\) is twofold: it determines not only the topology in which convergence is proved,  but also the perturbation kernel required for the frozen semigroup and therefore the final smallness threshold. Typical examples to be discussed below are:
\begin{itemize}
	\item $\mathfrak F=\sB_V(\mR^d)$,  yielding the weighted total variation distance ${\bf d}_V(\mu,\nu)$. In particular, if $V$ is bounded, then ${\bf d}_{\mathfrak F}$ reduces to the total variation distance;
	\item  $\mathfrak F={\bf C}^\alpha_V(\mR^d)$ with $\alpha\in(0,1)$, yielding the weighted H\"older  dual distance ${\bf d}_{\alpha,V}(\mu,\nu)$ defined by (\ref{alphaV});
	\item $\mathfrak F={\text{Lip}}(\mR^d)$, yielding the Wasserstein
	distance $W_1$ through the Kantorovich--Rubinstein duality.
\end{itemize}

The anchored assumptions are divided into two parts.
The first concerns the ergodic and derivative estimates of the frozen reference semigroup, the second quantifies the size of the law-dependent perturbation relative to $\mu_\ast$. This  structure separates properties of the reference
semigroup from properties of the law dependence.

\begin{enumerate}[$({\mathbf U}^1_{\mathfrak F})$]
	\item
	\begin{enumerate}[(a)]
		\item The frozen semigroup $P_t^\ast$ is ergodic in the class
		$\mathfrak F$: there exists a decreasing function
		$\ell_\ast:\mR_+\to\mR_+$ with $\ell_\ast(t)\to0$ as $t\to\infty$, such that for every $f\in\mathfrak F$ and admissible initial distribution $\nu_0$,
		\begin{align}\label{F0}
			\left|
			\nu_0 P_t^\ast f -\mu_\ast(f)
			\right|
			\leq
			C_{\nu_0}\,\ell_\ast(t)\,\|f\|_{\mathfrak F},
			\qquad t\geq0,
		\end{align}
		where $\nu_0 P_t^\ast$ is defined by (\ref{up}), and $C_{\nu_0}$ is a constant depending on ${\nu_0}$.
		
		\item
		The following derivative estimates hold for
		$P_t^\ast$ in the class $\mathfrak F$: there exist locally bounded functions
		$
		\Theta_{1,\mathfrak F}^\ast,\Theta_{2,\mathfrak F}^\ast:\mathbb R^d\to[1,\infty)
		$
		and locally integrable kernels
		$
		\sK_{1,\mathfrak F}^\ast,\sK_{2,\mathfrak F}^\ast:\mathbb R_+\to\mathbb R_+
		$
		such that
		\begin{align}\label{F1}
			|\nabla_x P_t^\ast f(x)|
			\leq
			\Theta_{1,\mathfrak F}^\ast(x)
			\sK_{1,\mathfrak F}^\ast(t)
			\|f\|_{\mathfrak F}, \qquad t>0,\ x\in\mR^d,
		\end{align}
		and, when the diffusion coefficient $\sigma(x,\mu)$ depends on the distribution,
		\begin{align}\label{F2}
			|\nabla_x^2 P_t^\ast f(x)|
			\leq
			\Theta_{2,\mathfrak F}^\ast(x)
			\sK_{2,\mathfrak F}^\ast(t)
			\|f\|_{\mathfrak F}, \qquad t>0,\ x\in\mR^d.
		\end{align}
	\end{enumerate}
\end{enumerate}

\begin{remark}
	Since the reference equation (\ref{ref-mustar}) is a classical SDE, various criteria ensuring $({\mathbf U}^1_{\mathfrak F})$ are available in the literature. What matters for our purposes is that the precise form of the weights $\Theta_{1,\mathfrak F}^\ast(x)$, $\Theta_{2,\mathfrak F}^\ast(x)$ and the kernels  $\sK_{1,\mathfrak F}^\ast(t)$, $\sK_{2,\mathfrak F}^\ast(t)$ in (\ref{F1}) and (\ref{F2}) depend strongly on the chosen test-function class $\mathfrak F$. This dependence is  exactly what makes topology-dependent thresholds possible.
	\begin{enumerate}[(i)]
		\item In the weighted total variation/weighted H\"older-total variation framework, explicit formulas for $\Theta_{1,\mathfrak F}^\ast(x)$
		and $\Theta_{2,\mathfrak F}^\ast(x)$  are provided in \cite{XXZ26}, see also Lemma  \ref{esu} below. Meanwhile,  the
		kernels $\sK_{1,\mathfrak F}^\ast(t)$ and $\sK_{2,\mathfrak F}^\ast(t)$ arise from parabolic smoothing estimates of the
		frozen semigroup and typically display short-time singularities.
		In fact, one may take for $\mathfrak F=\sB_V(\mR^d)$,
		\begin{align}\label{k1}
			\sK_{1,\mathfrak F}^\ast(t)
			:=
			\mathbf 1_{(0,2]}(t)t^{-1/2}
			+
			\mathbf 1_{(2,\infty)}(t)\ell_\ast(t-1),
		\end{align}
		where $\ell_\ast$ is given exactly in (\ref{F0}), and when $\sigma(x,\mu)$ depends on the distribution,
		\begin{align}\label{k2}
			\sK_{2,\mathfrak F}^\ast(t)
			:=
			\mathbf 1_{(0,2]}(t)t^{-1+\alpha/2}
			+
			\mathbf 1_{(2,\infty)}(t)\ell_\ast(t-1),
		\end{align}
		with $\mathfrak F={\bf C}^\a_V(\mR^d)$.
		
		\item In a Lipschitz test-function class  or $W_1$ framework, the reference semigroup may instead be contractive. The resulting kernel can be purely exponential and free of
		short-time singularities. This difference is responsible for the sharper
		thresholds obtained in models with sufficient contractive structure, see Corollary \ref{cor:W1-unique} and Remark~\ref{rem:W1} below.
	\end{enumerate}
\end{remark}

The next assumption gives a concrete coefficient-level control of the
distributional perturbation relative to $\mu_\ast$.

\begin{enumerate}[$({\mathbf U}^2_{\mathfrak F})$]
	\item
	\begin{enumerate}[(a)]
		\item
		There exist constants $\kappa_b,\kappa_a\geq0$ and non-negative measurable
		functions $\Psi_b,\Psi_a$ such that, for every  $\mu\in\sP(\mR^d)$,
		\begin{align}\label{HF-b-new}
			|b(x,\mu)-b(x,\mu_\ast)|
			\leq
			\kappa_b \Psi_b(x)\,
			{\bf d}_{\mathfrak F}(\mu,\mu_\ast),
			\qquad x\in\mathbb R^d,
		\end{align}
		and
		\begin{align}\label{HF-a-new}
			|a(x,\mu)-a(x,\mu_\ast)|
			\leq
			\kappa_a \Psi_a(x)\,
			{\bf d}_{\mathfrak F}(\mu,\mu_\ast),
			\qquad x\in\mathbb R^d.
		\end{align}
		If $\sigma$ is independent of $\mu$, then we set $\kappa_a=0$ and the condition
		\eqref{HF-a-new} is omitted.
		
		\item There exist constants
		$\cC_{b,\mathfrak F},\cC_{a,\mathfrak F}<\infty$ such that for every admissible initial random variable \(\xi\), the solution
\(X_t=X_t(\xi)\) of \eqref{sde0} satisfies
		\begin{align}\label{HF-weight-b}
			\mathbb E\left[
			\Theta_{1,\mathfrak F}^\ast(X_t)\Psi_b(X_t)
			\right]
			\leq \cC_{b,\mathfrak F},
			\qquad t\geq0,
		\end{align}
		and when $\kappa_a>0$,
		\begin{align}\label{HF-weight-a}
			\mathbb E\left[
			\Theta_{2,\mathfrak F}^\ast(X_t)\Psi_a(X_t)
			\right]
			\leq \cC_{a,\mathfrak F},
			\qquad t\geq0.
		\end{align}
	\end{enumerate}
\end{enumerate}

\begin{remark}
	Assumption $({\mathbf U}_{\mathfrak F}^2)$-$(b)$ says that the spatial growth in the perturbation needs to be compatible with the
	derivative weights of the frozen semigroup.
	Thus the perturbation size is not only determined by the Lipschitz constants
	$\kappa_b,\kappa_a$, but also by the compatibility between the spatial growth
	of the coefficient perturbations and the weights appearing in the gradient and
	Hessian estimates.
	
	In particular, if the coefficients are global Lipschitz under ${\bf d}_{\mathfrak F}$, then one may take $\Psi_b=\Psi_a\equiv1$, and if the derivative weights   are uniformly bounded in the chosen class $\mathfrak F$,
	then the condition  $({\mathbf U}_{\mathfrak F}^2)$-$(b)$ follows immediately. The weighted
	formulation is more flexible: it permits unbounded spatial dependence
	provided that the corresponding weighted moments remain uniformly
	controlled.
\end{remark}

Define the abstract anchored convolution kernel
\begin{align}\label{Gamma-F}
	\Gamma_{\mathfrak F}(t)
	:=
	\kappa_b\cC_{b,\mathfrak F}\sK_{1,\mathfrak F}^\ast(t)
	+
	\kappa_a\cC_{a,\mathfrak F}\sK_{2,\mathfrak F}^\ast(t),
	\qquad t>0.
\end{align}
If $\sigma(x,\mu)\equiv\sigma(x)$ is independent of $\mu$, then
$\kappa_a=0$ and only the first term remains.
The quantity $\Gamma_{\mathfrak F}$ is the effective perturbation profile of the nonlinear equation around the reference equilibrium $\mu_\ast$.
Its $L^1$-size measures the total accumulated nonlinear feedback and  turns out to be the decisive threshold for uniqueness and convergence.

\begin{theorem}\label{main2}
	Assume that
	$({\mathbf U}^1_{\mathfrak F})$ and $({\mathbf U}^2_{\mathfrak F})$ hold.  If
	\begin{align}\label{small-F-main}
		\Lambda_{\mathfrak F}
		:=
		\int_0^\infty
		\Gamma_{\mathfrak F}(s)\dif s
		<1,
	\end{align}
	then $\mu_\ast$ is the unique invariant  measure for the McKean--Vlasov SDE (\ref{sde0}).  Moreover, for every
	admissible initial value $\xi$ with distribution ${\nu_0}$, we have
	\begin{align}\label{abstract-convergence}
		{\bf d}_{\mathfrak F}\big(\cL_{X_t},\mu_\ast\big)\to0,
		\qquad t\to\infty.
	\end{align}
	Meanwhile, the following quantitative estimates hold.
	
	\begin{enumerate}[(i)]
		\item
		If
		$
		\ell_\ast(t)\leq \e^{-\lambda t}
		$
		for some $\lambda>0$, and if there exists $\theta\in(0,\lambda]$ such that
		\begin{equation*}\label{exp0}
			\Lambda_{\mathfrak F,\theta}
			:=
			\int_0^\infty \e^{\theta s}\Gamma_{\mathfrak F}(s)\dif s<1,
		\end{equation*}
		then
		\begin{align}\label{exp1}
			{\bf d}_{\mathfrak F}\big(\cL_{X_t},\mu_\ast\big)
			\leq
			\frac{C_{\nu_0}}{1-\Lambda_{\mathfrak F,\theta}}
			\e^{-\theta t},
			\qquad t\geq0,
		\end{align}
		where $C_{\nu_0}$ is the constant in (\ref{F0}).
		
		\item
		If
		$
		\ell_\ast(t)\leq (1+t)^{-\gamma}
		$
		for some   $\gamma>0$, and if
		\begin{equation*}\label{pol0}
			\Lambda_{\mathfrak F,\gamma}
			:=
			\sup_{t\geq0}
			\int_0^t
			\left(\frac{1+t}{1+s}\right)^\gamma
			\Gamma_{\mathfrak F}(t-s)\dif s
			<1,
		\end{equation*}
		then
		\begin{align}\label{pol1}
			{\bf d}_{\mathfrak F}\big(\cL_{X_t},\mu_\ast\big)
			\leq
			\frac{C_{\nu_0}}{1-\Lambda_{\mathfrak F,\gamma}}
			(1+t)^{-\gamma},
			\qquad t\geq0.
		\end{align}
		Alternatively, the same polynomial rate holds if
		\[
		\int_0^\infty
		(1+s)^\gamma\Gamma_{\mathfrak F}(s)\dif s
		<\infty.
		\]
		More precisely,
		\[
		{\bf d}_{\mathfrak F}
		\big( \mathcal L_{X_t},\mu_\ast\big)
		\leq
		C_{\nu_0}
		\left(
		1+
		\int_0^\infty
		(1+s)^\gamma \sR_{\mathfrak F}(s)\dif s
		\right)
		(1+t)^{-\gamma},
		\]
		where
		\[
		\sR_{\mathfrak F}
		:=
		\sum_{n=1}^\infty
		\Gamma_{\mathfrak F}^{\ast n}\quad\text{and}\quad\int_0^\infty
		(1+s)^\gamma \sR_{\mathfrak F}(s)\dif s<\infty.
		\]
	\end{enumerate}
\end{theorem}

\begin{remark}
	Several important remarks are in order.
	
	\smallskip
	(i) First, we point out that condition \eqref{small-F-main} is an anchored gap assumption at $\mu_\ast$, which measures
	the strength of the distribution dependence only relative to the reference equilibrium $\mu_\ast$. It does not require a contraction estimate between
	all pairs of probability measures \((\mu,\nu)\), nor does it require uniform mixing estimates for the full family of frozen equations.
	This is  the principal structural
	difference from standard global contraction arguments.
	
	\smallskip
	(ii) Second, the theorem is  quantitative. Once the ergodic profile $\ell_\ast$ and the perturbation kernel $\Gamma_{\mathfrak F}$ are identified, the convergence rate of the nonlinear system follows explicitly. In this sense, the result shows that the long-time behavior of the nonlinear McKean--Vlasov equation is inherited from the frozen dynamics up to a computable perturbative correction.
	
	\smallskip
	(iii) Third, the criterion is topology-sensitive. The framework does not impose a canonical distance a priori; rather, it depends on how well
	the chosen test-function class captures the effective direction of the
	distributional dependence: the better this adaptation is or the more closely the chosen test-function class is adapted to the actual
	direction of the distributional dependence, the closer the resulting smallness
	condition is to the true phase-transition threshold. Hence the distance is part of the stability mechanism. In favorable models, this  recovers the sharp phase-transition scale, as will be seen in the Curie--Weiss example.
\end{remark}

\subsubsection*{ Two topology-adapted realizations.}
Below, we formulate two realizations of the above abstract framework: the weighted total variation and H\"older-dual version and the $W_1$ version.  In the dynamical Curie--Weiss model, we shall show that the $W_1$
criterion recovers the exact phase-transition threshold, up to the critical equality, revealing the sharpness of the condition \eqref{small-F-main}.

\subsubsection*{(i) Weighted total variation and H\"older-dual realization.}

We first record the weighted total variation/weighted H\"older-dual realization of    Theorem \ref{main2}.
This version is robust under weak regularity assumptions and
is particularly suited to singular coefficients.
It is also the appropriate framework when the distributional  dependence is performed in duality with  locally bounded or locally H\"older observables, rather than in a Wasserstein framework.

In this case, $({\mathbf U}^1_{\mathfrak F})$-(a) becomes the following assumption:

\vspace{1mm}
\begin{enumerate}[$({\mathbf U}^1_{V})$]
	\item Suppose that the frozen equation \eqref{ref-mustar}
	satisfies
	\begin{equation*}\label{tf}
		|P_t^\ast f(x)-\mu_\ast(f)|
		\leq
		\ell_\ast(t)(1+V(x))\|f\|_{\mathcal B_{1+V}},
		\qquad t\geq0.
	\end{equation*}
\end{enumerate}

\vspace{2mm}
By Lemma \ref{esu} below, there exist locally bounded functions $\Theta_1^\ast(x)$ and $\Theta_{2,\a}^\ast(x)$ such that
\begin{equation*}\label{tg}
	|\nabla_x P_t^\ast f(x)|
	\leq
	\Theta_1^\ast(x)\sK_1^\ast(t)
	\|f\|_{\mathcal B_{1+V}},
\end{equation*}
and, when the diffusion depends on the distribution,
\begin{equation*}\label{th}
	|\nabla_x^2 P_t^\ast f(x)|
	\leq
	\Theta_{2,\alpha}^\ast(x)\sK_{2,\alpha}^\ast(t)
	\|f\|_{\mathbf C^\alpha_{1+V}},
\end{equation*}
where $\a\in(0,1)$, and the kernels { $\sK_1^\ast(t):=\sK_{1,\mathfrak F}^\ast(t), \sK_{2,\alpha}^\ast(t):=\sK_{2,\mathfrak F}^\ast(t)$} are given explicitly by (\ref{k1}) and (\ref{k2}).
The perturbation assumption $({\mathbf U}^2_{\mathfrak F})$ becomes:

\vspace{1mm}
\begin{enumerate}[$({\mathbf U}^2_{V})$]
	\item
	\begin{enumerate}[(a)]
		\item
		There exist constants $\kappa_b,\kappa_a\geq0$ and weights $\Psi_b,\Psi_a\geq 0$ such that
		\begin{equation*}\label{tb}
			|b(x,\mu)-b(x,\mu_\ast)|
			\leq
			\kappa_b \Psi_b(x)\mathfrak D(\mu,\mu_\ast),
		\end{equation*}
		and
		\begin{equation*}\label{ta}
			|a(x,\mu)-a(x,\mu_\ast)|
			\leq
			\kappa_a \Psi_a(x)\mathfrak D(\mu,\mu_\ast),
		\end{equation*}
		where
		$$
		\mathfrak D={\bf d}_{1+V}
		$$
		if $\sigma(x,\mu)\equiv\sigma(x)$ is independent of $\mu$, and
		$$
		\mathfrak D={\bf d}_{\alpha,1+V}
		$$
		when the diffusion depends on the distribution.

		\item There exist constants
		$\cC_b, \cC_a<\infty$ such that for every $t\geq0$,
		\begin{equation*}\label{tc1}
			\mathbb E\big[
			\Theta_1^\ast(X_t)\Psi_b(X_t)
			\big]
			\leq \cC_b<\infty,
		\end{equation*}
		and, when $\kappa_a>0$,
		\begin{equation*}\label{tc2}
			\mathbb E\big[
			\Theta_{2,\alpha}^\ast(X_t)\Psi_a(X_t)
			\big]
			\leq \cC_a<\infty.
		\end{equation*}
	\end{enumerate}
\end{enumerate}

\vspace{2mm}
Set
\begin{align}\label{tg}
	\Gamma(t)
	:=
	\kappa_b \cC_b\sK_1^\ast(t)
	+
	\kappa_a\cC_a\sK_{2,\alpha}^\ast(t).
\end{align}
If
$\sigma$ is independent of $\mu$, then  $\kappa_a=0$. We then obtain the following direct corollary.

\begin{corollary}[Weighted total variation/H\"older-dual version]\label{tu}
	Assume that $({\mathbf U}^1_{V})$ and $({\mathbf U}^2_{V})$ hold.
	If
	\begin{equation*}
		\int_0^\infty\Gamma(s)\dif s<1,
	\end{equation*}
	then $\mu_\ast$ is the unique invariant  measure for the McKean--Vlasov SDE (\ref{sde0}) in
	$\mathscr P_V(\mR^d)$, and
	\begin{equation*}\label{t}
		\mathfrak D(\cL_{X_t},\mu_\ast)\to0,
		\qquad t\to\infty.
	\end{equation*}
	Moreover, the exponential and polynomial rates follow from Theorem
	\ref{main2} with $\Gamma_{\mathfrak F}=\Gamma$.
\end{corollary}

\br
The kernel $\Gamma(t)$  quantifies the competition
between the strength of the
distribution dependence and the regularization/mixing of the frozen dynamics (\ref{ref-mustar}). Its first part comes from drift perturbations through gradient estimates, while
the second appears only when the diffusion coefficient depends on the distribution and
is governed by Hessian estimates  of the frozen semigroup. In view of (\ref{k1}) and (\ref{k2}), both components combine
short-time smoothing with the long-time decay of the reference semigroup.  Thus faster
mixing or stronger
regularization of the  frozen equation (\ref{ref-mustar}) allows stronger distribution dependence.
\er

\begin{proof}
	This is   a direct specialization of Theorem \ref{main2}. In the drift-only case, one takes
	$$
	\mathfrak F=\mathcal B_{1+V},
	\qquad
	{\bf d}_{\mathfrak F}={\bf d}_{1+V}.
	$$
	When the diffusion depends on the distribution, one takes
	$$
	\mathfrak F=\mathbf C^\alpha_{1+V},
	\qquad
	{\bf d}_{\mathfrak F}={\bf d}_{\alpha,1+V}
	$$
	The corresponding kernels are exactly
	$\sK_1^\ast$ and $\sK_{2,\alpha}^\ast$ , and the abstract anchored kernel \eqref{Gamma-F} reduces to \eqref{tg}.  The conclusion therefore follows immediately from Theorem \ref{main2}.
\end{proof}

\subsubsection*{(ii) The Wasserstein-\(1\) realization}

For mean-field models whose law dependence acts  through low-order moments, the weighted total variation framework may lead to
non-sharp thresholds because the kernel $\Gamma(t)$ carries the short-time singularities inherited from derivative estimates.
We now  present  a $W_1$ realization of Theorem \ref{main2}: one tests against Lipschitz functions,  and the perturbation kernel is determined
by the Lipschitz contraction of the frozen dynamics and thus  has no short-time singularity in gradient estimate. In many concrete phase-transition models, this captures the exact dissipative scale of the frozen dynamics and leads to the sharp bifurcation threshold.

We reformulate the assumptions $({\mathbf U}^1_{\mathfrak F})$ and $({\mathbf U}^2_{\mathfrak F})$ as the following simplified version:

\begin{enumerate}[$({\mathbf U}^1_{\operatorname{Lip}})$]
	\item Assume that the frozen dynamics (\ref{ref-mustar}) is contractive in the Lipschitz class: there
	exists a locally integrable function $\ell_1:\mathbb R_+\to\mathbb R_+$ with
	$\ell_1(t)\to0$ as $t\to\infty$ such that, for every Lipschitz continuous
	function $f$,
	\begin{equation}\label{W1c}
		\operatorname{Lip}(P_t^\ast f)
		\leq
		\ell_1(t)\operatorname{Lip}(f),
		\qquad t\geq0 .
	\end{equation}
	When the diffusion depends on the law,
	suppose further that
	there exist a locally bounded function
	$\Theta_{2,W_1}^\ast:\mathbb R^d\to[1,\infty)$ and a locally integrable
	kernel $\ell_2:\mathbb R_+\to\mathbb R_+$ such that
	\begin{align*}
		|\nabla_x^2 P_t^\ast f(x)|
		\leq
		\Theta_{2,W_1}^\ast(x)\ell_2(t)\operatorname{Lip}(f),
		\qquad t>0,\ x\in\mathbb R^d .
	\end{align*}
\end{enumerate}

\begin{enumerate}[$({\mathbf U}^2_{\operatorname{Lip}})$]
	\item
	\begin{enumerate}[(a)]
		\item
		There exist constants $\kappa_b,\kappa_a\geq0$ and weights $\Psi_b,\Psi_a\geq 0$ such that
		\begin{equation*}\label{tb}
			|b(x,\mu)-b(x,\mu_\ast)|
			\leq
			\kappa_b \Psi_b(x)W_1(\mu,\mu_\ast),
		\end{equation*}
		and
		\begin{equation*}\label{ta}
			|a(x,\mu)-a(x,\mu_\ast)|
			\leq
			\kappa_a \Psi_a(x)W_1(\mu,\mu_\ast).
		\end{equation*}

		\item There exist constants
		$\cC_b, \cC_a<\infty$ such that for every $t\geq0$,
		\begin{equation*}\label{tc1}
			\mathbb E\big[\Psi_b(X_t)
			\big]
			\leq \cC_b<\infty,
		\end{equation*}
		and, when $\kappa_a>0$,
		\begin{equation*}\label{tc2}
			\mathbb E\big[
			\Theta_{2,W_1}^\ast(X_t)\Psi_a(X_t)
			\big]
			\leq \cC_a<\infty.
		\end{equation*}
	\end{enumerate}
\end{enumerate}

Note that (\ref{W1c}) implies that the frozen semigroup is ergodic in $W_1$ distance with rate $\ell_1(t)$. Set
\begin{align}\label{W1-Gamma}
	\Gamma_{W_1}(t)
	:=
	\kappa_b \cC_b \ell_1(t)
	+
	\kappa_a \cC_a \ell_2(t),
	\qquad t>0 .
\end{align}
If
$\sigma$ is independent of $\mu$, then  $\kappa_a=0$. The following result follows directly from Theorem \ref{main2}.

\begin{corollary}[$W_1$ version]\label{cor:W1-unique}
	Assume that \(({\mathbf U}^1_{\operatorname{Lip}})\) and
\(({\mathbf U}^2_{\operatorname{Lip}})\) hold.
	If
	\begin{align}\label{W1-small-general}
		\int_0^\infty \Gamma_{W_1}(s)\dif s
		<1,
	\end{align}
	then $\mu_\ast$ is the unique invariant  measure for the McKean--Vlasov SDE (\ref{sde0}) in
	$\mathscr P_V(\mR^d)$, and
	\begin{equation*}
		W_1(\mathcal L_{X_t},\mu_\ast)
		\longrightarrow0,
		\qquad t\to\infty .
	\end{equation*}
	Moreover, the exponential and polynomial rates follow from Theorem
	\ref{main2} with $\Gamma_{\mathfrak F}=\Gamma_{W_1}$ and $\ell_\ast(t)=\ell_1(t)$.
\end{corollary}

\begin{remark}
	Suppose that the diffusion coefficient is independent of the distribution, and  the drift perturbation is globally controlled by
	\[
	|b(x,\mu)-b(x,\mu_\ast)|
	\leq
	\kappa_b W_1(\mu,\mu_\ast),
	\]
	then one may take $\kappa_a=0$, $\Psi_b\equiv1$ and $\cC_b=1$.  Thus the smallness condition (\ref{W1-small-general})
	becomes
	\[
	\kappa_b\int_0^\infty \ell_1(s)\dif s<1.
	\]
	In particular, if $\ell_1(t)=\e^{-\lambda t}$, this reduces to
	$
	\kappa_b<\lambda.
	$
	This identity has a direct dynamical interpretation: the equilibrium is
	stable whenever the contractive response of the reference dynamics
	dominates the mean-field feedback. No loss is introduced by short-time
	smoothing estimates. In Ornstein--Uhlenbeck-type mean-field models, this
	balance may coincide with the intrinsic stability or bifurcation threshold.
\end{remark}

\begin{proof}
	This is a direct specialization of Theorem \ref{main2} with
	\[
	\mathfrak F=\mathrm{Lip}(\mathbb R^d),
	\qquad
	{\bf d}_{\mathfrak F}=W_1 .
	\]
	The perturbation kernel  \eqref{Gamma-F} becomes exactly
	(\ref{W1-Gamma}).
	The smallness condition \eqref{W1-small-general} is precisely  (\ref{small-F-main}).  Hence  the conclusion follows immediately from Theorem \ref{main2}.
\end{proof}

\begin{remark}[A robustness--sharpness trade-off]\label{off}
	The weighted total variation/H\"older-dual and Wasserstein criteria describe
	different stabilization mechanisms and should be viewed as complementary.
	The former exploits noise-induced regularization and therefore remains
	available for singular spatial coefficients, but the resulting kernel contains
	short-time  singularities. The latter
	requires additional Lipschitz or one-sided Lipschitz structure, but it
	directly captures the contraction scale of the frozen dynamics and may
	therefore yield a sharper threshold.
\end{remark}

\subsubsection{Local anchored uniqueness and attraction}
The anchored assumptions in Theorem~\ref{main2} are global in the
sense that the coefficient perturbation estimates around \(\mu_\ast\) in \(({\mathbf U}_{\mathfrak F}^2)\) (in particular,
(\ref{HF-b-new}) and (\ref{HF-a-new}))  are required to hold for every admissible probability measure. Such a condition
yields global uniqueness  and attraction from
all admissible initial values.
In a phase-transition regime, however, several invariant measures may
coexist, and no global anchored smallness condition can hold around any one
of them. We therefore formulate
a local version of the anchored  principle, which yields
local uniqueness and local attraction around an individual invariant
measure and is particularly useful in phase-transition regimes.

Let \(\mu_\ast\) be an invariant  measure and
\({\bf d}_{\mathfrak F}\) be a dual distance defined by (\ref{dF-new}). For \(r>0\), set
\[
B_r^{\mathfrak F}(\mu_\ast)
:=
\left\{
\mu\in\sP(\mR^d):
{\bf d}_{\mathfrak F}(\mu,\mu_\ast)<r
\right\}.
\]
We make the following assumption.

\begin{enumerate}[$({\mathbf U}^{\mathrm{loc}}_{\mathfrak F,r})$]
	\item Assume that, for some \(r>0\) the frozen semigroup at \(\mu_\ast\) satisfies
	\begin{equation}\label{local2}
		{\bf d}_{\mathfrak F}
		\big(
		{\nu_0} P_t^\ast,\mu_\ast
		\big)
		\leq
		\ell_\ast(t){\bf d}_{\mathfrak F}({\nu_0},\mu_\ast),
		\qquad t\geq0, \quad
		{\nu_0}\in B_r^{\mathfrak F}(\mu_\ast),
	\end{equation}
	where ${\nu_0} P_t^\ast$ is defined by (\ref{up}),
	\[
	L_\ast
	:=
	\sup_{t\geq0}\ell_\ast(t)<\infty,
	\qquad
	\ell_\ast(t)\longrightarrow0
	\quad\text{as }t\to\infty.
	\]
	The coefficient perturbation assumptions
	\(({\mathbf U}_{\mathfrak F}^2)\) hold for every
	\(\mu\in B_r^{\mathfrak F}(\mu_\ast)\), with a local perturbation kernel
	\[
	\Gamma_{\mathfrak F,r}(t)
	:=
	\kappa_b(r)\mathcal C_{b,\mathfrak F}(r)
	\sK_{1,\mathfrak F}^\ast(t)
	+
	\kappa_a(r)\mathcal C_{a,\mathfrak F}(r)
	\sK_{2,\mathfrak F}^\ast(t),
	\]
	with the second term omitted when $\kappa_a(r)=0$.
\end{enumerate}

\begin{theorem}[Local anchored uniqueness and attraction]
	\label{local-anchored}
	Assume $({\mathbf U}^{\mathrm{loc}}_{\mathfrak F,r})$ holds and
	\begin{equation*}
		\Lambda_{\mathfrak F,r}
		:=
		\int_0^\infty
		\Gamma_{\mathfrak F,r}(s)\,\dif s
		<1.
	\end{equation*}
	Then \(\mu_\ast\) is the unique invariant
measure in
	\(B_r^{\mathfrak F}(\mu_\ast)\).
	Moreover, let
	\begin{equation}\label{local-initial-radius}
		0<r_0<
		\frac{1-\Lambda_{\mathfrak F,r}}{L_\ast}\,r.
	\end{equation}
	If  \(\xi\) is an admissible initial random variable with $\nu_0:=\cL_{\xi}$ and
	\(
	{\bf d}_{\mathfrak F}(\nu_0,\mu_\ast)\leq r_0,
	\)
	then the distribution  $\mathcal L_{X_t}$ of the solution of McKean--Vlasov (\ref{sde0}) remains in
	\(B_r^{\mathfrak F}(\mu_\ast)\) for all \(t\geq0\), and
	\begin{equation*}
		{\bf d}_{\mathfrak F}(\mathcal L_{X_t},\mu_\ast)
		\longrightarrow0,
		\qquad t\to\infty.
	\end{equation*}
	
	Meanwhile, the exponential and polynomial estimates of
	Theorem~\ref{main2} \textup{(i)--(ii)}, including the
	resolvent-kernel alternative, remain valid with
	\[
	\Gamma_{\mathfrak F}
	\ \text{replaced by }\Gamma_{\mathfrak F,r},
	\qquad
	C_{\nu_0}
	\ \text{replaced by }
	{\bf d}_{\mathfrak F}(\nu_0,\mu_\ast),
	\]
	provided the corresponding decay and weighted-kernel
	conditions hold.
\end{theorem}

\begin{remark}
	(i) Theorem~\ref{local-anchored} provides  a quantitative attracting neighborhood
	contained in the basin of \(\mu_\ast\), but it does not  identify
	the entire basin. This  is intrinsic: local
	perturbation estimates cannot, by themselves, determine the global geometry
	of the nonlinear law flow. To obtain basin-dependent convergence from a larger set of
	initial conditions, one needs additional   information showing that the
	law flow eventually enters this local attracting neighborhood, then   the
	anchored estimate applies from that entrance time onward and yields a
	quantitative convergence rate. In models with
	a closed  order-parameter equation, Lyapunov monotonicity, or invariant sign regions, this entrance property
	can often be verified directly. The Curie--Weiss model below provides an
	explicit example.
	
	\smallskip
	(ii) The above result relies on the linear contraction estimate
	\eqref{local2}, which provides a uniform contraction factor
	\(L_\ast\). If only a weaker
	pointwise   estimate
	holds:
	\[
	{\bf d}_{\mathfrak F}({\nu_0} P_t^\ast,\mu_\ast)
	\le C_{\nu_0}\,\ell_\ast(t),
	\qquad t\ge0,\quad {\nu_0}\in B_r^{\mathfrak F}(\mu_\ast),
	\]
	with \(C_{\nu_0}\) depending on \({\nu_0}\) and \(\ell_\ast(t)\to0\), then the same proof of local uniqueness remains valid, and for any individual law \(\mathcal L_{X_t}\) that is known to stay inside \(B_r^{\mathfrak F}(\mu_\ast)\)
	for all times,  Lemma~\ref{gron} yields $
	{\bf d}_{\mathfrak F}(\mathcal L_{X_t},\mu_\ast)
	$ convergence to zero with the
	same rate as \(\ell_\ast\). However, without the contraction
	\eqref{local2}, one cannot guarantee that a small initial
	neighborhood remains invariant; the above theorem's uniform
	attracting neighborhood \(r_0\) is therefore lost. Thus the
	stronger assumption is essential for a genuine local stability
	theorem, whereas the weaker one suffices for conditional attraction
	when entrance into the local ball is known \emph{a priori}.
\end{remark}

\subsection{Non-symmetric singular granular media dynamics}

We now apply the abstract framework to a class of
non-symmetric granular media equations. Besides providing concrete existence
criteria under weak regularity and structural assumptions, this example provides a concrete comparison between two
topologies that detect different stabilization mechanisms.

More precisely, consider the McKean--Vlasov SDE
\begin{align}\label{gran}
	\dif X_t
	=
	\left[
	-\alpha X_t
	+
	\int_{\mathbb R^d}F(X_t,y)\,\mathcal L_{X_t}(\dif y)
	\right]\dif t
	+
	\sigma\,\dif W_t,
\end{align}
where $\alpha>0$, $\sigma$ is a constant non-degenerate matrix, and
$F:\mathbb R^d\times\mathbb R^d\to\mathbb R^d$ is a measurable interaction
kernel.
We first record two sufficient conditions for existence of invariant measures. The first   covers
rough interactions that are uniformly bounded in   local $L^p$
spaces, while the second allows polynomial growth in both the state and measure variables and is formulated through the one-level Lyapunov balance.

\begin{theorem}[Existence of invariant measures]
	\label{coe}
	The following assertions
	hold.
	
	\begin{enumerate}[(i)]
		\item {\it ($L^p$-interaction).}
		Assume that
		for some \(p\in(d,\infty]\),
		\begin{equation*}
			\sup_{y\in\mathbb R^d}
			\|F(\cdot,y)\|_{L^p}
			<\infty,
		\end{equation*}
		and for every \(R>0\), the map
		\(
		y\longmapsto F(\cdot,y)
		\)
		is continuous from \(\mathbb R^d\) into
		\(L^p(B_R)\).
		Then \eqref{gran} admits at least one invariant  measure.
		
		\item {\it (Interaction with polynomial growth).}
		Let $r\geq2$. Assume that there exist $\delta\in(0,r)$ and, for every
		$R>0$, a constant $C_R>0$ such that
		$$
		|F(x,y)|
		\leq
		C_R(1+|y|^\delta),
		\qquad |x|\leq R,
		$$
		and the map
		\(
		y\longmapsto F(\cdot,y)
		\)
		is continuous from \(\mathbb R^d\) into
		\(L^\infty(B_R)\).
		Assume moreover that there exist constants
		$
		0\leq r_2<2,  0\leq r_4<r,
		$
		and constants $c_1,c_2>0$, $c_3,c_4\geq0$, such that for every
		$\mu\in\mathscr P_r(\mathbb R^d)$,
		$$
		\begin{aligned}
			&
			2\left\langle x,
			-\alpha x+
			\int_{\mathbb R^d}F(x,y)\mu(\dif y)
			\right\rangle
			+
			(r-1)\|\sigma\|_{\mathrm{HS}}^2
			\\
			&\qquad\leq
			-c_1|x|^2
			+
			c_2|x|^{r_2}
			\left(
			\|\mu\|_r^{r_4}+c_3
			\right)
			+
			c_4.
		\end{aligned}
		$$
		Suppose , in addition, that one of the balance conditions in
		$(\tilde{\mathbf H}_1)$ is satisfied with
		\(
		r_1=2, r_3=r.
		\)
		Then (\ref{gran}) admits at least one invariant
		measure
		in $\mathscr P_r(\mathbb R^d)$.
	\end{enumerate}
\end{theorem}

We next turn to uniqueness and convergence to equilibrium.
The two criteria
below rely on different mechanisms, depending on which topology is used to measure the law dependence. This is where the abstract theory reveals its strongest structural message: the optimal distance is model-dependent, and different choices capture different mechanisms of stabilization. The first criterion is formulated in total variation and uses the regularization effect of the non-degenerate noise, whereas the second criterion  is formulated in \(W_1\) and exploits synchronous contraction of the frozen  dynamics.

\begin{theorem}[Topology-adapted uniqueness and ergodicity]
	\label{ou}
	The following assertions hold.
	
	\begin{enumerate}[(i)]
		
		\item {\it (Total variation criterion under an $L^p$-interaction).}
		Let $p\in(d,\infty]$ and set
		$
		q:=p/(p-1),
		$ with the convention $q=1$ when $p=\infty$.
		Assume that
		\begin{align}\label{gran-Mp}
			M_p
			:=
			\sup_{y\in\mathbb R^d}
			\|F(\cdot,y)\|_{L^p(\mathbb R^d)}
			<\infty.
		\end{align}
		Let $\mathfrak C_{p,M_p}$ and $A_{\alpha,\sigma,p}$ be constants given by \eqref{invariant-density-constant} and (\ref{Ap}). If
		\begin{align}\label{gran-Lp-uniqueness-small}
			M_p
			\left(
			A_{\alpha,\sigma,p}
			+
			\mathfrak C_{p,M_p}
			A_{\alpha,\sigma,\infty}
			\right)
			<1,
		\end{align}
		then \eqref{gran} admits a unique invariant
		measure
		$\mu_\ast$.
		Moreover, for every  admissible initial value $\xi$,
		there exist  constants
		$C>0$ and $\lambda>0$ such that
		$$
			\|\mathcal L_{X_t}-\mu_\ast\|_{\rm TV}
			\leq
			C \e^{-\lambda t},
			\qquad t\geq0.
			$$
		
		In particular, when $p=\infty$, one has
		\[
		\mathfrak C_{\infty,M_\infty}=1,
		\qquad
		A_{\alpha,\sigma,\infty}
		=
		\|\sigma^{-1}\|\sqrt{\frac{\pi}{\alpha}},
		\]
		and \eqref{gran-Lp-uniqueness-small} reduces to
		\[
		\|F\|_\infty
		<
		\frac12
		\|\sigma^{-1}\|^{-1}
		\sqrt{\frac{\alpha}{\pi}}.
		\]

		\item \textbf{$W_1$-criterion.}
		Assume that
		$$
		\mathfrak{L}_y
		:=
		\sup_{x\in\mathbb R^d}
		\operatorname{Lip}\big(F(x,\cdot)\big)
		<\infty,
		$$
		and
		$$
		\mathfrak{L}_x
		:=
		\sup_{x,y\in\mathbb R^d}
		\lambda_{\max}
		\left(
		\frac{\nabla_xF(x,y)+\nabla_xF(x,y)^\ast}{2}
		\right)
		<\infty.
		$$
		If
		\begin{align}\label{sh}
			\mathfrak{L}_x+\mathfrak{L}_y<\alpha,
		\end{align}
		then \eqref{gran} admits a unique invariant
		measure $\mu_\ast$. Moreover, for every initial value $\xi$,
		$$
		W_1(\mathcal L_{X_t},\mu_\ast)
		\leq
		\e^{-(\alpha-\mathfrak{L}_x-\mathfrak{L}_y)t}
		W_1(\cL_{\xi},\mu_\ast),
		\qquad t\geq0.
		$$
		In particular,
		\begin{enumerate}
			\item if
			$
			F(x,y)=\bar F(x-y),
			$
			then (\ref{sh}) is implied by
			$
			\|\nabla \bar F\|_\infty<\alpha/2.
			$
			\item If $F(x,y)=\bar F(y),$ then (\ref{sh}) is implied by
			$
			\|\nabla\bar F\|_\infty<\alpha.
			$
		\end{enumerate}
	\end{enumerate}
\end{theorem}

\begin{remark}[Different roles of the noise]\label{23}
	An important feature of Theorem \ref{ou} is that the two uniqueness criteria
	respond differently to the noise intensity.  The  total variation
	framework exploits the
	regularization for rough kernels induced by the non-degenerate noise. In the bounded case \(p=\infty\), the
	admissible size is
	\[
	\|F\|_\infty
	<
	\frac12
	\|\sigma^{-1}\|^{-1}
	\sqrt{\frac{\alpha}{\pi}} .
	\]
	This
	dependence is natural: a larger diffusion matrix  gives stronger regularization of rough observables and thus allows larger  interactions.
	
	By contrast, the $W_1$ criterion reads
	\[
	\mathfrak{L}_x+\mathfrak{L}_y<\alpha,
	\]
	which is independent of the constant diffusion matrix \(\sigma\).  This is
	because the $W_1$ estimate relies on the Lipschitz contraction of the frozen
	dynamics, and   the additive noise cancels in synchronous coupling.  Hence the contraction rate is determined only by the one-sided
	dissipativity of the drift, namely the competition between the linear
	confinement \(-\alpha x\), the spatial one-sided Lipschitz constant \(\mathfrak{L}_x\),
	and the law-dependence strength \(\mathfrak{L}_y\). This contrast gives a concrete instance of the
	robustness--sharpness trade-off mentioned in Remark \ref{off}.
\end{remark}

\subsection{Dynamical Curie--Weiss model: phase transition and basin-dependent ergodicity}

We  finally consider a model in which the phase-transition
mechanism can be analyzed explicitly.  This example serves three purposes.
First, it confirms that the existence theorem developed before is
compatible with non-uniqueness of invariant   measures. There are parameter regimes in which several invariant measures coexist, and thus the global contraction argument can not be used.
Second, it shows that the anchored uniqueness criterion (more precisely, its
$W_1$-realization, since the law dependence acts through the mean) recovers the exact phase-transition threshold,  up to the critical equality.
Third, it demonstrates that even in the phase-transition regime,
the local anchored principle (Theorem~\ref{local-anchored}), combined with a closed mean equation and the stability estimates in Section 3, enables us to obtain basin-dependent quantitative
ergodicity.

More precisely, consider a dynamical version of the Curie--Weiss mean-field model in $\mR$:
\begin{align}\label{ex2}
\dif X_t
=
-(\gamma+c)X_t\dif t
+
\sqrt c\,\tanh\left(\sqrt c\,\mE X_t\right)\dif t
+
\dif W_t,
\qquad X_0=\xi,
\end{align}
where $c>0$, $\gamma>-c$, and the hyperbolic tangent function is defined by
$$
\tanh x := \sinh x / \cosh x,\quad\text{where}\quad \sinh x:=(\e^x-\e^{-x})/2,\quad \cosh x:=(\e^x+\e^{-x})/2.
$$
The parameter $c$ measures the strength of the mean-field interaction, whereas
$\gamma+c$ is the dissipation rate of the frozen Ornstein--Uhlenbeck dynamics. Their competition determines whether the system has
a unique equilibrium or undergoes a phase transition. We first establish the following existence and phase transition result.

\begin{theorem}[Invariant measures and phase transition]\label{tan1}
Let $c>0$ and $\gamma>-c$.  Then the McKean--Vlasov SDE \eqref{ex2} admits at least one invariant
 measure. Every invariant measure is of the form
\begin{align}\label{inv-tanh}
\mu_m(\dif x)
=
\sqrt{\frac{\gamma+c}{\pi}}
\exp\Big\{-(\gamma+c)(x-m)^2\Big\}\dif x,
\end{align}
where  $m\in\mR$ is the solution of the self-consistency equation
\begin{align}\label{self-con-tanh}
(\gamma+c)m=\sqrt c\,\tanh(\sqrt c\,m).
\end{align}
More precisely, the following assertions hold.

\begin{enumerate}[(i)]

\item If $\gamma\geq0$, then \eqref{ex2} has a unique invariant
measure  $\mu_{0}$ given by (\ref{inv-tanh}) with $m=0$.

\item If $-c<\gamma<0$, then \eqref{ex2} has exactly three invariant
measures
$
\mu_{-},
\mu_{0}$ and $
\mu_{+},
$
corresponding to three solutions  of \eqref{self-con-tanh}:
$$
m_-<0,
\qquad
m_0=0,
\qquad
m_+>0,
$$
and $m_+=-m_-$. The equilibria \(m_-\) and \(m_+\) are stable,
whereas \(m_0=0\) is unstable.
\end{enumerate}
\end{theorem}

\begin{remark}[The global anchored gap]
Set
\[
h(m):=\sqrt c\,\tanh(\sqrt c\,m).
\]
For $i\in\{-,0,+\}$, freezing the distribution of the solution $X_t$ as $\mu_i$,  the resulting frozen reference equation is
\begin{align}\label{ouf}
\dif\bar X_t^{\,i}
=
-(\gamma+c)(\bar X_t^{\,i}-m_i)\dif t+\dif W_t.
\end{align}
Let $\bar P_t^i$ be the corresponding semigroup. Then it holds that
\[
W_1(\mu \bar P_t^i,\nu \bar P_t^i)
\leq
\e^{- (\gamma+c)t}W_1(\mu,\nu),
\qquad t\geq0.
\]
On the other hand,
\[
h'(m)
=
c\,\operatorname{sech}^2(\sqrt c\,m),
\qquad
\sup_{m\in\mathbb R}|h'(m)|=c.
\]
Hence the global anchored condition in
Corollary~\ref{cor:W1-unique} becomes
\[
c<\gamma+c\,\,\Leftrightarrow\,\,\gamma>0.
\]
Thus the \(W_1\)-criterion recovers the uniqueness regime up to the
critical equality \(\gamma=0\). At criticality the strict contraction
gap vanishes, and the exponential relaxation is replaced by the
polynomial rate \(t^{-1/2}\) as shown in
Theorem~\ref{tan2} below.
\end{remark}

We next consider the ergodicity of the system (\ref{ex2}) in the uniqueness regime  $\gamma\geq0$. By Theorem \ref{tan1}, the
unique invariant measure is\[
\mu_0(\dif x)
=
\sqrt{\frac{\gamma+c}{\pi}}
\exp\big\{-(\gamma+c)x^2\big\}\dif x.
\]
Meanwhile, in view of (\ref{ouf}), the corresponding frozen  equation at $\mu_0$ is the centered
Ornstein--Uhlenbeck system
\[
\dif\bar X_t=-(\gamma+c)\bar X_t\dif t+\dif W_t.
\]
Let $V\geq1$ be a weight for which $\cL_{\bar X_t}$ converges to $\mu_0$ in ${\bf d}_V$ with an exponential rate
$\lambda_V\in(0,\gamma+c]$, namely,
\begin{align}\label{cw1}
{\bf d}_V(\cL_{\bar X_t},\mu_0)
\leq
C_0\,\e^{-\lambda_Vt},\qquad t\geq0.
\end{align}
For bounded $V$, and in particular for $V\equiv1$ so that ${\bf d}_V$ reduces to total variation, one may take
$\lambda_V=\gamma+c$.

\begin{theorem}[Ergodicity in the uniqueness regime and critical slowing down]
\label{tan2}
Assume $\gamma\geq 0$, $\mE |\xi|<\infty$ and $\mE V(\xi)<\infty$. The following assertions hold.

\begin{enumerate}[(i)]

\item If $\gamma>0$, then for every  $0<\lambda<\lambda_V\wedge\gamma$,
 there exists a constant $C_{1}>0$ such that
\begin{align}\label{cw-positive-gamma-exp}
{\bf d}_V(\cL_{X_t},\mu_0)
\leq
C_{1}\,\e^{-\lambda t},
\qquad t\geq0.
\end{align}
Moreover, if
$\lambda_V>\gamma$, then the endpoint
$\lambda=\gamma$ is also admissible.

\item If $\gamma=0$, then there exists a constant $C_{2}>0$ such that
\begin{align}\label{cw-critical-poly}
{\bf d}_V(\cL_{X_t},\mu_0)
\leq
 C_{2}(1+t)^{-1/2},
\qquad t\geq0.
\end{align}
More precisely:
\begin{enumerate}[(a)]

\item If $\mE\xi=0$, then there exists a constant
$ C_{3}>0$ such that
\begin{align}\label{cw-critical-zero-mean-exp}
{\bf d}_V(\cL_{X_t},\mu_0)
\leq
C_{3}\,\e^{-\lambda_Vt},
\qquad t\geq0.
\end{align}

\item If $\mE\xi\neq0$, then the polynomial rate in
\eqref{cw-critical-poly} is optimal. In particular, for
$V\equiv1$, we have
\begin{align}\label{critical-exact-limit}
\lim_{t\to\infty}
\sqrt t\,
\|\cL_{X_t}-\mu_0\|_{\rm TV}
=
\sqrt{\frac{6}{\pi c}}.
\end{align}
\end{enumerate}
\end{enumerate}
\end{theorem}

\begin{remark}
The formulation of Theorem \ref{tan2} in the weighted total variation distance
is one realization of a general principle underlying the proof.
More generally, let ${\bf d}_{\mathfrak F}$ be a dual distance in which the
frozen Ornstein--Uhlenbeck semigroup   converges to $\mu_{0}$   and for which the corresponding
stability estimate is available,
  then the nonlinear process converges to $\mu_{0}$ in the same
distance. This transfer principle highlights the usefulness of the global stability
estimates established in Section 3: they
provide a mechanism for lifting
known long-time properties of a tractable frozen reference process to the
nonlinear McKean--Vlasov dynamics.
\end{remark}

Finally, we consider the phase-transition regime $-c<\gamma<0$. Let
\[
m_-<m_0=0<m_+
\]
be the three solutions of \eqref{self-con-tanh}, and let
$\mu_i:=\mu_{m_i}$, $i\in\{-,0,+\}$ be the corresponding invariant
measures given by \eqref{inv-tanh}. The corresponding frozen equations at  $\mu_i$ are given by (\ref{ouf}), i.e.,
\[
\dif\bar X_t^{\,i}
=
-(\gamma+c)(\bar X_t^{\,i}-m_i)\dif t+\dif W_t.
\]
For $i\in\{-,0,+\}$, let $V\geq1$ be a weight for which the law $\cL_{\bar X_t^i}$ converges to $\mu_i$:
\begin{equation}\label{OU1}
{\bf d}_V(\cL_{\bar X_t^i},\mu_i)
\leq
C_i\,\e^{-\lambda_Vt},\qquad t\geq0.
\end{equation}
for some $\lambda_V\in(0,\gamma+c]$.
For a stable solution $m_i$ of \eqref{self-con-tanh}, i.e.  $i\in\{-,+\}$, define the exponential
stability exponent of the mean equation by
$$
\lambda_i^{\rm mean}
:=\gamma+c\tanh^2(\sqrt c\,m_i).
$$
By symmetry, $\lambda_+^{\rm mean}=\lambda_-^{\rm mean}:=\lambda^{\rm mean}$. We shall show that  $\lambda^{\rm mean}>0$
in Section 7. We have the following
basin-dependent ergodicity statement.

\begin{theorem}[Basin-dependent ergodicity]
\label{tan3}
Assume $-c<\gamma<0$, \(\mathbb E|\xi|<\infty\)  and $\mE V(\xi)<\infty$.  The basin of attraction is determined by the sign of the initial mean. More precisely, the following assertions hold:

\begin{enumerate}[(i)]

\item If $\mE\xi>0$, then for $0<\lambda<\lambda_V\wedge\lambda^{\rm mean}$, there exists a constant
$C_{1}>0$ such that
\begin{align}
\label{r1}
{\bf d}_V(\cL_{X_t},\mu_{+})
\leq
C_{1}\,\e^{-\lambda t},
\qquad t\geq0.
\end{align}

\item If $\mE\xi<0$, then for $0<\lambda<\lambda_V\wedge\lambda^{\rm mean}$,
    there exists a constant
$C_{1}>0$ such that
\begin{align}
\label{r2}
{\bf d}_V(\cL_{X_t},\mu_{-})
\leq
C_{1}\,\e^{-\lambda t},
\qquad t\geq0.
\end{align}

\item If $\mE\xi=0$, then there exists a constant $C_{2}>0$ such that
\begin{align}
\label{r0}
{\bf d}_V(\cL_{X_t},\mu_{0})
\leq
C_{2}\,\e^{-\lambda_Vt},
\qquad t\geq0.
\end{align}

\end{enumerate}
Moreover, if $\lambda_V>\lambda^{\rm mean}$, the endpoint $\lambda=\lambda^{\rm mean}$ is also admissible in \textup{(i)} and \textup{(ii)}.
\end{theorem}

\begin{remark}[Local anchored gap and linear stability]
At an invariant mean \(m_i\), the infinitesimal anchored gap is
\[
\gamma+c-h'(m_i)
=
\gamma+c\tanh^2(\sqrt c\,m_i).
\]
For \(i\in\{-,+\}\), this quantity is positive and equals
\(\lambda_i^{\rm mean}\), the linear stability exponent of the nonlinear
mean equation.
At the centered equilibrium,
\[
\gamma+c-h'(0)=\gamma<0.
\]
Hence the sign of the anchored gap distinguishes the three regimes:
\[
\begin{array}{c|c}
\text{parameter regime}
&
\text{dynamical behavior}
\\ \hline
\gamma>0
&
\text{global exponential stability},
\\
\gamma=0
&
\text{critical polynomial relaxation},
\\
-c<\gamma<0
&
\text{instability of the centered equilibrium}.
\end{array}
\]
Thus, the local anchored principle  gives a measure-level interpretation of
the linearized mean stability, while the scalar mean equation identifies
the global basins of attraction.
\end{remark}

\section{Global stability estimates for weak solutions of SDEs}

This section establishes quantitative, global-in-time stability estimates for weak solutions of SDEs under perturbations of both drift and diffusion coefficients. These estimates
are the main analytical tool of the paper. They compare a time-inhomogeneous
perturbed equation with a time-homogeneous reference equation, and express the
difference of their laws through derivative estimates of the reference
semigroup and coefficient errors evaluated along the perturbed process.
They will be used to provide the essential link between the tractable frozen reference dynamics and the nonlinear McKean--Vlasov system. The estimates are formulated in an abstract test-function framework, and then specialized to two concrete topologies--weighted total variation/H\"older-dual and Wasserstein-1--that are tailored to the different regularity and structural assumptions encountered in the subsequent applications.

Consider the following two SDEs in $\mR^d$:
\begin{align}\label{eq1}
	\dif X_t
	=
	b_t(X_t)\dif t+\sigma_t(X_t)\dif W_t,
	\qquad
	X_0=\xi,
\end{align}
and
\begin{align}\label{eq2}
	\dif \tilde X_t
	=
	\tilde b(\tilde X_t)\dif t+\tilde\sigma(\tilde X_t)\dif W_t,
	\qquad
	\tilde X_0=\tilde \xi,
\end{align}
where $\xi$ and $\tilde\xi$ are two random variables, the coefficients
$$
b: \mR_+\times\mR^d\to\mR^d, \qquad \sigma: \mR_+\times\mR^d\to\mR^d\otimes\mR^d,
$$
and
$$
\tilde b: \mR^d\to\mR^d,\qquad \tilde\sigma: \mR^d\to\mR^d\otimes\mR^d
$$
are measurable functions. We set
$$
a_t(x):=\tfrac{1}{2}\sigma_t(x)\sigma_t(x)^*,
\qquad
\tilde a(x):=\tfrac{1}{2}\tilde\sigma(x)\tilde\sigma(x)^*.
$$
We call  (\ref{eq2}) the {\it reference equation} and (\ref{eq1}) the {\it perturbed equation}.
We shall  write $X_t(\xi)$ and $\tilde X_t(\tilde\xi)$ to stress the dependence on the initial values of the solutions,  when $\xi=x$ and $\tilde \xi=\tilde x\in\mR^d$, we  also write  $X_t(x)$ and $\tilde X_t(\tilde x)$.

We impose the following standing assumptions, which guarantee weak well-posedness and moment bounds for both equations.

\begin{enumerate}[$({\mathbf A_0})$]
	\item
	\begin{enumerate}[(a)]
		\item
		There exist $\alpha\in(0,1)$, $\tilde p\in(d,\infty]$ and
		$q,p\in(2,\infty]$ with
		$
		d/p+2/q<1
		$
		such that
		$$
		\tilde\sigma\in\mathbf C^\alpha_{\rm loc}(\mR^d),
		\qquad
		\tilde b\in L^{\tilde p}_{\rm loc}(\mR^d),
		$$
		and
		$$
		\sigma\in L^\infty_{\rm loc}
		\big(\mR_+;\mathbf C^\alpha_{\rm loc}(\mR^d)\big),
		\qquad
		b\in L^q_{\rm loc}\big(\mR_+;L^p_{\rm loc}(\mR^d)\big).
		$$
		Moreover, there exist locally bounded functions
		$
		0<\tilde\lambda(x)\leq \tilde\Lambda(x)<\infty
		$
		such that, for all $t>0$ and $x,\eta\in\mR^d$,
		$$
		\tilde\lambda(x)|\eta|^2
		\leq
		|\tilde\sigma(x)\eta|^2,\ |\sigma_t(x)\eta|^2
		\leq
		\tilde\Lambda(x)|\eta|^2.
		$$
		
		\item
		There exist weights
		$
		\rho_0,\rho_1:\mR^d\to[1,\infty)
		$
		and an increasing continuous function $L_0:\mR_+\to[1,\infty)$ such that
		\begin{equation}\label{a2}
			\mE\left[
			\rho_0(\tilde X_t(x))+\rho_0(X_t(x))
			\right]
			\leq
			L_0(t)\rho_1(x),
			\qquad
			t\geq0,\ x\in\mR^d.
		\end{equation}
	\end{enumerate}
\end{enumerate}

Under $(\mathbf A_0)$, the standard theory for SDEs with singular drift and locally non-degenerate diffusion ensures weak existence and non-explosion for both equations, see e.g. \cite{SV,XXZ26}.

\subsection{Abstract semigroup formulation stability estimate}

We first formulate an abstract stability estimate in terms of a test-function class.
Let $\mathfrak F\subseteq\cB_{\rho_0}$ be a class of test functions on $\mR^d$ equipped with a norm
$\|\cdot\|_{\mathfrak F}$.  Recall  the associated dual distance ${\bf d}_{\mathfrak F}(\mu,\nu)$   defined in (\ref{dF-new}).
The Markov semigroup of the reference
equation is denoted by
\begin{equation}\label{tt}
	\tilde{P}_t f(x)
	:=
	\mE f(\tilde X_t(x)),\quad\forall f\in \mathfrak F.
\end{equation}
We assume that the reference semigroup satisfies the following derivative
estimates.

\begin{enumerate}[$({\mathbf A}_{\mathfrak F})$]
	\item
	\begin{enumerate}[(a)]
		\item There exist locally bounded functions
		$
		\Theta_{1,\mathfrak F},\Theta_{2,\mathfrak F}:\mR^d\to[1,\infty)
		$
		and locally integrable kernels
		$
		\sK_{1,\mathfrak F},\sK_{2,\mathfrak F}:\mR_+\to\mR_+
		$
		such that for every $f\in\mathfrak F$,
		\begin{equation}\label{abstract-grad}
			|\nabla_x \tilde{P}_t f(x)|
			\leq
			\Theta_{1,\mathfrak F}(x)
			\sK_{1,\mathfrak F}(t)
			\|f\|_{\mathfrak F},
			\qquad t>0,\ x\in\mR^d,
		\end{equation}
		and when $a_t\neq\tilde a$,
		\begin{equation}\label{abstract-hess}
			|\nabla_x^2 \tilde{P}_t f(x)|
			\leq
			\Theta_{2,\mathfrak F}(x)
			\sK_{2,\mathfrak F}(t)
			\|f\|_{\mathfrak F},
			\qquad t>0,\ x\in\mR^d.
		\end{equation}
		
		\item
		There exists a constant $C_{\mathfrak F}\geq1$ such that for any random variables
		$\xi,\tilde\xi$ whose distributions lie in the domain of
		${\bf d}_{\mathfrak F}$,
		\begin{equation}\label{abstract-initial}
			\big|
			\mE \tilde{P}_t f(\xi)
			-
			\mE \tilde{P}_t f(\tilde\xi)
			\big|
			\leq
			C_{\mathfrak F}
			\|f\|_{\mathfrak F}
			{\bf d}_{\mathfrak F}(\cL_\xi,\cL_{\tilde\xi}),
			\qquad t\geq0.
		\end{equation}
	\end{enumerate}
\end{enumerate}

The following theorem is the main result of this abstract framework. It quantifies how the difference between the distributions of the two processes is controlled by the initial distance, the drift perturbation, and--when present--the diffusion perturbation, all weighted by the derivative kernels of the reference semigroup.

\begin{theorem}[Stability estimate]\label{thm2-abstract}
	Assume $({\mathbf A_0})$ and $({\mathbf A}_{\mathfrak F})$.  Then, for every
	$f\in\mathfrak F$ and $t\geq0$,
	\begin{align}\label{abstract-stability}
		\begin{split}
			\Big|
			\mE f\big(X_t(\xi)\big)
			&-
			\mE f\big(\tilde X_t(\tilde\xi)\big)
			\Big|
			\leq
			\|f\|_{\mathfrak F}
			\bigg[
			C_{\mathfrak F}
			{\bf d}_{\mathfrak F}(\cL_\xi,\cL_{\tilde\xi})
			\\
			&
			+
			\mE\int_0^t
			\Theta_{1,\mathfrak F}\big(X_s(\xi)\big)
			\sK_{1,\mathfrak F}(t-s)
			\left|
			b_s\big(X_s(\xi)\big)
			-
			\tilde b\big(X_s(\xi)\big)
			\right|
			\dif s
			\\
			&
			+
			\mE\int_0^t
			\Theta_{2,\mathfrak F}\big(X_s(\xi)\big)
			\sK_{2,\mathfrak F}(t-s)
			\Big|
			a_s\big(X_s(\xi)\big)
			-
			\tilde a\big(X_s(\xi)\big)
			\Big|
			\dif s
			\bigg].
		\end{split}
	\end{align}
	If $a_t\equiv\tilde a$, then the last term vanishes.
\end{theorem}
\br
(i) Theorem \ref{thm2-abstract} shows  that the perturbation kernels are completely determined by
derivative estimates of the reference semigroup in the chosen
test-function class.  Hence different choices of $\mathfrak F$ may lead to
different  kernels, a flexibility that will be exploited in the uniqueness analysis below. In this sense, the topology of the test functions is not merely a technical device but a constitutive part of the quantitative stability theory.

(ii) The estimate (\ref{abstract-stability}) is asymmetric: all derivatives are taken on the reference
semigroup, while  the coefficient
differences are evaluated along the perturbed process
$X_t$. This form is particularly convenient when the reference equation represents a limiting  or frozen dynamics and the perturbed equations are approximations.
\er

\begin{remark}[On degenerate extensions]
	\label{deg}
	The local ellipticity imposed in this paper is used to invoke the regularity
	theory for the backward Kolmogorov equation and justify the generalized
	It\^o--Krylov formula. It is not intrinsic to the perturbation mechanism of
	Theorem~\ref{thm2-abstract}.
	
	Indeed, the proof extends to a possibly degenerate reference equation provided
	that for
	\[
	u(t,x)=\tilde{P}_{T-t}f(x),
	\]
	the following properties hold: \(u\) solves the backward Kolmogorov equation
	in a suitable generalized sense; an It\^o formula is valid for
	\(u(t,X_t)\) along the perturbed process; and the derivatives appearing in
	\[
	(\mathscr L_t-\tilde{\mathscr L})u
	=
	(b_t-\tilde b)\cdot\nabla  u
	+
	\operatorname{tr}\big((a_t-\tilde a)\cdot\nabla^2u\big)
	\]
	satisfy suitable weighted estimates with locally integrable time kernels. Note that, for genuinely degenerate equations, full Euclidean derivative estimates are
	not necessary. It is enough to control the directional derivatives in which
	the drift and diffusion perturbations act. In such settings, the
	derivative kernels    may exhibit
	stronger short-time singularities, whose integrability must be checked on a
	model-by-model basis.
	Once this is achieved, the proof gives the same convolution stability inequality, and  the anchored uniqueness and
	ergodicity arguments remain unchanged: the effective perturbation kernel is
	again obtained by combining the coefficient errors with the derivative
	estimates of the frozen reference evolution, and the decisive condition is
	the corresponding \(L^1\)-smallness condition.

	Consequently, the framework of the paper can be extended to hypoelliptic and kinetic
	McKean--Vlasov equations, degenerate Langevin systems, and other equations
	satisfying a H\"ormander-type regularization property. A systematic treatment of such degenerate models is
	left for future work.
\end{remark}

\begin{proof}[{\bf Proof of Theorem \ref{thm2-abstract}}]
	Let $\tilde \sL$ and $\sL_t$ be the generators of the reference process $\tilde X_t$ and the perturbed process $X_t$, respectively:
	\begin{equation*}\label{L}
		\tilde \sL\varphi(x):=\tr\Big(\tilde a(x)\nabla_x^2\varphi(x)\Big)+\tilde b(x)\cdot\nabla_x\varphi(x),
	\end{equation*}
	and
	\begin{align*}
		\sL_t\varphi(x) := \tr\Big(a_t(x)\nabla_x^2\varphi(x)\Big)+b_t(x)\cdot\nabla_x\varphi(x).
	\end{align*}
	Fix \(T>0\) and set
	\begin{equation}\label{uu}
		\tilde u(t,x):=\tilde{P}_{T-t} f(x)=\mE f(\tilde X_{T-t}(x))
	\end{equation}
	and
	$$
	u(t,x)=\mE f(X_{t,T}(x)),
	$$
	where $X_{t,T}(x)$ is the weak solution of SDE (\ref{eq1}) starting from $x$ at time $t$.
	By \cite[Theorem 4.3]{XXZ26}, $\tilde u$ and $u$ are the unique solutions in
	$W^{1,2}_{q,p;{\rm loc}}(\mR_+\times\mR^d)$ of
	the following backward Kolmogorov equations in $[0,T]\times\mR^d$:
	\begin{equation*}\left\{\begin{array}{l}
			\displaystyle
			\p_t u(t,x)+\sL_t u(t,x)=0,\quad t\in [0, T),\\
			u(T,x)=f(x),
		\end{array}\right.
	\end{equation*}
	and
	\begin{equation*}\left\{\begin{array}{l}\label{PDE0}
			\displaystyle
			\p_t\tilde u(t,x)+\tilde\sL \tilde u(t,x)=0,\quad t\in [0, T),\\
			\tilde u(T,x)=f(x),
		\end{array}\right.
	\end{equation*}
	respectively.
	Then
	$$
	v(t,x):=u(t,x)-\tilde u(t,x)
	$$
	satisfies
	\begin{equation*}\left\{\begin{array}{l}
			\displaystyle
			\p_t v(t,x)+\sL_t v(t,x)=-\big(b_t(x)-\tilde b(x)\big)\cdot\nabla_x\tilde u(t,x)\\
			\qquad\qquad\qquad\qquad\quad\quad-\tr\Big(\big(a_t(x)-\tilde a(x)\big)\cdot\nabla^2_x\tilde u(t,x)\Big),\quad t\in [0, T),\\
			v(T,x)=0.
		\end{array}\right.
	\end{equation*}
	Let
	\begin{align}\label{ff}
		\hat f(t,x):=\big(b_t(x)-\tilde b(x)\big)\cdot\nabla_x\tilde u(t,x)+\tr\Big(\big(a_t(x)-\tilde a(x)\big)\cdot\nabla^2_x\tilde u(t,x)\Big).
	\end{align}
If \(a_t\equiv\tilde a\), the second term in (\ref{ff}) can be dropped.
	Using \cite[Theorem 4.3]{XXZ26} again,  we have
	\begin{align*}
		v(t,x)=\mE\left(\int_t^T \hat f(s,X_{t,s}(x))\dif s\right).
	\end{align*}
	Taking $t=0$ gives
	\begin{align*}
		\mE f\big(X_T(\xi)\big)-\mE f\big(\tilde X_T(\xi)\big)=v(0,\xi)=\mE\left(\int_0^T \hat f(s,X_s(\xi))\dif s\right).
	\end{align*}
	Using \eqref{abstract-grad} and \eqref{abstract-hess} with $T-s$ in place of
	$t$, we obtain
	\begin{align*}
		|v(0,\xi)|
		&\leq
		\|f\|_{\mathfrak F}
		\mE\int_0^T
		\Theta_{1,\mathfrak F}(X_s(\xi))
		\sK_{1,\mathfrak F}(T-s)
		\left|
		b_s(X_s(\xi))-\tilde b(X_s(\xi))
		\right|
		\dif s
		\\
		&\quad
		+
		\|f\|_{\mathfrak F}
		\mE\int_0^T
		\Theta_{2,\mathfrak F}(X_s(\xi))
		\sK_{2,\mathfrak F}(T-s)
		\Big|
		a_s(X_s(\xi))-\tilde a(X_s(\xi))
		\Big|
		\dif s .
	\end{align*}
	The difference caused by the initial distributions is controlled by
	\eqref{abstract-initial}.  Combining these estimates and replacing $T$ by $t$
	proves \eqref{abstract-stability}.
\end{proof}

\subsection{Weighted total variation and H\"older-dual realization}

We now recover the weighted total variation version of Theorem \ref{thm2-abstract}. This specialization yields explicit kernels and is particularly well suited to SDEs with singular coefficients.

Instead of  $({\mathbf A_\mathfrak F})$, we  make the following more straightforward ergodic assumption on the reference equation.

\begin{enumerate}[$({\mathbf A_1})$]
	\item
	The reference semigroup (\ref{tt}) has a unique invariant measure
	$\tilde\mu$ with
	$
	\int_{\mR^d}\rho_0(x)\tilde\mu(\dif x)<\infty,
	$
	and there exists a decreasing function
	$\ell_0:[1,\infty)\to\mR_+$ with $\ell_0(t)\to0$ as $t\to\infty$ such that,
	for all $f\in\mathcal B_{\rho_0}$,
	\begin{align}\label{a3}
		|\tilde{P}_t f(x)-\tilde\mu(f)|
		\leq
		\ell_0(t)\rho_1(x)\|f\|_{\mathcal B_{\rho_0}},
		\qquad t\geq1,\ x\in\mR^d.
	\end{align}
\end{enumerate}

For simplicity, define the kernels
\begin{align}\label{eqH1}
	\sK_1(t)
	:=
	\mathbf 1_{(0,2]}(t)t^{-1/2}
	+
	\mathbf 1_{(2,\infty)}(t)\ell_0(t-1),
\end{align}
and, for $\a\in(0,1)$,
\begin{align}\label{eqH2}
	\sK_{2,\alpha}(t)
	:=
	\mathbf 1_{(0,2]}(t)t^{-1+\a/2}
	+
	\mathbf 1_{(2,\infty)}(t)\ell_0(t-1),
\end{align}
where $\ell_0(t)$ is given in (\ref{a3}). These kernels encode the short-time parabolic smoothing and the long-time ergodic decay of the reference semigroup.

The next lemma, taken from \cite{XXZ26}, provides the required derivative estimates of the reference
semigroup in the weighted total variation and weighted H\"older classes.

\begin{lemma}\label{esu}
	Assume $(\mathbf A_0)$ and $(\mathbf A_1)$ hold.  Then, for every
	$f\in\mathcal B_{\rho_0}$,
	there exists a locally bounded function
	$\Theta_1:\mR^d\to[1,\infty)$ such that
	\begin{align}\label{gr1}
		|\nabla_x \tilde{P}_t f(x)|
		\leq
		\Theta_1(x)\sK_1(t)\|f\|_{\mathcal B_{\rho_0}},
		\qquad t>0,\ x\in\mR^d.
	\end{align}
	If, in addition, $\tilde b\in\mathbf C^\alpha_{\rm loc}(\mR^d)$ with
	$\alpha\in(0,1)$, then there exists a locally bounded function
	$\Theta_{2,\alpha}:\mR^d\to[1,\infty)$ such that for all
	$f\in\mathbf C^\alpha_{\rho_0}$,
	\begin{align}\label{gr2}
		|\nabla_x^2 \tilde{P}_t f(x)|
		\leq
		\Theta_{2,\alpha}(x)
		\sK_{2,\alpha}(t)
		\|f\|_{\mathbf C^\alpha_{\rho_0}},
		\qquad t>0,\ x\in\mR^d.
	\end{align}
\end{lemma}

\begin{proof}
	The short-time estimates in (\ref{gr1}) and (\ref{gr2}) with $t\leq 2$ follow from \cite[Theorem 1.2]{XXZ26}, while the long-time estimates follow  from \cite[Theorem 1.5]{XXZ26}.
	Explicit formulas for the weight functions can be given as
	\begin{equation*}\label{Theta1}
		\Theta_1(x)
		=
		C_0
		\tilde\Lambda_1^\eps(x)
		\left(
		\frac{\tilde\Lambda_1(x)}{\tilde\lambda(x)}
		\right)^{3\eps+d/\tilde p}
		\tilde\Gamma_1(x)
		\|\rho_1\|_{L^\infty(B_1(x))},
	\end{equation*}
	with $C_0>0$, $\eps>0$, $\tilde\Lambda_1(x):=\tilde\Lambda(x)+1,$ and
	\begin{align*}
		\tilde\Gamma_1(x)
		&:=
		\frac{\tilde\Lambda_1(x)
			[\tilde a]_{\mathbf C^\alpha(B_1(x))}^{1/\alpha}}
		{\tilde\lambda^{1+1/\alpha}(x)}
		+
		\frac{\|\tilde b\|_{L^{\tilde p}(B_1(x))}^{1/\theta_{\tilde b}}}
		{\tilde\lambda(x)^{1/\theta_{\tilde b}}}
		+
		\frac{\tilde\Lambda_1(x)}{\tilde\lambda(x)},
		\qquad
		\theta_{\tilde b}:=1-\frac{d}{\tilde p},
	\end{align*}
	and
	\begin{equation*}\label{Theta2}
		\Theta_{2,\a}(x)
		=
		C_0\tilde\Gamma_{2,\alpha}(x)
		\|\rho_1\|_{L^\infty(B_1(x))},
	\end{equation*}
	with
	\begin{align*}
		\tilde{ \Gamma}_{2,\a}(x)&:=\left[\bigg(\frac{ \tilde\Lambda_{1}(x) [\tilde a]^{1/\alpha}_{\mathbf C^\alpha(B_1(x))}}{\tilde\lambda^{1+1/\alpha}(x)}+\frac{\|\tilde b\|_{\mathbf C^\alpha(B_{1}(x))}}{\tilde\lambda(x)}\bigg) +\frac{\tilde\Lambda_{1}(x)}{\tilde\lambda(x)}\right]^2.
	\end{align*}
	The local boundedness of $\Theta_1(x)$ and $\Theta_{2,\a}(x)$ follows directly from these expressions.
\end{proof}

The weighted total variation version stability estimate is now an immediate consequence of the
abstract Theorem \ref{thm2-abstract}.

\begin{theorem}[Weighted total variation/H\"older-dual  stability]\label{thm2}
	Assume $(\mathbf A_0)$ and $(\mathbf A_1)$, and suppose
	\begin{align}\label{xi}
		\mE\left[\rho_1(\xi)+\rho_1(\tilde\xi)\right]<\infty.
	\end{align}
	Then the following assertions hold.
	
	\begin{enumerate}[(i)]
		\item (Drift-only perturbation) If $a_t\equiv\tilde a$, then for every
		$f\in\mathcal B_{\rho_0}$ and $t>0$,
		\begin{equation*}\label{es1}
			\begin{split}
				&
				\left|
				\mE f\big(X_t(\xi)\big)
				-
				\mE f\big(\tilde X_t(\tilde\xi)\big)
				\right|
				\leq
				\|f\|_{\mathcal B_{\rho_0}}
				\bigg[
				C_0{\bf d}_{\rho_1}(\cL_\xi,\cL_{\tilde\xi})
				\\
				&\qquad\qquad\quad
				+
				\mE\int_0^t
				\Theta_1\big(X_s(\xi)\big)\sK_1(t-s)
				\left|
				b_s\big(X_s(\xi)\big)
				-
				\tilde b\big(X_s(\xi)\big)
				\right|
				\dif s
				\bigg],
			\end{split}
		\end{equation*}
		where $\sK_1(t)$ is defined by (\ref{eqH1}), and $C_0>0$ is independent of $t$, $f$, $\xi$ and $\tilde\xi$.

		\item (Full drift-diffusion perturbation) If, in addition, $\tilde b\in\mathbf C^\alpha_{\rm loc}(\mR^d)$, then
		for every $f\in\mathbf C^\alpha_{\rho_0}$ and $t>0$,
		\begin{equation*}\label{es2}
			\begin{split}
				&
				\left|
				\mE f\big(X_t(\xi)\big)
				-
				\mE f\big(\tilde X_t(\tilde\xi)\big)
				\right|
				\leq
				\|f\|_{\mathcal B_{\rho_0}}\bigg[C_0
				{\bf d}_{\rho_1}(\cL_\xi,\cL_{\tilde\xi})
				\\
				&\qquad\qquad\quad
				+
				\mE\int_0^t
				\Theta_1\big(X_s(\xi)\big)\sK_1(t-s)
				\left|
				b_s\big(X_s(\xi)\big)
				-
				\tilde b\big(X_s(\xi)\big)
				\right|
				\dif s\bigg]
				\\
				&\qquad\qquad\quad
				+
				\|f\|_{\mathbf C^\alpha_{\rho_0}}
				\mE\int_0^t
				\Theta_{2,\alpha}\big(X_s(\xi)\big)\sK_{2,\alpha}(t-s)
				\left|
				a_s\big(X_s(\xi)\big)
				-
				\tilde a\big(X_s(\xi)\big)
				\right|
				\dif s,
			\end{split}
		\end{equation*}
		where $\sK_{2,\a}(t)$ is defined by (\ref{eqH2}).
	\end{enumerate}
\end{theorem}

\begin{proof}
	The first estimate follows from Theorem \ref{thm2-abstract} with
	$$
	\mathfrak F=\mathcal B_{\rho_0},
	\qquad
	{\bf d}_{\mathfrak F}={\bf d}_{\rho_1},
	\qquad
	\sK_{1,\mathfrak F}=\sK_1,
	\qquad
	\Theta_{1,\mathfrak F}=\Theta_1.
	$$
	The second estimate is obtained  by taking
	$$
	\mathfrak F=\mathbf C^\alpha_{\rho_0},
	\qquad
	\sK_{2,\mathfrak F}=\sK_{2,\alpha},
	\qquad
	\Theta_{2,\mathfrak F}=\Theta_{2,\alpha}.
	$$
	The initial-distribution term is controlled by \eqref{a2} and \eqref{a3}.
\end{proof}

\begin{remark}
	Estimate (ii) contains
	an additional diffusion perturbation term, which requires a Hessian estimate
	of the reference semigroup.  The singularity
	$t^{-1+\alpha/2}$ in $\sK_{2,\alpha}$  is integrable near zero precisely because
	$\alpha>0$. This is the reason why the full drift-diffusion
	perturbation estimate is stated for H\"older test functions. If
	$f$ is only locally bounded, then the short-time Hessian singularity would be
	of order $t^{-1}$, which is not integrable at $0$. This  will be important when we discuss models where the diffusion coefficient depends on the distribution.
\end{remark}

It is often useful to combine Theorem \ref{thm2}
with a local Krylov estimate for $X_t$, particularly when the coefficients converge only in local Lebesgue norms.
For $R>0$, set
$$
\tau_R:=\inf\{t\geq0:\ |X_t|\geq R\}.
$$
We recall the following local Krylov type estimate, see e.g. \cite{Zhang11}.

\bl
Assume $(\mathbf{A_{0}})$ holds. Let $T,R>0$ and  $X_t$ be the unique weak solution to  SDE (\ref{eq1}). Then there exists a constant $C_{T,R}$ such that for every $f\in \mL^q_p((0,T)\times B_R)$ with $d/p+2/q<2$, we have
\begin{align*}
	\mE\left(\int_0^{T\wedge\tau_R}|f(s,X_s)|\dif s\right)\leq C_{T,R}\|f\|_{\mL^q_p((0,T)\times B_R)}.
\end{align*}
\el

For a locally bounded function $\Theta$, define
$$
\Theta_R:=\|\Theta\|_{L^\infty(B_R)}.
$$
The following localized stability result separates the local coefficient convergence from the tail contribution.

\begin{corollary}[Local $L^q_p$ stability]\label{Lp}
	Assume $(\mathbf{A_0})$, $(\mathbf{A_1})$ and  (\ref{xi}). Fix $T, R>0$. The following assertions hold.
	
	\begin{enumerate}[(i)]
		\item (Drift-only perturbation)
		If $a_t\equiv \tilde a$, then for every
		$f\in\mathcal B_{\rho_0}$,
		\begin{equation*}\label{qp1}
			\begin{split}
				&\left|
				\mathbb E f\big(X_T(\xi)\big)-\mathbb E f\big(\tilde X_T(\tilde\xi)\big)
				\right|
				\leq
				C_0\|f\|_{\mathcal B_{\rho_0}}
				{\bf d}_{\rho_1}(\cL_{\xi},\cL_{\tilde\xi})
				\\
				&\qquad\qquad\quad\,\,+
				C_{T,R}\,
				\Theta_{1,R}
				\|b-\tilde b\|_{\mL^q_p((0,T)\times B_R)}
				\|f\|_{\mathcal B_{\rho_0}}+
				\mathcal R_R(T,f)
			\end{split}
		\end{equation*}
		where
		\begin{equation*}\label{rr}
			\mathcal R_R(T,f)
			:=
			2\|f\|_{\mathcal B_{\rho_0}}
			\mathbb E\big[
			\rho_1(X_T(\xi))\cdot\mathbf 1_{\{T>\tau_R\}}
			\big].
		\end{equation*}
		
		\item (Full drift-diffusion perturbation) If, in addition $\tilde b\in \mathbf C^\a_{\rm loc}$, then   for every
		$f\in\mathbf C^\a_{\rho_0}$,
		\begin{align}\label{qp2}
			\begin{split}
				&\left|
				\mathbb E f\big(X_T(\xi)\big)-\mathbb E f\big(\tilde X_T(\tilde\xi)\big)
				\right|
				\leq
				C_0\|f\|_{\mathcal B_{\rho_0}}
				{\bf d}_{\rho_1}(\cL_{\xi},\cL_{\tilde\xi})
				\\
				&\qquad\qquad\quad\,\,+
				C_{T,R}\,
				\Theta_{1,R}
				\|b-\tilde b\|_{\mL^q_p((0,T)\times B_R)}
				\|f\|_{\mathcal B_{\rho_0}}
				\\
				&\qquad\qquad\quad\,\,+
				C_{T,R}\,
				\Theta_{2,\a,R}
				\|a-\tilde a\|_{\mL^\infty((0,T)\times B_R)}
				\|f\|_{\mathbf C^\a_{\rho_0}}
				+
				\mathcal R_R(T,f).
			\end{split}
		\end{align}
	\end{enumerate}
\end{corollary}

\begin{proof}
	We  only prove the drift-diffusion estimate (\ref{qp2}), the drift-only case is the same with the diffusion term removed.
	Write
	\begin{align*}
		\Big|\mathbb E f\big(X_T(\xi)\big)-\mathbb E f\big(\tilde X_T(\tilde\xi)\big)
		\Big|&\leq \Big|\mathbb E f\big(\tilde X_T(\xi)\big)-\mathbb E f\big(\tilde X_T(\tilde\xi)\big)
		\Big|\\
		&+\Big|\mathbb E\,\tilde u\big(T\wedge\tau_R,X_{T\wedge\tau_R}(\xi)\big)-\mE\,\tilde u(0,\xi)\Big|\\
		&+\Big|\mathbb E\,\tilde u\big(T,X_T(\xi)\big)-\mathbb E\,\tilde u\big(T\wedge\tau_R,X_{T\wedge\tau_R}(\xi)\big)\Big|=:I_1+I_2+I_3,
	\end{align*}
	where  $\tilde u$ is defined by (\ref{uu}).
	For the  initial-distribution term, as in the proof of
	Theorem \ref{thm2}, we have
	$$
	I_1\leq C_0\|f\|_{\mathcal B_{\rho_0}}
	{\bf d}_{\rho_1}(\cL_{\xi},\cL_{\tilde\xi}).
	$$
	To control the second term, by the generalized It\^o--Krylov formula applied up to the stopping time
	$T\wedge\tau_R$, we have
	\begin{align*}
		I_2
		=
		\mathbb E\int_0^{T\wedge\tau_R}
		\hat f(s,X_s(\xi))\dif s,
	\end{align*}
	where $\hat f$ is defined by (\ref{ff}).
	On $\{s<\tau_R\}$, one has $X_s\in B_R$. Hence
	$$ \Theta_1(X_s)\leq\Theta_{1,R},\qquad\Theta_{2,\a}(X_s)\leq\Theta_{2,\a,R}.
	$$
	Since
	$d/p+2/q<1$, choose $r\in(1,2)$ sufficiently close
	to $2$ such that, with
	$$
	\frac1{\bar q}:=\frac1q+\frac1r,
	$$
	we have
	$$
	\frac d p+\frac2{\bar q}<2.
	$$
	By the local Krylov estimate applied to
	$$
	f(s,x)=
	\sK_1(T-s)|b_s(x)-\tilde b(x)|\mathbf 1_{B_R}(x),
	$$
	and H\"older's inequality, we get
	\begin{align*}
		\mathbb E\int_0^{T\wedge\tau_R}
		\sK_1(T-s)|b_s(X_s)-\tilde b(X_s)|\dif s
		&\leq
		C_{T,R}
		\|\sK_1(T-\cdot)(b-\tilde b)\|_{\mathbb L^{\bar q}_p((0,T)\times B_R)}
		\\
		&\leq
		C_{T,R}
		\|\sK_1(T-\cdot)\|_{L^r(0,T)}
		\|b-\tilde b\|_{\mathbb L^q_p((0,T)\times B_R)}
		\\
		&\leq
		C_{T,R}
		\|b-\tilde b\|_{\mathbb L^q_p((0,T)\times B_R)}.
	\end{align*}
	For the diffusion term, using the Hessian estimate \eqref{gr2}, we directly have
	\begin{align*}
		&\mathbb E\int_0^{T\wedge\tau_R}
		\sK_{2,\a}(T-s)|a_s(X_s)-\tilde a(X_s)|\dif s\\
		&\leq
		\Theta_{2,\alpha,R}
		\|f\|_{\mathbf C^\alpha_{\rho_0}}
		\|a-\tilde a\|_{\mathbb L^\infty((0,T)\times B_R)}
		\int_0^T \sK_{2,\alpha}(T-s)\dif s.
	\end{align*}
	Since $\alpha>0$, the singularity of $\sK_{2,\alpha}(t)$ at $0$ is integrable.
	It remains to estimate the stopping remainder term $I_3$.
	On the event $\{\tau_R\geq T\}$, one has $T\wedge\tau_R=T$, and therefore $I_3$ is supported on $\{\tau_R< T\}$. Since
	$$
	\tilde u(t,x)\leq C_0\rho_1(x),
	$$
	we have
	$$
	I_3\leq 2\|f\|_{\mathcal B_{\rho_0}}
	\mathbb E\big[
	\rho_1(X_T(\xi))\cdot\mathbf 1_{\{T>\tau_R\}}
	\big].
	$$
	Combining the estimates completes the proof.
\end{proof}

Corollary \ref{Lp} is useful for proving  existence of invariant measures (see Section 4), where one typically takes limits of approximating coefficients and needs to control the error locally. However, for uniqueness and ergodicity arguments,   a direct application of the
Krylov estimate  would destroy the
 convolution structure
\[
\int_0^t \sK(t-s)\sD(s)\dif s,
\]
which is essential for the Gronwall-type inequality in Lemma \ref{gron}. We provide the following $L^q$-$L^p$ pair version, which preserves exactly the convolution kernel and  will be the key tool in Section 6 for deriving the explicit $L^p$-uniqueness threshold for granular media equations.

\begin{proposition}[Total variation stability under $L^p$ drift perturbations]
	\label{Lp-drift-TV-stability}
	Assume that $a_t\equiv\tilde a$. For $p\in(d,\infty]$, let
$
q:=p/(p-1),
$
with the convention $q=1$ when $p=\infty$.
	Suppose that  for every bounded   function $f$,
	\begin{align}\label{Lp-reference-gradient}
		\|\nabla_x \tilde P_t f\|_\infty
		\leq
		\sK(t)\|f\|_\infty,
		\qquad t>0,
	\end{align}
	where $\sK$ is locally integrable.
	Fix $0\leq t_0<t$. Assume that
	$X_s$ has a density $\rho_s$ for $s\in[t_0,t]$ and that
	\begin{align}\label{Lp-density}
		\sup_{s\in[t_0,t]}
		\|\rho_s\|_{L^q}
		\leq
		C_{p,t_0,t}<\infty.
	\end{align}
	Let $\tilde X_r(X_{t_0})$  denote a solution of the reference
	equation \eqref{eq2} starting from the random initial value
	$X_{t_0}$, then for every bounded measurable function $f$,
	\begin{align}\label{Lp-drift-stability}
		&
		\left|
		\mathbb E f(X_t)
		-
		\mathbb E f\big(\tilde X_{t-t_0}(X_{t_0})\big)
		\right|
		\leq
		C_{p,t_0,t}\|f\|_\infty
		\int_{t_0}^t
		\sK(t-s)
		\|b_s-\tilde b\|_{L^p}\dif s.
	\end{align}

	If, for some $t_0\geq0$,
	\begin{align}\label{den}
		\sup_{s\geq t_0}\|\rho_s\|_{L^q}
		\leq C_{p,t_0}<\infty,
	\end{align}
	then $C_{p,t_0,t}$ in the preceding estimates can be replaced by
	$C_{p,t_0}$, uniformly for all $t\geq t_0$.
	
	In particular, if $p=\infty$, then $q=1$,   both (\ref{Lp-density}) and (\ref{den}) hold autonomously.
\end{proposition}
\begin{proof}
	For a bounded measurable function $f$, similar arguments as before and using It\^o--Krylov formula  on $[t_0,t]$ gives
	\begin{align}\label{Lp-Duhamel-expectation}
		\mathbb E f(X_t)
		-
		\mathbb E
		\left[
		\tilde P_{t-t_0} f(X_{t_0})
		\right]
		=
		\int_{t_0}^t
		\mathbb E
		\left[
		\big(
		b_s(X_s)-\tilde b(X_s)
		\big)
		\cdot
		\nabla_x \tilde P_{t-s} f(X_s)
		\right]\dif s.
	\end{align}
	By H\"older's inequality and \eqref{Lp-reference-gradient}, for \(s\in[t_0,t]\) we have
	\begin{align*}
		&
		\left|
		\mathbb E
		\left[
		\big(
		b_s(X_s)-\tilde b(X_s)
		\big)
		\cdot
		\nabla_x \tilde P_{t-s} f(X_s)
		\right]
		\right|
		\\
		&\qquad\leq
		\|\rho_s\|_{L^q}
		\|b_s-\tilde b\|_{L^p}
		\|\nabla_x \tilde P_{t-s} f\|_\infty
		\leq
		C_{p,t_0,t}
		\sK(t-s)
		\|b_s-\tilde b\|_{L^p}
		\|f\|_\infty.
	\end{align*}
	Integrating this estimate over $s\in[t_0,t]$ in
	\eqref{Lp-Duhamel-expectation} proves
	\eqref{Lp-drift-stability}.
\end{proof}

\subsection{Wasserstein-1 realization}

While the weighted total variation framework provides robust estimates under minimal regularity, the derivative kernels $\sK_1$ and
$\sK_{2,\alpha}$
display short-time singularities that can lead to conservative smallness conditions.
In concrete models with more regular coefficients,  the same reference semigroup
may satisfy sharper estimates in a smaller test-function class.  This leads to
sharper perturbation kernels, which in turn yield sharper smallness conditions in the uniqueness analysis below.

We now record the Lipschitz--Wasserstein version  of the stability estimate, which is particularly useful for mean-field examples.  This version is not meant to replace the
weighted total variation estimate above; rather, it shows that sharper kernels
are available when the reference semigroup contracts Lipschitz functions.

\begin{corollary}[$W_1$ stability]
	Assume $(\mathbf A_0)$.  Suppose that the reference semigroup satisfies
	\begin{equation}\label{Lip-semigroup}
		|\nabla_x\tilde{P}_t f(x)|
		\leq
		\ell_1(t)\operatorname{Lip}(f),
		\qquad t\geq0,
	\end{equation}
	for every Lipschitz continuous function $f$, where $\ell_1$ is locally integrable.
	
	Then, in the drift-only case $a_t\equiv\tilde a$, we have
	\begin{equation}\label{W1-stability-drift}
		\begin{split}
			&
			\left|
			\mE f\big(X_t(\xi)\big)
			-
			\mE f\big(\tilde X_t(\tilde\xi)\big)
			\right|
			\\
			&\leq
			\operatorname{Lip}(f)
			\bigg[
			\ell_1(t)W_1(\cL_\xi,\cL_{\tilde\xi})
			+
			\mE\int_0^t
			\ell_1(t-s)
			\left|
			b_s\big(X_s(\xi)\big)
			-
			\tilde b\big(X_s(\xi)\big)
			\right|
			\dif s
			\bigg].
		\end{split}
	\end{equation}
	
	If diffusion perturbations are also present and, for the chosen test class,
	\begin{equation}\label{Lip-hessian}
		|\nabla_x^2\tilde{P}_t f(x)|
		\leq
		\Theta_{2,{\rm Lip}}(x)\ell_2(t)\operatorname{Lip}(f),
	\end{equation}
	then the additional term
	\begin{equation*}\label{W1-diff-term}
		\operatorname{Lip}(f)\,
		\mE\int_0^t
		\Theta_{2,{\rm Lip}}\big(X_s(\xi)\big)\ell_2(t-s)
		\left|
		a_s\big(X_s(\xi)\big)
		-
		\tilde a\big(X_s(\xi)\big)
		\right|
		\dif s
	\end{equation*}
	should be added to the right-hand side of \eqref{W1-stability-drift}.
\end{corollary}

\begin{proof}
	By Kantorovich--Rubinstein duality and \eqref{Lip-semigroup},
	$$
	\left|
	\mE \tilde{P}_t f(\xi)
	-
	\mE \tilde{P}_t f(\tilde\xi)
	\right|
	\leq
	\operatorname{Lip}(\tilde{P}_t f)
	W_1(\cL_\xi,\cL_{\tilde\xi})
	\leq
	\ell_1(t)\operatorname{Lip}(f)
	W_1(\cL_\xi,\cL_{\tilde\xi}).
	$$
	The conclusion follows from Theorem \ref{thm2-abstract} with
	$$
	\sK_{1,\mathfrak F}=\ell_1,
	\qquad
	\sK_{2,\mathfrak F}=\ell_2.
	$$
	The diffusion
	perturbation term follows from the additional Hessian estimate
	\eqref{Lip-hessian}.
\end{proof}

\br\label{rem:W1}
The
$W_1$ realization is particularly useful when the reference equation is a contraction in the Lipschitz sense.  For instance, consider
$$
\dif X_t= b(X_t)\dif t+\sigma \dif W_t,
$$
where $\sigma$ is a constant non-degenerate matrix and  $b\in C^1(\mathbb{R}^d)$ satisfies
$$
-\<\nabla_x b(x)v,v\>\geq \lambda|v|^2,\quad\forall x,v\in\mathbb{R}^d,
$$
for some $\lambda>0$. In such cases, synchronous coupling gives
\[
|\nabla_x\tilde{P}_t f(x)|
\leq
\e^{-\lambda t}\operatorname{Lip}(f).
\]
Thus \eqref{Lip-semigroup} holds with
$\ell_1(t)=\e^{-\lambda t}$, which  decays exponentially without a short-time singularity. This absence of short-time singularity is precisely
what leads to sharper \(W_1\)-thresholds in models where the law dependence
acts through low-order moments.
\er
\section{Existence of invariant measures}

In this section, we focus on proving the existence of invariant measures.
Throughout this section, $M_0>0$ and $V$ are as in assumption $(\mathbf H_1)$.

\subsection{Properties of the frozen invariant measure map}
For every
$\mu\in \sP_V^{M_0}(\mathbb{R}^d)$, consider the frozen autonomous SDE (\ref{000}), i.e.,
\begin{equation}\label{fr3}
	\dif X_t^{\mu}
	=
	b(X_t^{\mu},\mu)\dif t
	+
	\sigma(X_t^{\mu},\mu)\dif W_t,
	\qquad
	X_0^{\mu}=x.
\end{equation}
Its Markov semigroup is denoted by
\begin{equation*}
	P_t^\mu f(x)
	:=
	\mathbb E f(X_t^{\mu}(x)),
	\qquad t\geq0.
\end{equation*}
Under $(\mathbf H_0)$, the local non-degeneracy implies the strong Feller property and irreducibility of the frozen semigroup (see \cite{XZ}). Moreover, the Lyapunov condition $(\mathbf H_1)$ ensures that the frozen equation  is non-explosive and admits a unique invariant measure (see \cite[Theorem 4.2]{MT93}). We denote this invariant measure by $\nu_\mu$ and define the frozen invariant measure map
\begin{equation*}
	\sT(\mu):=\nu_\mu,
	\qquad
	\mu\in \sP_V^{M_0}(\mathbb{R}^d).
\end{equation*}

We begin by establishing the weak compactness of the Lyapunov moment ball.

\begin{lemma}
\label{weak-compactness}
Assume $(\mathbf H_1)$. Then $\sP_V^{M_0}(\mathbb{R}^d)$ is a compact convex
subset of $\mathscr P(\mathbb R^d)$ endowed with the weak topology.
\end{lemma}

\begin{proof}
The convexity of $\sP_V^{M_0}(\mathbb{R}^d)$ is  immediate.
For every $L>0$ and every $\mu\in \sP_V^{M_0}(\mathbb{R}^d)$,  Chebyshev's inequality gives
\[
\mu(\{V>L\})
\leq
\frac{\mu(V)}{L}
\leq
\frac{M_0}{L}.
\]
Since $V$ has compact level sets, $\sP_V^{M_0}(\mathbb{R}^d)$ is uniformly tight.
Now suppose
\[
\mu_n\in \sP_V^{M_0}(\mathbb{R}^d),
\qquad
\mu_n\Rightarrow\mu.
\]
Since $V$ is nonnegative and lower semicontinuous, the Portmanteau
theorem yields
\[
\mu(V)
\leq
\liminf_{n\to\infty}\mu_n(V)
\leq M_0.
\]
Thus the moment ball is  tight and weakly closed. By Prokhorov's theorem, it is weakly compact.
\end{proof}

Next, we establish the crucial invariance of the Lyapunov moment ball under the  frozen invariant-measure map $\sT$.
\begin{lemma}
\label{tightness}
Assume $(\mathbf H_0)$ and $(\mathbf H_1)$. Then the map $\sT$ is well
defined on $\sP_V^{M_0}(\mathbb{R}^d)$ and
\begin{align}\label{T-preserves-moment-ball}
\sT(\sP_V^{M_0}(\mathbb{R}^d))\subset \sP_V^{M_0}(\mathbb{R}^d).
\end{align}
\end{lemma}

\begin{proof}
Fix $\mu\in \sP_V^{M_0}(\mathbb{R}^d)$ and let
\[
\nu:=\sT(\mu).
\]
Since $\nu$  is the invariant measure for \eqref{fr3},  a standard
localization argument  (see, e.g., \cite[Theorem 4.2]{MT93})  yields
$$
\int_{\mR^d}\sL_{\mu}U(x)\nu(\dif x)=0.
$$
Using the Lyapunov inequality \eqref{HL-general}  and integrating with respect to
$\nu$, we obtain
\[
\kappa_0\nu(V)
\leq
\sum_{i=1}^N
c_i\nu(V^{\alpha_i})\mu(V)^{\beta_i}
+c_0.
\]
Note that $r\mapsto r^{\alpha_i}$ is concave on $[0,\infty)$ with $\alpha_i\in[0,1)$,
Jensen's inequality gives
\[
\nu(V^{\alpha_i})
\leq
\nu(V)^{\alpha_i}.
\]
Set $r:=\nu(V).$
Since $\mu(V)\leq M_0$, we have
\begin{align}\label{scalar-T-estimate}
\kappa_0r
\leq
\sum_{i=1}^N
c_ir^{\alpha_i}M_0^{\beta_i}
+c_0.
\end{align}
Suppose, for contradiction, that $r>M_0$. Dividing \eqref{scalar-T-estimate} by $r$ gives
\[
\kappa_0
\leq
\sum_{i=1}^N
c_ir^{\alpha_i-1}M_0^{\beta_i}
+c_0r^{-1}.
\]
Since $\alpha_i<1$ and $r>M_0$, we have
\[
r^{\alpha_i-1}<M_0^{\alpha_i-1},
\qquad
r^{-1}<M_0^{-1}.
\]
Consequently,
\[
\kappa_0
<
\sum_{i=1}^N
c_iM_0^{\alpha_i+\beta_i-1}
+c_0M_0^{-1}
\leq
\kappa_0,
\]
where the last inequality follows from
\eqref{trapping-general}. This contradiction proves
\[
\sT(\mu)(V)\leq M_0.
\]
Hence \eqref{T-preserves-moment-ball} holds.
\end{proof}

To prove the continuity of the invariant measure map $\sT$, we first establish a uniform compact containment property for the frozen processes.

\begin{lemma}
	\label{uniform-compact-containment}
	Assume $(\mathbf H_1)$. Then, for every $t>0$ and $L>0$,
	\begin{align*}
		\lim_{R\to\infty}
		\sup_{\mu\in \sP_V^{M_0}(\mathbb{R}^d)}
		\sup_{x\in B_L}
		\mathbb P
		\left(
		\sup_{0\leq s\leq t}|X_s^{\mu}(x)|\geq R
		\right)
		=0.
	\end{align*}
	In particular, all frozen equations with parameter in $\sP_V^{M_0}(\mathbb{R}^d)$ are
	non-explosive.
\end{lemma}

\begin{proof}
	Since $\alpha_i<1$, for every $\varepsilon>0$ there exists
	$C_\varepsilon>0$ such that
	\[
	\sum_{i=1}^N
	c_iM_0^{\beta_i}r^{\alpha_i}
	\leq
	\varepsilon r+C_\varepsilon,
	\qquad r\geq0.
	\]
	Using \eqref{HL-general} and choosing
	$\varepsilon\in(0,\kappa_0)$, we obtain
	\begin{align}\label{uniform-LU-upper}
		\mathscr L_\mu U(x)
		\leq C_{M_0},
		\qquad
		x\in\mathbb R^d,\quad
		\mu\in \sP_V^{M_0}(\mathbb{R}^d),
	\end{align}
	for some constant $C_{M_0}>0$  independent of $\mu$.
	Set
	\[
	U_R:=\inf_{|x|\geq R}U(x).
	\]
	Since $U$ has compact level sets,
	\[
	U_R\longrightarrow\infty
	\qquad\text{as }R\to\infty.
	\]
	Define the exit time
	\[
	\tau_R^{x,\mu}
	:=
	\inf\{s\geq0:|X_s^{\mu}(x)|\geq R\}.
	\]
	By It\^o's formula applied to $U$ up to $t\wedge\tau_R^{x,\mu}$, together with  \eqref{uniform-LU-upper},
	\[
	\mathbb E
	U\left(
	X_{t\wedge\tau_R^{x,\mu}}^{\mu}(x)
	\right)
	\leq
	U(x)+C_{M_0}t.
	\]
	On the event $\{\tau_R^{x,\mu}\leq t\}$, we have
	\[
	U\left(X_{\tau_R^{x,\mu}}^{\mu}(x)\right)\geq U_R.
	\]
	Therefore,
	\[
	\mathbb P(\tau_R^{x,\mu}\leq t)
	\leq
	\frac{U(x)+C_{M_0}t}{U_R}.
	\]
	Taking supremum over
	$\mu\in \sP_V^{M_0}(\mathbb{R}^d)$ and $x\in B_L$ gives
	\[
	\sup_{\mu\in \sP_V^{M_0}(\mathbb{R}^d)}
	\sup_{x\in B_L}
	\mathbb P(\tau_R^{x,\mu}\leq t)
	\leq
	\frac{
		\sup_{x\in B_L}U(x)+C_{M_0}t
	}{U_R},
	\]
	which tends to zero as $R\to\infty$.
\end{proof}

We now establish the  weak continuity of the frozen invariant measure map.

\begin{proposition}
\label{weak-continuity-T}
Assume $(\mathbf H_0)$--$(\mathbf H_2)$. Then
\[
\sT:\sP_V^{M_0}(\mathbb{R}^d)\longrightarrow \sP_V^{M_0}(\mathbb{R}^d)
\]
is continuous with respect to the weak topology.
\end{proposition}

\begin{proof}
Let $
\mu_n,\mu\in \sP_V^{M_0}(\mathbb{R}^d)$ with
$
\mu_n\Rightarrow\mu,
$
and set
$
\nu_n:=\sT(\mu_n).
$
By Lemma \ref{tightness},
\(
\nu_n(V)\leq M_0.
\)
Since $V$ has compact level sets, the family
$
\{\nu_n:n\geq1\}
$
is tight.
  Passing to a subsequence, assume that
\(
\nu_n\Rightarrow\nu.
\)
It remains to show that
\(\nu=\mathcal T(\mu)\).

Fix $t>0$ and
$
\varphi\in C_c^\infty(\mathbb R^d).
$
Since $\nu_n$ is invariant for the frozen equation with parameter
$\mu_n$,
\begin{equation}\label{invariance-frozen-n}
\nu_n(\varphi)
=
\nu_n(P_t^{\mu_n}\varphi).
\end{equation}
We first prove that, for every $L>0$,
\begin{equation}\label{local-semigroup-convergence-Lp}
\lim_{n\to\infty}
\sup_{x\in B_L}
\left|
P_t^{\mu_n}\varphi(x)
-
P_t^\mu\varphi(x)
\right|
=0.
\end{equation}
We only need to consider case (i) in $(\mathbf H_0)$.
Fix $R>L$. For $x\in B_L$, let $X^{n}(x)$ and $X^\mu(x)$ be the solutions
of the frozen equations
\begin{align*}
	\dif X_s^{n}
	&=
	b(X_s^{n},\mu_n)\dif s
	+
	\sigma(X_s^{n},\mu_n)\dif W_s,
	\qquad
	X_0^{n}=x,
	\\
	\dif X_s^\mu
	&=
	b(X_s^\mu,\mu)\dif s
	+
	\sigma(X_s^\mu,\mu)\dif W_s,
	\qquad
	X_0=x.
\end{align*}
Set
\[
\tau_R^{n,x}
:=
\inf\{s\geq0:|X_s^{n}(x)|\geq R\},
\qquad
\tau_R^x
:=
\inf\{s\geq0:|X_s^\mu(x)|\geq R\}.
\]
Applying the   stability estimate  Corollary~\ref{Lp} to the two frozen equations starting
from the same point $x$, we obtain
\begin{align}\label{local-stability-case-i}
\sup_{x\in B_L}
\left|
P_t^{\mu_n}\varphi(x)-P_t^\mu\varphi(x)
\right|&\leq
C_{t,L,R,\varphi}
\|b(\cdot,\mu_n)-b(\cdot,\mu)\|_{L^p(B_R)}
\nonumber\\
&\quad+
C_{t,L,R,\varphi}
\|a(\cdot,\mu_n)-a(\cdot,\mu)\|_{L^\infty(B_R)}
\nonumber\\
&\quad+
2\|\varphi\|_\infty
\sup_{x\in B_L}
\left[
\mathbb P(\tau_R^{n,x}\leq t)
+
\mathbb P(\tau_R^x\leq t)
\right].
\end{align}
In fact, since the frozen coefficients
are autonomous, for any sufficiently large finite $q$,
\[
\begin{aligned}
&
\|b(\cdot,\mu_n)-b(\cdot,\mu)\|
_{\mathbb L^q_p((0,t)\times B_R)}=
t^{1/q}
\|b(\cdot,\mu_n)-b(\cdot,\mu)\|_{L^p(B_R)}.
\end{aligned}
\]
Because $p>d$, one can choose $q>2$ sufficiently large such that
\[
\frac{d}{p}+\frac{2}{q}<1.
\]
The local Krylov argument in Corollary \ref{Lp} therefore gives the
first term on the right-hand side of
\eqref{local-stability-case-i}.
Assumption $(\mathbf H_2)$ implies that, for every fixed
$R>L$,
\begin{align*}
&
\|b(\cdot,\mu_n)-b(\cdot,\mu)\|_{L^p(B_R)}+
\|a(\cdot,\mu_n)-a(\cdot,\mu)\|_{L^\infty(B_R)}
\longrightarrow0.
\end{align*}
Moreover, Lemma \ref{uniform-compact-containment} yields
\begin{align*}
\lim_{R\to\infty}
\sup_{\mu\in \sP_V^{M_0}(\mathbb{R}^d)}
\sup_{x\in B_L}
\mathbb P
\left(
\sup_{0\leq s\leq t}|X_s^{\mu}(x)|\geq R
\right)
=0.
\end{align*}
Taking first $n\to\infty$ and then $R\to\infty$ in
\eqref{local-stability-case-i}, proves
\eqref{local-semigroup-convergence-Lp}.

We now return to the invariant measures. From
\eqref{invariance-frozen-n}, we have
\begin{align}\label{weak-continuity-decomposition-Lp}
\left|
\nu(\varphi)-\nu(P_t^\mu\varphi)
\right|
\leq
\left|
\nu(\varphi)-\nu_n(\varphi)
\right|
+
\left|
\nu_n\left(
P_t^{\mu_n}\varphi-P_t^\mu\varphi
\right)
\right|
+
\left|
\nu_n(P_t^\mu\varphi)-\nu(P_t^\mu\varphi)
\right|.
\end{align}
The first and third terms on the right-hand side tend to zero because
\(
\nu_n\Rightarrow\nu
\) and the strong Feller property of the frozen semigroup.
For the middle term, for every $L>0$,
\begin{align}\label{middle-term-Lp}
\left|
\nu_n\left(
P_t^{\mu_n}\varphi-P_t^\mu\varphi
\right)
\right|\leq
\sup_{x\in B_L}
\left|
P_t^{\mu_n}\varphi(x)
-
P_t^\mu\varphi(x)
\right|
+
2\|\varphi\|_\infty\nu_n(B_L^c).
\end{align}
For fixed $L$, the first term on the right-hand side tends to zero by
\eqref{local-semigroup-convergence-Lp}. On the other hand,
\[
\nu_n(B_L^c)
\leq
\frac{\nu_n(V)}
{\inf_{|x|\geq L}V(x)}
\leq
\frac{M_0}
{\inf_{|x|\geq L}V(x)}.
\]
Since $V$ has compact level sets,
\[
\inf_{|x|\geq L}V(x)\longrightarrow\infty
\qquad\text{as }L\to\infty.
\]
Taking first $n\to\infty$ and then $L\to\infty$ in
\eqref{middle-term-Lp}, we obtain
\begin{equation*}
\nu_n\left(
P_t^{\mu_n}\varphi-P_t^\mu\varphi
\right)
\longrightarrow0.
\end{equation*}
Letting $n\to\infty$ in
\eqref{weak-continuity-decomposition-Lp}, we conclude that
\[
\nu(\varphi)
=
\nu(P_t^\mu\varphi),
\qquad
\varphi\in C_c^\infty(\mathbb R^d).
\]
Thus $\nu$ is an invariant  measure for the frozen equation
with parameter $\mu$.
By uniqueness of the invariant  measure of the frozen
equation,
\[
\nu=\sT(\mu).
\]
The proof is finished.
\end{proof}

\subsection{Proof of the existence theorems}
With the properties of $\mathcal T$ established, we divide the proof of the existence results into two parts: first, the general existence under one-level Lyapunov trapping, and second, the extension to systems with singular drifts via the Zvonkin transform.

\subsubsection{Existence under one-level Lyapunov trapping}

We give:

\begin{proof}[{\bf Proof of Theorem \ref{main1}}]
By Lemma \ref{weak-compactness}, the set $\sP_V^{M_0}(\mathbb{R}^d)$ is a compact convex
subset of $\mathscr P(\mathbb R^d)$ endowed with the weak topology.
By Lemma \ref{tightness},
\[
\sT(\sP_V^{M_0}(\mathbb{R}^d))\subset \sP_V^{M_0}(\mathbb{R}^d),
\]
and Proposition \ref{weak-continuity-T}  shows that
\(\mathcal T\) is weakly continuous.
Therefore, by the Schauder--Tychonoff fixed-point theorem,
there exists
$
\mu_\ast\in \sP_V^{M_0}(\mathbb{R}^d)
$
such that
\[
\sT(\mu_\ast)=\mu_\ast.
\]
By the definition of $\sT$, the measure $\mu_\ast$ is invariant for the
frozen equation
\[
	\dif X_t
	=
	b(X_t,\mu_\ast)\dif t
	+
	\sigma(X_t,\mu_\ast)\dif W_t.
	\]
If
$
\mathcal L_{X_0}=\mu_\ast,
$
then
\[
\mathcal L_{X_t}=\mu_\ast,
\qquad t\geq0.
\]
Consequently,
\[
b(X_t,\mu_\ast)
=
b(X_t,\mathcal L_{X_t}),
\qquad
\sigma(X_t,\mu_\ast)
=
\sigma(X_t,\mathcal L_{X_t}),
\]
and the same process is a stationary weak solution of the
McKean--Vlasov equation \eqref{sde0}. Hence $\mu_\ast$ is an invariant
 measure of \eqref{sde0}.
\end{proof}

The following proof verifies that the explicit coefficient-level criterion $(\tilde{\mathbf H}_1)$ implies the abstract Lyapunov condition $(\mathbf H_1)$.
\begin{proof}[{\bf Proof of Corollary \ref{cor1}}]
	We show that $(\mathbf{\tilde H_1})$ implies $(\mathbf H_1)$ with the choice
	\[
	U(x)=|x|^{2+r_3-r_1},
	\qquad
	V(x)=1+|x|^{r_3}.
	\]
	Set
	\[
	q:=2+r_3-r_1.
	\]
	Since \(r_3\geq r_1\), we have \(q\geq2\), and hence
	\(U\in C^2(\mathbb R^d)\).
	
	For the frozen generator \(\mathscr L_\mu\), a direct computation gives
	\[
	\begin{aligned}
		\mathscr L_\mu U(x)
		&=
		q|x|^{q-2}\langle x,b(x,\mu)\rangle
		+
		\frac q2 |x|^{q-2}\|\sigma(x,\mu)\|_{\mathrm{HS}}^2
		\\
		&\quad+
		\frac{q(q-2)}2
		|x|^{q-4}|\sigma(x,\mu)^*x|^2
		\\
		&\leq
		\frac q2 |x|^{q-2}
		\left[
		2\langle x,b(x,\mu)\rangle
		+
		(q-1)\|\sigma(x,\mu)\|_{\mathrm{HS}}^2
		\right].
	\end{aligned}
	\]
	Since
	\[
	q-1=1+r_3-r_1,
	\]
	the assumption $(\mathbf{\widetilde H_1})$ yields
	\[
	\begin{aligned}
		\mathscr L_\mu U(x)
		&\leq
		\frac q2 |x|^{q-2}
		\left[
		-c_1|x|^{r_1}
		+
		c_2|x|^{r_2}\|\mu\|_{r_3}^{r_4}
		+
		c_3
		\right]
		\\
		&=
		-\frac q2c_1|x|^{r_3}
		+
		\frac q2c_2
		|x|^{r_3+r_2-r_1}\|\mu\|_{r_3}^{r_4}
		+
		\frac q2c_3|x|^{r_3-r_1}.
	\end{aligned}
	\]
	Define
	\[
	\eta:=\frac{r_3+r_2-r_1}{r_3},
	\qquad
	\beta:=\frac{r_4}{r_3}.
	\]
	Since $0\leq r_2<r_1$ and $r_3\ge r_1$, we have $0\leq \eta<1$.  Moreover,
	\[
	|x|^{r_3+r_2-r_1}\leq V(x)^\eta,
	\qquad
	|x|^{r_3-r_1}\leq V(x)^\eta,
	\]
	and
	\[
	\|\mu\|_{r_3}^{r_4}
	=
	\left(
	\int_{\mathbb R^d}|x|^{r_3}\mu(\dif x)
	\right)^{r_4/r_3}
	\leq
	\mu(V)^\beta .
	\]
	Therefore
	\[
	\mathscr L_\mu U(x)
	\leq
	-A_0|x|^{r_3}
	+
	A_1V(x)^\eta\mu(V)^\beta
	+
	A_2V(x)^\eta,
	\]
	where
	\[
	A_0:=\frac q2c_1,
	\qquad
	A_1:=\frac q2c_2,
	\qquad
	A_2:=\frac q2c_3.
	\]
	Since \(|x|^{r_3}=V(x)-1\), we get
	\[
	\mathscr L_\mu U(x)
	\leq
	-A_0V(x)
	+
	A_1V(x)^\eta\mu(V)^\beta
	+
	A_2V(x)^\eta
	+
	A_0 .
	\]
	Thus the Lyapunov drift condition \eqref{HL-two-term} holds with
$$
	\kappa_1=A_0,\quad \kappa_2=A_1,\quad\vartheta_1=\eta,\quad\vartheta_2=\beta, \quad \kappa_3=A_2/A_1,\quad\kappa_4=A_0.
$$
The subcritical, critical, and supercritical balance conditions in $(\tilde{\mathbf H}_1)$ correspond exactly to
	\begin{align*}
		{\vartheta_1}+{\vartheta_2}<1
		\quad&\Longleftrightarrow\quad
		r_2+r_4<r_1,\\
		{\vartheta_1}+{\vartheta_2}=1, \kappa_2<\kappa_1
		\quad&\Longleftrightarrow\quad
		r_2+r_4=r_1, c_2<c_1,\\
		{\vartheta_1}+{\vartheta_2}>1
		\quad&\Longleftrightarrow\quad
		r_2+r_4>r_1.
	\end{align*}
	It remains to verify  when $r_2+r_4>r_1$, the trapping condition \eqref{trapping-2} holds,  i.e.
	\begin{equation}\label{A0leq}
		A_0\leq (1-\vartheta^{-1})A_0M_0-A_2M_0^\eta,
	\end{equation}
	where
	$$
	M_0
	=
	\left(
	\frac{c_1}{\vartheta c_2}
	\right)^{\frac1{\vartheta-1}}.
	$$
 Dividing by \(q/2\) in \eqref{A0leq}, we only need to verify that
	\begin{equation*}
	c_1+c_3M_0^{1+(r_2-r_1)/r_3}
	<
	\frac{\vartheta-1}{\vartheta}c_1M_0,
	\end{equation*}
	which is \eqref{trapping-r}.
Hence $(\mathbf H_1)$ holds, and Theorem \ref{main1} applies.
\end{proof}

\subsubsection{Extension to singular drifts}

We now prove the existence result to McKean--Vlasov systems with an additional singular drift.

\begin{proof}[{\bf Proof of Corollary \ref{sing}}]
We only indicate the main steps, since the argument is the same as
that in the proof of \cite[Theorem 2.5]{Zhang2023}; see also
\cite[Lemma 2.5 and the proof of Theorem 2.1]{Wang2023b} for the
corresponding Zvonkin estimates and the preservation of the
coercive Lyapunov structure.
For a frozen measure \(\mu\in\mathscr P_{r_3}\), consider the elliptic
equation
\begin{equation}\label{sing-Zvonkin-equation}
\lambda u_\mu
-
\operatorname{tr}\big(a(\cdot,\mu)\cdot\nabla^2_xu_\mu\big)
-
b_1(\cdot,\mu)\cdot\nabla_x u_\mu
=
b_1(\cdot,\mu).
\end{equation}
By the global \(L^p\)-theory for uniformly elliptic equations and
\eqref{sing-H1-2}, for all sufficiently large \(\lambda\),
\eqref{sing-Zvonkin-equation} has a solution
\(u_\mu\in W^{2,p}(\mathbb R^d)\), and
\begin{equation*}
\sup_{\mu\in\mathscr P_{r_3}}
\left(
\|u_\mu\|_\infty+\|\nabla u_\mu\|_\infty
\right)
\longrightarrow0
\qquad\text{as }\lambda\to\infty.
\end{equation*}
Hence, for large \(\lambda\),
\[
\Phi_\mu(x):=x+u_\mu(x)
\]
is a  \(C^1\)-diffeomorphism, uniformly in \(\mu\).
For the frozen equation (\ref{fr3}), set
\[
Y_t^\mu:=\Phi_\mu(X_t^\mu).
\]
The generalized It\^o formula and
\eqref{sing-Zvonkin-equation} give
\begin{align*}
\dif Y_t^\mu
&=
\tilde b_\mu(Y_t^\mu)\dif t
+
\tilde\sigma_\mu(Y_t^\mu)\dif W_t,
\end{align*}
where
$$
\tilde b_\mu
:=
\big[
(I+\nabla_x u_\mu)\cdot b_0(\cdot,\mu)
+\lambda u_\mu
\big]\circ\Phi_\mu^{-1},
\quad
\tilde\sigma_\mu
:=
\big[
(I+\nabla_x u_\mu)\cdot\sigma(\cdot,\mu)
\big]\circ\Phi_\mu^{-1}.
$$
Thus the singular drift \(b_1\) is removed from the transformed
equation.

Let
\[
q:=2+r_3-r_1.
\]
Using
\eqref{sing-b0-growth}, the uniform Lipschitz estimates for
\(\Phi_\mu\), and Young's inequality, one obtains, exactly as in
\cite[Lemma 2.10]{Zhang2023},
\begin{align}\label{sing-transformed-Lyapunov}
\tilde{\mathscr L}_\mu |y|^q
\leq
-\tilde c_1|y|^{r_3}
+
\tilde c_2
|y|^{r_3+r_2-r_1}\|\mu\|_{r_3}^{r_4}
+
\tilde c_3,
\end{align}
where the constants are independent of \(\mu\). By taking
\(\lambda\) sufficiently large, the perturbation of the leading
constants can be made small. Consequently,
\[
r_2+r_4<r_1,
\]
or, in the critical case,
\[
r_2+r_4=r_1,\qquad
\tilde c_2<\tilde c_1.
\]
It follows from the standard invariant measure estimate associated
with \eqref{sing-transformed-Lyapunov} that there exists \(M>0\) such
that the frozen invariant-measure map
\(
\mathcal T:\mu\longmapsto\nu_\mu
\)
maps \(\mathscr P_{r_3}^M\) into itself.
Finally, the weak continuity of \(\mathcal T\) follows from
\eqref{sing-H2}, the local continuity of \(b_0\)
and \(a\), and the usual localized Zvonkin--Krylov stability
argument. More precisely, one applies the same Zvonkin map to two
frozen equations, stops the processes before leaving a large ball,
and controls the resulting coefficient errors by the Krylov
estimate and the stochastic Gronwall inequality; this is precisely
the argument in \cite[Lemma 2.11]{Zhang2023}. Therefore
\(\mathcal T\) is weakly continuous on the compact convex set
\(\mathscr P_{r_3}^M\).

The Schauder--Tychonoff fixed-point theorem yields a measure
\(\mu_\ast\in\mathscr P_{r_3}^M\) satisfying
\[
\mathcal T(\mu_\ast)=\mu_\ast.
\]
Since \(\Phi_{\mu_\ast}\) is a bijection, the invariant probability measure
of the transformed frozen equation can be pushed forward by
\(\Phi_{\mu_\ast}^{-1}\). The resulting measure is invariant for the original
frozen equation with parameter \(\mu_\ast\). Since \(\mu_\ast\) is a fixed
point of the corresponding frozen invariant measure map, it is an invariant
probability measure of the original McKean--Vlasov SDE \eqref{sde02}. The proof is finished.
\end{proof}

\section{Global and local anchored uniqueness  principle}

In this section, we prove anchored uniqueness and quantitative convergence to equilibrium. The proof is based on the
stability estimate of Section 3, applied with the frozen equation at
$\mu_\ast$ as the reference equation.  A key feature of the argument is its flexibility with respect to the choice of test-function class \(\mathfrak F\). Since the stability estimates in Section~3 provide derivative bounds for the reference semigroup in different topologies, the resulting smallness condition adapts to the specific structure of the law dependence.

\subsection{A global convolution-type Gronwall inequality}

The following uniform-in-time convolution Gronwall lemma provides a unified way to turn the convolution inequality arising from the stability estimates into explicit decay rates, and  its flexibility in handling both exponential and polynomial decays is essential for the quantitative ergodicity results.

\begin{lemma}\label{gron}
Let $\sD:\mathbb R_+\to\mathbb R_+$ be locally bounded and satisfy
\begin{align}\label{geq}
\sD(t)
\leq
A\ell(t)
+
\int_0^t\Gamma(t-s)\sD(s)\dif s,
\qquad t\geq0,
\end{align}
where $A\geq0$, $\ell:\mathbb R_+\to\mathbb R_+$ is locally bounded with
$\ell(t)\to0$ as $t\to\infty$, and $\Gamma(t)\geq0$ is locally integrable. If
\begin{align}\label{gs}
\Lambda_0:=\int_0^\infty\Gamma(s)\dif s<1,
\end{align}
 then
$$
\sD(t)\to0,
\qquad t\to\infty.
$$
Moreover,  the following quantitative estimates hold.
\begin{enumerate}[(i)]

\item
If $\ell(t)\leq \e^{-\lambda t}$
for some $\lambda>0$, and if there exists $\theta\in(0,\lambda]$ such that
\begin{align}\label{es}
\Lambda_\theta
:=
\int_0^\infty \e^{\theta s}\Gamma(s)\dif s<1,
\end{align}
then
\begin{align}\label{eg}
\sD(t)
\leq
\frac{A }{1-\Lambda_\theta}\e^{-\theta t},
\qquad t\geq0.
\end{align}

\item
If
$
\ell(t)\leq  (1+t)^{-\gamma}
$
for some $ \gamma>0$, and if
$$
\Lambda_\gamma:=\sup_{t\geq0}
\int_0^t
\left(\frac{1+t}{1+s}\right)^\gamma
\Gamma(t-s)\dif s
<1,
$$
then
\begin{align}\label{pg}
\sD(t)\leq \frac{A }{1-\Lambda_\gamma}(1+t)^{-\gamma},
\qquad t\geq0.
\end{align}
More generally, the same polynomial rate holds if
\begin{align}\label{gm}
M_\gamma
:=
\int_0^\infty(1+s)^\gamma\Gamma(s)\dif s
<\infty.
\end{align}
More precisely,
\begin{align}\label{gp}
\mathscr D(t)
\leq
A
\left(
1+\int_0^\infty(1+s)^\gamma \sR(s)\dif s
\right)
(1+t)^{-\gamma},
\qquad t\geq0,
\end{align}
where $\sR(t)$ is
the resolvent kernel given by
$
\sR
:=
\sum_{n=1}^\infty\Gamma^{\ast n},
$
and
\begin{align}\label{gR}
\int_0^\infty(1+s)^\gamma \sR(s)\dif s
\leq
M_\gamma
\sum_{n=1}^\infty
n^{\gamma\vee1}\Lambda_0^{\,n-1}
<\infty.
\end{align}
\end{enumerate}
\end{lemma}

\begin{proof}
For the first assertion, by (\ref{gs}), we have
\[
\sum_{n=1}^\infty
\|\Gamma^{\ast n}\|_{L^1(\mR_+)}
=
\sum_{n=1}^\infty\Lambda_0^n
=
\frac{\Lambda_0}{1-\Lambda_0}.
\]
Thus
the resolvent series
$$
\sR:=\sum_{n=1}^\infty \Gamma^{\ast n}
$$
converges in $L^1(\mathbb R_+)$.
Iterating \eqref{geq} yields
\begin{align}\label{res}
\sD(t)\leq A\ell(t)+A\int_0^t\sR(t-s)\ell(s)\dif s.
\end{align}
Since $\ell(t)\to0$ and $\sR\in L^1(\mR_+)$, we have $(\sR\ast\ell)(t)\to0$ as $t\to\infty$. Hence
$\sD(t)\to0$.

For the exponential estimate, define
$$
\sG(t):=\sup_{0\leq s\leq t}\e^{\theta s}\sD(s).
$$
Multiplying \eqref{geq} by $\e^{\theta t}$, we get
\begin{align*}
\e^{\theta t}\sD(t)
&\leq
A \e^{-(\lambda-\theta)t}
+
\int_0^t
\e^{\theta(t-s)}\Gamma(t-s)\e^{\theta s}\sD(s)\dif s\\
&\leq
A
+
\sG(t)\int_0^t \e^{\theta r}\Gamma(r)\dif r
\leq
A +\Lambda_\theta \sG(t).
\end{align*}
Taking the supremum over $0\leq s\leq t$ gives
$$
\sG(t)\leq A +\Lambda_\theta \sG(t).
$$
Since $\Lambda_\theta<1$,
$$
\sG(t)\leq\frac{A }{1-\Lambda_\theta}.
$$
This proves \eqref{eg}.

For the first polynomial estimate, define
$$
\sG_\gamma(t):=\sup_{0\leq s\leq t}(1+s)^\gamma \sD(s).
$$
Multiplying \eqref{geq} by $(1+t)^\gamma$, we obtain
\begin{align*}
(1+t)^\gamma \sD(t)
&\leq
A
+
\int_0^t
\left(\frac{1+t}{1+s}\right)^\gamma
\Gamma(t-s)(1+s)^\gamma \sD(s)\dif s\\
&
\leq
A +\Lambda_\gamma \sG_\gamma(t).
\end{align*}
Taking the supremum over $0\leq s\leq t$ gives
$$
\sG_\gamma(t)\leq A +\Lambda_\gamma \sG_\gamma(t).
$$
Since $\Lambda_\gamma<1$, we get
$$
\sG_\gamma(t)\leq \frac{A }{1-\Lambda_\gamma},
$$
which proves \eqref{pg}.

Finally, assume \eqref{gm}. We first show that the resolvent
kernel has a finite \(\gamma\)-moment. For \(n\geq1\), using
\[
\left(\sum_{i=1}^na_i\right)^\gamma
\leq
n^{(\gamma-1)_+}\sum_{i=1}^na_i^\gamma,
\qquad a_i\geq0,
\]
we obtain
\[
\begin{aligned}
\int_0^\infty
(1+s)^\gamma\Gamma^{\ast n}(s)\dif s
&\quad\leq
n^{(\gamma-1)_+}
\sum_{i=1}^n
\int_{\mathbb R_+^n}
(1+s_i)^\gamma
\prod_{j=1}^n\Gamma(s_j)
\dif s_1\cdots\dif s_n
\\
&\quad=
n^{\gamma\vee1}
M_\gamma\Lambda_0^{\,n-1}.
\end{aligned}
\]
Therefore,
\[
\begin{aligned}
\int_0^\infty(1+s)^\gamma \sR(s)\dif s
&=
\sum_{n=1}^\infty
\int_0^\infty
(1+s)^\gamma\Gamma^{\ast n}(s)\dif s
\\
&\leq
M_\gamma
\sum_{n=1}^\infty
n^{\gamma\vee1}\Lambda_0^{\,n-1}
<\infty,
\end{aligned}
\]
which proves \eqref{gR}. Since
\[
1+t
\leq
(1+s)(1+t-s),
\qquad 0\leq s\leq t,
\]
we have
\[
(1+t-s)^{-\gamma}
\leq
(1+s)^\gamma(1+t)^{-\gamma}.
\]
Using \eqref{res}, we conclude that
\begin{align*}
\mathscr D(t)
&\leq
A (1+t)^{-\gamma}
+
A \int_0^t\sR(s)(1+t-s)^{-\gamma}\dif s
\\
&\leq
A (1+t)^{-\gamma}
\left(
1+\int_0^\infty(1+s)^\gamma \sR(s)\dif s
\right).
\end{align*}
This proves \eqref{gp}.
\end{proof}

\begin{remark}
The lemma shows that the decay rate of \(\mathscr D(t)\) is inherited from that of the forcing term \(\ell(t)\), provided the perturbation kernel \(\Gamma\) is small in the appropriate weighted \(L^1\)-sense.
	In the exponential case, the admissible rate is any
	$
	0<\theta\leq\lambda
	$
	such that (\ref{es}) holds.
	In the polynomial case, the condition \(\Lambda_\gamma<1\) gives a direct
	weighted contraction and the explicit constant
	\((1-\Lambda_\gamma)^{-1}\) in (\ref{pg}). Since
	\[
	\Lambda_\gamma
	\leq
	\int_0^\infty(1+s)^\gamma\Gamma(s)\dif s=M_\gamma,
	\]
	the condition \(M_\gamma<1\) is sufficient for
	\(\Lambda_\gamma<1\). On the other
	hand, under the standing assumption \(\Lambda_0<1\), the weaker requirement
	\(M_\gamma<\infty\) still preserves the rate \((1+t)^{-\gamma}\), no smallness of the
	weighted moment is required. The price for this more flexible criterion is
	that the
	corresponding constant is expressed through the weighted moment of the
	resolvent kernel \(\sR(t)\).
\end{remark}
\subsection{Proof of the uniqueness and ergodicity theorems}

Let $\mu_\ast$ be an invariant   measure of
the McKean--Vlasov SDE \eqref{sde0}, and recall that   $P_t^\ast$ is the semigroup of the frozen reference equation (\ref{ref-mustar}).
We now give:

\begin{proof}[{\bf Proof of Theorem \ref{main2}}]
Let \(\xi\) be an admissible initial random variable, set
\(\nu_0:=\mathcal L_\xi\), and write \(X_t=X_t(\xi)\).
For a test function
$f\in\mathfrak F$, write
\begin{align}\label{anchored-decomposition}
	{\mathcal L_{X_t}}(f)-\mu_\ast(f)
	=
	\left[
	\nu_0P_t^\ast f-\mu_\ast(f)
	\right]
	+
	\left[
	{\mathcal L_{X_t}}(f)-\nu_0P_t^\ast f
	\right].
\end{align}
By the frozen ergodic estimate \eqref{F0}, the first term is controlled by
\begin{align}\label{fi}
	{|\nu_0P_t^\ast f-\mu_\ast(f)|
		\leq C_{\nu_0}
		\ell_\ast(t)\|f\|_{\mathfrak F}
		.}
\end{align}
We now estimate the second term in \eqref{anchored-decomposition} by applying Theorem \ref{thm2-abstract} with the frozen equation
\eqref{ref-mustar} as reference equation and the McKean--Vlasov
equation (\ref{sde0}) as the perturbed equation. Since the  initial distributions are the
same, the initial-distribution term in
Theorem \ref{thm2-abstract} vanishes.  By the  assumption
$({\mathbf U}_{\mathfrak F}^1)$, we get
\begin{equation*}\label{se}
	\begin{split}
		\left|
		{\mathcal L_{X_t}}(f)-\nu_0P_t^\ast f
		\right|
		&\leq
		\|f\|_{\mathfrak F}
		\mathbb E\int_0^t
		\Theta^\ast_{1,\mathfrak F}\big(X_s\big)
		\sK^\ast_{1,\mathfrak F}(t-s)
		|b(X_s,\mathcal L_{X_s})-b(X_s,\mu_\ast)|
		\dif s
		\\
		&\quad+
		\|f\|_{\mathfrak F}
		\mathbb E\int_0^t
		\Theta^\ast_{2,\mathfrak F}\big(X_s\big)
		\sK^\ast_{2,\mathfrak F}(t-s)
		|a(X_s,\mathcal L_{X_s})-a(X_s,\mu_\ast)|
		\dif s .
	\end{split}
\end{equation*}
If the diffusion coefficient is independent of the law, the second integral is
omitted.
Using the coefficient estimates \eqref{HF-b-new}--\eqref{HF-a-new} and the
compatibility bounds \eqref{HF-weight-b}--\eqref{HF-weight-a}, we obtain
\begin{align}\label{second-term-new}
	\left|
	{\mathcal L_{X_t}}(f)-\nu_0P_t^\ast f
	\right|
	\leq
	\|f\|_{\mathfrak F}
	\int_0^t
	\Gamma_{\mathfrak F}(t-s)
	{\bf d}_{\mathfrak F}(\mathcal L_{X_s},\mu_\ast)
	\dif s,
\end{align}
where \(\Gamma_{\mathfrak F}\) is  defined in \eqref{Gamma-F}.
Combining \eqref{fi} and \eqref{second-term-new} and taking the
supremum over $\|f\|_{\mathfrak F}\leq1$ gives the  inequality
\begin{equation}\label{abf}
	{{\bf d}_{\mathfrak F}({\mathcal L_{X_t}},\mu_\ast)
		\leq C_{\nu_0}
		\ell_\ast(t)
		+
		\int_0^t
		\Gamma_{\mathfrak F}(t-s)
		{\bf d}_{\mathfrak F}(\mathcal L_{X_s},\mu_\ast)
		\dif s.}
\end{equation}

To prove uniqueness, let $\nu$ be any invariant   measure of the McKean--Vlasov SDE (\ref{sde0}). Taking
$\nu_0=\nu$, we have
$\mathcal L_{X_t}\equiv\nu$ for any $t\geq0$.
Applying
\eqref{abf} to this stationary law flow gives
$$
{\bf d}_{\mathfrak F}(\nu,\mu_\ast)
\leq C_{\nu_0}
\ell_\ast(t)
+
{\bf d}_{\mathfrak F}(\nu,\mu_\ast)
\int_0^t\Gamma_{\mathfrak F}(s)\dif s.
$$
Letting $t\to\infty$ and using
$$
\ell_\ast(t)\to0,
\qquad
\int_0^\infty\Gamma_{\mathfrak F}(s)\dif s<1,
$$
we obtain
$$
{\bf d}_{\mathfrak F}(\nu,\mu_\ast)=0.
$$
Since ${\bf d}_{\mathfrak F}$ separates probability measures in the
considered class, we conclude that
$$
\nu=\mu_\ast.
$$
Thus uniqueness holds.
For convergence, taking
$$
\sD(t):={\bf d}_{\mathfrak F}(\mathcal L_{X_t},\mu_\ast),
\qquad
A:=C_{\nu_0}
$$  in Lemma \ref{gron} and applying the result to  \eqref{abf}, we obtain convergence to zero \eqref{abstract-convergence}.  The exponential and
polynomial decay estimates \eqref{exp1} and \eqref{pol1} follow
from parts (i) and (ii) of Lemma \ref{gron}, respectively.
This completes the proof.
\end{proof}

Next, we give the proof of the local anchored principle.

\begin{proof}[{\bf Proof of Theorem~\ref{local-anchored}}]
Let \(X_t=X_t(\xi)\) and set
\[
\mathscr D(t):={\bf d}_{\mathfrak F}(\mathcal L_{X_t},\mu_\ast).
\]
According to assumption $({\mathbf U}^{\mathrm{loc}}_{\mathfrak F,r})$, as long as \(\sD(s)<r\), the local coefficient estimates in \(({\mathbf U}_{\mathfrak F}^2)\) are valid.
 By the  anchored
frozen estimate \eqref{local2} and applying the stability estimate of Section~3 with the frozen equation at
\(\mu_\ast\) as the reference dynamics, we obtain
\begin{equation}\label{local-Volterra-inequality}
\sD(t)
\leq  \sD(0)
\ell_\ast(t)
+
\int_0^t
\Gamma_{\mathfrak F,r}(t-s)\sD(s)\,\dif s.
\end{equation}

We first prove local uniqueness. Let \(\nu\) be an invariant
measure such that
\[
{\bf d}_{\mathfrak F}(\nu,\mu_\ast)<r.
\]
The stationary law flow \(\mathcal L_{X_t}\equiv\nu\) remains in
\(B_r^{\mathfrak F}(\mu_\ast)\), so
\eqref{local-Volterra-inequality} holds for every \(t\geq0\). Hence
\[
\begin{aligned}
{\bf d}_{\mathfrak F}(\nu,\mu_\ast)
&\leq \sD(0)
\ell_\ast(t)+
{\bf d}_{\mathfrak F}(\nu,\mu_\ast)
\int_0^t
\Gamma_{\mathfrak F,r}(s)\,\dif s.
\end{aligned}
\]
Letting \(t\to\infty\) and using
\(\ell_\ast(t)\to0\) and
\(\Lambda_{\mathfrak F,r}<1\), we obtain
\[
{\bf d}_{\mathfrak F}(\nu,\mu_\ast)
\leq
\Lambda_{\mathfrak F,r}
{\bf d}_{\mathfrak F}(\nu,\mu_\ast).
\]
Therefore \(\nu=\mu_\ast\).

We next establish the positive invariance of a smaller neighborhood.
Define the first exit time
\[
\tau_r
:=
\inf\{t\geq0:\sD(t)\geq r\}.
\]
For \(T<\tau_r\), set
\[
\sM(T):=\sup_{0\leq t\leq T}\sD(t).
\]
By \eqref{local-Volterra-inequality}, we have
\[
\sM(T)
\leq
L_\ast \sD(0)
+
\Lambda_{\mathfrak F,r}\sM(T).
\]
Consequently,
\begin{equation}\label{local-uniform-bootstrap}
\sM(T)
\leq
\frac{L_\ast \sD(0)}{
1-\Lambda_{\mathfrak F,r}
}.
\end{equation}
If \(\sD(0)\leq r_0\), then
\eqref{local-initial-radius} and
\eqref{local-uniform-bootstrap} imply
\[
\sup_{0\leq t<T}\sD(t)<r.
\]
By the definition of \(\tau_r\), this
excludes a finite exit time. Hence
\(
\tau_r=\infty.
\)

Consequently, \eqref{local-Volterra-inequality} holds for
all $t\geq0$.
Lemma~\ref{gron}, applied with
\[
A=\sD(0),\qquad
\ell=\ell_\ast,\qquad
\Gamma=\Gamma_{\mathfrak F,r},
\]
proves convergence to zero and all the stated quantitative
estimates.
\end{proof}

\section{Non-symmetric granular media dynamics: existence and uniqueness}

In this section, we study the  granular media dynamics (\ref{gran}). In the stability estimates of Section~3, the perturbation kernel describes how the
reference semigroup propagates coefficient errors.  For the uniqueness and ergodicity arguments in Theorem~\ref{main2},   the smallness condition depends
only on the time integral of this kernel. Thus,  it is enough
to identify an admissible kernel and compute its $L^1(\mR_+)$-norm.
We shall derive the explicit kernels for the Ornstein--Uhlenbeck type semigroup that will be used in the granular media analysis in subsection 6.1, and prove the existence and uniqueness of invariant measures in subsection 6.2.

\subsection{Ornstein--Uhlenbeck type reference semigroup and explicit  kernels}

Consider the following Ornstein-Uhlenbeck type SDE:
\begin{equation}\label{ousde}
	\dif X_t
	=
	-\alpha X_t\dif t
	+
	G(X_t)\dif t
	+
	\sigma\dif W_t,
	\qquad X_0=x\in\mathbb R^d,
\end{equation}
where $\alpha>0$ is a constant, $\sigma$ is a  non-degenerate matrix, and
$G:\mathbb R^d\to\mathbb R^d$ is a measurable vector field. Denote by
$P_t^G$ the Markov semigroup associated with \eqref{ousde}. Two distinct regimes will be  considered for the derivative estimates of $P_t^G$: the bounded-test-function   framework, where the kernel inherits short-time singularities from gradient estimates, and the Lipschitz  framework, where the kernel is simply exponential and has no singularity.

Let $p\in(d,\infty]$ and set
$
q:=p/(p-1),
$
with the convention $q=1$ when $p=\infty$.
We first give the following result, which serves as the building block for the perturbed case.
\bl
Let $P_t^0$ be the semigroup of the linear Ornstein--Uhlenbeck equation
\[
\dif Y_t=-\alpha Y_t\dif t+\sigma\dif W_t.
\]
Then we have for every $f\in L^p(\mathbb R^d)$,
\begin{align}\label{0-kp}
	\|\nabla_x P_t^0f\|_\infty
	\leq
	k_{\alpha,\sigma,p}(t)\|f\|_{L^p},
\end{align}
where
\begin{align}\label{OU-kp}
	k_{\alpha,\sigma,p}(t)
	:=
	\mathfrak c_{d,q}
	\|\sigma^{-1}\|
	|\det\sigma|^{-1/p}
	\e^{-\alpha t}
	\left(
	\frac{2\alpha}{1-\e^{-2\alpha t}}
	\right)^{\frac12+\frac{d}{2p}},
	\qquad t>0,
\end{align}
with
\begin{align*}
	\mathfrak c_{d,q}
	:=
	(2\pi)^{-d/2}
	\left(
	\int_{\mathbb R^d}
	|z_1|^q \e^{-q|z|^2/2}\dif z
	\right)^{1/q}.
\end{align*}
Moreover,
\begin{align}\label{Ap}
	A_{\alpha,\sigma,p}
	:=
	\int_0^\infty k_{\alpha,\sigma,p}(t)\dif t=\mathfrak c_{d,q}
	\|\sigma^{-1}\|
	|\det\sigma|^{-1/p}
	(2\alpha)^{-\frac12+\frac{d}{2p}}
	B\left(
	\frac12,\frac12-\frac{d}{2p}
	\right).
\end{align}

In particular, for $p=\infty$, we have
\begin{align*}
	\mathfrak c_{d,1}
	=
	\sqrt{\frac{2}{\pi}},\qquad k_{\alpha,\sigma,\infty}(t)
	=
	2\|\sigma^{-1}\|
	\sqrt{\frac{\alpha}{\pi}}
	\frac{\e^{-\alpha t}}
	{\sqrt{1-\e^{-2\alpha t}}},
	\qquad t>0,
\end{align*}
and
\begin{align}\label{Ainfty}
	A_{\alpha,\sigma,\infty}
	=
	\|\sigma^{-1}\|
	\sqrt{\frac{\pi}{\alpha}}.
\end{align}
\el

\begin{proof}
	By direct computation, we have
	\[
	P_t^0f(x)
	=
	\int_{\mathbb R^d}
	f(\e^{-\alpha t}x+z)g_{Q_t}(z)\dif z,
	\]
	where
	\[
	Q_t
	=
	\frac{1-\e^{-2\alpha t}}{2\alpha}\sigma\sigma^\ast
	\]
	and $g_{Q_t}$ is the centered Gaussian density with covariance $Q_t$.
	For a unit vector $v\in\mathbb R^d$,
	\[
	\<\nabla_xP_t^0f(x),v\>
	=
	-\e^{-\alpha t}
	\int_{\mathbb R^d}
	f(\e^{-\alpha t}x+z)
	\left\langle Q_t^{-1}z,v\right\rangle
	g_{Q_t}(z)\dif z.
	\]
	Therefore,  H\"older's inequality gives
	\[
	|\<\nabla_xP_t^0f(x),v\>|
	\leq
	\e^{-\alpha t}\|f\|_{L^p}
	\left\|
	\left\langle Q_t^{-1}\cdot,v\right\rangle
	g_{Q_t}
	\right\|_{L^q}.
	\]
	After the change of variables $z=Q_t^{1/2}u$, we obtain
	\[
	\left\|
	\left\langle Q_t^{-1}\cdot,v\right\rangle
	g_{Q_t}
	\right\|_{L^q}
	\leq
	\mathfrak c_{d,q}
	|Q_t^{-1/2}v|
	(\det Q_t)^{-1/(2p)}.
	\]
	Since
	\[
	|Q_t^{-1/2}v|
	\leq
	\|\sigma^{-1}\|
	\left(
	\frac{2\alpha}{1-\e^{-2\alpha t}}
	\right)^{1/2}
	\]
	and
	\[
	(\det Q_t)^{-1/(2p)}
	=
	|\det\sigma|^{-1/p}
	\left(
	\frac{2\alpha}{1-\e^{-2\alpha t}}
	\right)^{d/(2p)},
	\]
	we obtain (\ref{0-kp}).

	We next calculate the time integral of $k_{\alpha,\sigma,p}$. Set
	\[
	\delta_p:=\frac12+\frac{d}{2p}.
	\]
	Since \(p>d\), we have \(\frac12-\frac{d}{2p}>0\). Using the substitution $r=\e^{-2\alpha t}$, we get
	\[
	\begin{aligned}
		\int_0^\infty
		\e^{-\alpha t}
		\left(
		\frac{2\alpha}{1-\e^{-2\alpha t}}
		\right)^{\delta_p}
		\dif t
		&=
		(2\alpha)^{\delta_p-1}
		\int_0^1
		r^{-1/2}(1-r)^{-\delta_p}\dif r
		\\
		&=
		(2\alpha)^{-\frac12+\frac{d}{2p}}
		B\left(
		\frac12,\frac12-\frac{d}{2p}
		\right).
	\end{aligned}
	\]
	This proves \eqref{Ap}. The \(p=\infty\) case follows by direct substitution.
\end{proof}

Next, we establish the following result.

\begin{theorem}
	\label{ouk}
	The following assertions hold.
	
	\begin{enumerate}[(i)]
		\item \textbf{Bounded-test-function framework under an $L^p$ drift.}
		\begin{align}\label{gg}
		G\in L^p(\mathbb R^d),
		\qquad
		M:=\|G\|_{L^p}<\infty.
		\end{align}
		Then for every bounded function $f$,
		\begin{align*}
			\|\nabla_x P_t^Gf\|_\infty
			\leq
			K_{\alpha,\sigma,M,p}(t)\|f\|_\infty,
			\qquad t>0,
		\end{align*}
		where for $t>0$,
		\begin{align}\label{KM}
			K_{\alpha,\sigma,M,p}(t)
			:=
			\sum_{n=0}^\infty
			M^n
			\left(
			k_{\alpha,\sigma,p}^{\ast n}
			\ast
			k_{\alpha,\sigma,\infty}
			\right)(t).
		\end{align}
		Here $k^{*n}$ denotes the $n$-fold convolution on $\mathbb R_+$, and
		$k_{\alpha,\sigma,p}(t)$ is given by (\ref{OU-kp}).
		
		If
		\begin{align*}
			M A_{\alpha,\sigma,p}<1,
		\end{align*}
		where $A_{\alpha,\sigma,p}$ is given by (\ref{Ap}), then $K_{\alpha,\sigma,M,p}\in L^1(\mathbb R_+)$ and
		\begin{align}\label{OU-K-Mp-integral}
			\int_0^\infty
			K_{\alpha,\sigma,M,p}(t)\dif t
			=
			\frac{
				\|\sigma^{-1}\|\sqrt{\pi/\alpha}
			}{
				1-MA_{\alpha,\sigma,p}
			}.
		\end{align}
		
		In particular, for $p=\infty$, we have
		\begin{align}\label{OU-recover-infty}
			\int_0^\infty
			K_{\alpha,\sigma,M,\infty}(t)\dif t
			=
			\frac{
				\|\sigma^{-1}\|\sqrt{\pi/\alpha}
			}{
				1-
				M\|\sigma^{-1}\|\sqrt{\pi/\alpha}
			}.
		\end{align}

		\item \textbf{Lipschitz-test-function framework.}
		Assume that $G$ is one-sided Lipschitz, namely there exists $\eta_G\in\mathbb R$
		such that
		\begin{equation*}\label{one-sided-G}
			\langle G(x)-G(y),x-y\rangle
			\leq
			\eta_G|x-y|^2,
			\qquad x,y\in\mathbb R^d.
		\end{equation*}
		Then, for every Lipschitz continuous function $f$,
		\begin{align}\label{OUW}
			|\nabla_x P_t^Gf(x)|
			\leq
			\e^{-(\alpha-\eta_G)t}\operatorname{Lip}(f),
			\qquad t\geq0,\ x\in\mathbb R^d.
		\end{align}
		Thus, in the $W_1$ framework, the drift perturbation kernel is
		\begin{equation*}\label{OW}
			K_{W_1}^{G}(t)
			=
			\e^{-(\alpha-\eta_G)t}.
		\end{equation*}
		If $\eta_G<\alpha$, then
		\begin{equation*}\label{Wk}
			\int_0^\infty K_{W_1}^{G}(t)\dif t
			=
			\frac1{\alpha-\eta_G}.
		\end{equation*}
	\end{enumerate}
\end{theorem}

\begin{proof}
	We first prove assertion (i). For $G\in L^p$,
	by Duhamel's formula we have
	\[
	P_t^Gf=
	P_t^0f
	+
	\int_0^t
	P_{t-s}^0
	\left(
	G\cdot\nabla_x P_s^Gf
	\right)\dif s.
	\]
	Under the assumption (\ref{gg}),
	\[
	\|G\cdot\nabla P_s^Gf\|_{L^p}
	\leq
	M\|\nabla P_s^Gf\|_\infty.
	\]
	Consequently, for bounded $f$, by (\ref{0-kp}) we deduce that
	\[
	\begin{aligned}
		\|\nabla_x P_t^Gf\|_\infty
		&\leq
		k_{\alpha,\sigma,\infty}(t)\|f\|_\infty
		+
		M\int_0^t
		k_{\alpha,\sigma,p}(t-s)
		\|\nabla_x P_s^Gf\|_\infty\dif s.
	\end{aligned}
	\]
	Iterating this  inequality yields
	\[
	\|\nabla_x P_t^Gf\|_\infty
	\leq
	K_{\alpha,\sigma,M,p}(t)\|f\|_\infty,
	\]
	where $K_{\alpha,\sigma,M,p}$ is given by (\ref{KM}).
	If $MA_{\alpha,\sigma,p}<1$, then Tonelli's theorem gives
	\[
	\begin{aligned}
		\int_0^\infty
		K_{\alpha,\sigma,M,p}(t)\dif t
		&=
		\sum_{n=0}^\infty
		M^n
		A_{\alpha,\sigma,p}^{\,n}
		A_{\alpha,\sigma,\infty}
		\\
		&=
		\frac{
			A_{\alpha,\sigma,\infty}
		}{
			1-MA_{\alpha,\sigma,p}
		}.
	\end{aligned}
	\]
	When $p=\infty$, \eqref{OU-recover-infty} follows by (\ref{Ainfty}).
	
	We now prove the assertion (ii).  Let $X_t^x$ and $X_t^y$ be two solutions of
	\eqref{ousde} driven by the same Brownian motion and starting from $x$ and
	$y$, respectively.  Then
	\[
	\dif(X_t^x-X_t^y)
	=
	-\alpha(X_t^x-X_t^y)\dif t
	+
	\big(G(X_t^x)-G(X_t^y)\big)\dif t .
	\]
	By the one-sided Lipschitz condition \eqref{one-sided-G},
	\[
	\frac{\dif}{\dif t}|X_t^x-X_t^y|^2
	\leq
	-2(\alpha-\eta_G)|X_t^x-X_t^y|^2.
	\]
	Thus we have
	\[
	|X_t^x-X_t^y|
	\leq
	\e^{-(\alpha-\eta_G)t}|x-y|.
	\]
	Consequently, for every Lipschitz continuous function $f$,
	\[
	\begin{aligned}
		|P_t^Gf(x)-P_t^Gf(y)|
		&\leq
		\mathbb E|f(X_t^x)-f(X_t^y)|
		\\
		&\leq
		\operatorname{Lip}(f)
		\e^{-(\alpha-\eta_G)t}|x-y|,
	\end{aligned}
	\]
	which in turn  implies (\ref{OUW}). The \(W_1\) kernel and its integral follow immediately.
\end{proof}

\begin{remark}[Two explicit kernels and their roles]
	The two preceding results make explicit the relation between the general stability estimates of Section~3 and the concrete Ornstein--Uhlenbeck type computations. In the bounded-test-function or  total variation framework, the perturbation estimate uses the gradient of the Ornstein--Uhlenbeck semigroup. The kernel takes the form
	\[
	K_{\rm TV}^{\rm OU}(t)
	=
	2\|\sigma^{-1}\|
	\sqrt{\frac{\alpha}{\pi}}
	\frac{\e^{-\alpha t}}{\sqrt{1-\e^{-2\alpha t}}},
	\]
	which behaves like \(t^{-1/2}\) as \(t\downarrow0\) and like \(\e^{-\alpha t}\) as \(t\to\infty\).  This is precisely the Ornstein--Uhlenbeck counterpart of the kernel \(\sK_{1,\mathfrak F}\) in Section~3, where the short-time singularity reflects the parabolic regularization of the semigroup and the long-time tail is controlled by the ergodic rate of the reference equation.
	
	In contrast, in the Lipschitz or Wasserstein-1 framework, the explicit linear structure of the Ornstein--Uhlenbeck equation yields the sharper kernel
	$$
	K_{W_1}^{\rm OU}(t)=\e^{-\alpha t},
	$$
	which has no short-time singularity. Since \(W_1(\mu,\nu)\le {\bf d}_V(\mu,\nu)\) whenever \(V(x)\ge 1+|x|\), this \(W_1\)-based estimate is compatible with the weighted total variation framework, but it is  sharper for mean-field examples where the law dependence acts through low-order moments.
\end{remark}

The following result records the ergodic properties of the Ornstein--Uhlenbeck semigroup with an \(L^p\)-drift and the uniform \(L^q\)-bound on its invariant density, which will be used in the uniqueness proof for granular media equations. The proof is omitted, see e.g. \cite{Huang2025,XZ}.

\begin{lemma}\label{frozen-invariant-density}
	For every \(G\in L^p\) with \(p\in(d,\infty]\), SDE \eqref{ousde} admits a unique invariant  measure \(\nu_G\), and for every initial distribution $\nu_0$, there exist   constants \(C_{\nu_0},\lambda>0\) such that
	\begin{align}\label{ouTV}
			\left\|
	    \nu_0 P_t^G-\nu_G
		\right\|_{\rm TV}
		\leq
		C_{\nu_0}\, \e^{-{ \lambda} t},
		\qquad t\geq0.
	\end{align}
	Moreover, \(\nu_G\) has a density \(\rho_G\in L^q(\mathbb R^d)\) with \(q=p/(p-1)\), and
	\begin{align}\label{invariant-density-constant}
		\mathfrak C_{p,M}:=\sup_{\|G\|_{L^p}\leq M}
		\|\rho_G\|_{L^q}<\infty.
	\end{align}
	When \(p=\infty\), one has \(q=1\) and \(\mathfrak C_{\infty,M}=1\).
\end{lemma}

\subsection{Existence and uniqueness of invariant measures}

We give the proof of the existence result stated in Theorem~\ref{coe} by  verifying the assumptions of the general existence theorems.

\begin{proof}[\bf{Proof of Theorem \ref{coe}}]
	Since $\sigma$ is constant and non-degenerate, the local
	ellipticity assumption $(\mathbf H_0)$ is satisfied.  For
	\(\mu\in\mathscr P(\mathbb R^d)\), define
	\[
	G_\mu(x)
	:=
	\int_{\mathbb R^d}F(x,y)\mu(\dif y).
	\]
	Then the drift in \eqref{gran} is
	\[
	b(x,\mu)=-\alpha x+G_\mu(x).
	\]
	
	\medskip
	\noindent
	\textit{Proof of assertion (i).} By Minkowski's inequality, we have
	\[
	\sup_{\mu\in\mathscr P(\mathbb R^d)}
	\left\|
	\int_{\mathbb R^d}F(\cdot,y)\,\mu(\dif y)
	\right\|_{L^p}
	\leq
	\sup_{y\in\mathbb R^d}\|F(\cdot,y)\|_{L^p}
	<\infty.
	\]
	Moreover, the  continuity of
	$y\mapsto F(\cdot,y)|_{B_R}$ as an $L^p(B_R)$-valued map
	implies that
	\[
	\mu_n\Rightarrow\mu
	\quad\Longrightarrow\quad
	\|G_{\mu_n}-G_\mu\|_{L^p(B_R)}\longrightarrow0,
	\qquad R>0.
	\]
	We apply Corollary~\ref{sing} with
	$b_0(x,\mu)=-\alpha x$ and $b_1(x,\mu)=G_\mu(x)$.
	Indeed, the constant non-degenerate diffusion satisfies the
	required regularity assumptions. For a uniformly chosen
	Zvonkin transformation with
	$\sup_\mu\|\nabla_x u_\mu\|_\infty<1/2$, the transformed
	coefficients satisfy
	\[
	2\langle y,\tilde b_\mu(y)\rangle
	+\|\tilde\sigma_\mu(y)\|_{\mathrm{HS}}^2
	\leq -A|y|^2+C
	\]
	with constants $A>0$ and $C<\infty$ independent of $\mu$.
	Writing $D:=\sup_\mu\|u_\mu\|_\infty$, the trapping condition
	holds with $r_1=r_3=2$, $r_2=r_4=0$ and any sufficiently
	large $M>D+\sqrt{C/A}$.
	Corollary~\ref{sing} therefore yields an invariant
	measure in $\sP_2(\mathbb R^d)$.

	\medskip
	\noindent
	\textit{Proof of assertion (ii).} Let \(M_0>0\) be the trapping level appearing in the relevant
	subcritical, critical, or supercritical balance condition. By assumption, the map
	\[
	\Phi_R:y\longmapsto F(\cdot,y)|_{B_R}
	\]
	is continuous with values in $L^\infty(B_R)$ and satisfies
	\[
	\|\Phi_R(y)\|_{L^\infty(B_R)}
	\leq C_R(1+|y|^\delta),
	\qquad \delta<r.
	\]
	For $L>1$, choose $\chi_L\in C_c(\mathbb R^d;[0,1])$
	with $\chi_L=1$ on $B_L$.
	Since $\chi_L\Phi_R$ is a bounded continuous
	$L^\infty(B_R)$-valued map with compact support, weak
	convergence implies
	\[
	\left\|
	\int_{\mathbb R^d}
	\chi_L(y)\Phi_R(y)(\mu_n-\mu)(\dif y)
	\right\|_{L^\infty(B_R)}
	\longrightarrow0.
	\]
	On the other hand, the uniform $r$-moment bound gives
	\[
	\begin{aligned}
		&
		\int_{\mathbb R^d}
		(1-\chi_L(y))
		\|\Phi_R(y)\|_{L^\infty(B_R)}
		(\mu_n+\mu)(\dif y)
		\\
		&\qquad\leq
		2C_RM_0\left(L^{-r}+L^{\delta-r}\right).
	\end{aligned}
	\]
	Consequently,
	\[
	\limsup_{n\to\infty}
	\|G_{\mu_n}-G_\mu\|_{L^\infty(B_R)}
	\leq
	2C_RM_0\left(L^{-r}+L^{\delta-r}\right).
	\]
	Letting $L\to\infty$ proves the continuity assumption
	$(\widetilde{\mathbf H}_2)$, since $-\alpha x$ and $\sigma$
	are independent of the law.
	The assumed Lyapunov and balance conditions then allow us
	to apply Corollary~\ref{cor1}, yielding an invariant
	measure in $\mathscr P_r(\mathbb R^d)$.
\end{proof}

We now prove the uniqueness and convergence results stated in Theorem~\ref{ou}.

\begin{proof}[{\bf Proof of Theorem \ref{ou}}]
	We divide the proof into two parts, corresponding to the two different topologies.

	\vspace{1mm}
	\noindent {\it  (i) (Total variation criterion)}.
	Let $\mu_\ast$ be an invariant  measure, whose existence follows
	from Theorem \ref{coe}, and define the frozen drift
	\[
	G_\ast(x)
	:=
	\int_{\mathbb R^d}F(x,y)\mu_\ast(\dif y).
	\]
	By Minkowski's inequality,
	\[
	\|G_\ast\|_{L^p}
	\leq
	\int_{\mathbb R^d}
	\|F(\cdot,y)\|_{L^p}\mu_\ast(\dif y)
	\leq
	M_p.
	\]
	Thus the frozen reference equation is  an Ornstein--Uhlenbeck equation perturbed by a
	$L^p$-drift:
	\[
	\dif\bar X_t
	=
	[-\alpha\bar X_t+G_\ast(\bar X_t)]\dif t
	+
	\sigma\dif W_t.
	\]
	By Theorem \ref{ouk}, its semigroup \(P_t^\ast\) satisfies
	\[
	\|\nabla_x P_t^\ast f\|_\infty
	\leq
	K_{\alpha,\sigma,M_p,p}(t)\|f\|_\infty.
	\]
	Let $\nu$ be another invariant   measure of \eqref{gran}, and define
	\[
	G_\nu(x)
	:=
	\int_{\mathbb R^d}F(x,y)\nu(\dif y).
	\]
	The stationary McKean--Vlasov equation with marginal law $\nu$ is precisely
	the frozen autonomous equation
	\[
	\dif X_t
	=
	[-\alpha X_t+G_\nu(X_t)]\dif t+\sigma\dif W_t.
	\]
	Since
	\(
	\|G_\nu\|_{L^p}\leq M_p,
	\)
	Lemma \ref{frozen-invariant-density} implies that $\nu$ has a density
	$\rho_\nu\in L^q(\mathbb R^d)$ satisfying
	\[
	\|\rho_\nu\|_{L^q}
	\leq
	\mathfrak C_{p,M_p}.
	\]
	
	Let $(X_t^\nu)_{t\geq0}$ be a stationary solution with
	\[
	\mathcal L_{X_t^\nu}=\nu,
	\qquad t\geq0.
	\]
	Since the marginal law is constant, its density equals  $\rho_\nu$ at
	every time. In particular,
	\[
	\sup_{s\geq0}
	\|\rho_s^\nu\|_{L^q}
	=
	\|\rho_\nu\|_{L^q}
	\leq
	\mathfrak C_{p,M_p}.
	\]
	Applying  Proposition \ref{Lp-drift-TV-stability} with \(t_0=0\) gives, for every \(T>0\),
	\[
	\left\|
	\nu-\nu P_T^\ast
	\right\|_{\rm TV}
	\leq
	\mathfrak C_{p,M_p}
	\int_0^T
	K_{\alpha,\sigma,M_p,p}(T-s)
	\|G_\nu-G_\ast\|_{L^p}\dif s.
	\]
	By the definition \eqref{gran-Mp}, we have
	\[
	\|G_\nu-G_\ast\|_{L^p}
	\leq
	M_p\|\nu-\mu_\ast\|_{\rm TV}.
	\]
	Therefore,
	\[
	\left\|
	\nu-\nu P_T^\ast
	\right\|_{\rm TV}
	\leq
	\mathfrak C_{p,M_p}M_p
	\|\nu-\mu_\ast\|_{\rm TV}
	\int_0^T
	K_{\alpha,\sigma,M_p,p}(s)\dif s.
	\]
	On the other hand,
	\[
	\begin{aligned}
		\|\nu-\mu_\ast\|_{\rm TV}
		&\leq
		\|\nu-\nu P_T^\ast\|_{\rm TV}
		+
		\|\nu P_T^\ast-\mu_\ast\|_{\rm TV}
		\\
		&\leq
		\mathfrak C_{p,M_p}M_p
		\|\nu-\mu_\ast\|_{\rm TV}
		\int_0^T
		K_{\alpha,\sigma,M_p,p}(s)\dif s
		+
		{ C_0}\,\e^{-\lambda T}.
	\end{aligned}
	\]
	The second term decays exponentially by the frozen ergodic estimate \eqref{ouTV}.  Letting $T\to\infty$, we obtain
	\[
	\|\nu-\mu_\ast\|_{\rm TV}
	\leq
	\mathfrak C_{p,M_p}M_p
	\left(
	\int_0^\infty
	K_{\alpha,\sigma,M_p,p}(s)\dif s
	\right)
	\|\nu-\mu_\ast\|_{\rm TV}.
	\]
	Hence $\nu=\mu_\ast$ provided
	\[
	\mathfrak C_{p,M_p}M_p
	\int_0^\infty
	K_{\alpha,\sigma,M_p,p}(s)\dif s<1.
	\]
	By \eqref{OU-K-Mp-integral}, this condition is equivalent to
	\[
	M_p
	\left(
	A_{\alpha,\sigma,p}
	+
	\mathfrak C_{p,M_p}
	A_{\alpha,\sigma,\infty}
	\right)
	<1,
	\]
	which is exactly \eqref{gran-Lp-uniqueness-small}. This proves uniqueness.
	
	To prove exponential convergence, let the initial distribution  $\mathcal{L}_{X_0}=\nu_0$, and $\rho_t^{\nu_0}$ be the density of $\mathcal{L}_{X_t}$.  Fix $t_0>0$ such that
	\[
	\mathfrak C_{p,t_0}^{\nu_0}
	:=
	\sup_{t\geq t_0}
	\|\rho_t^{\nu_0}\|_{L^q}
	<\infty,
	\]
	which is guaranteed by  \cite[Proposition 5.4]{Huang2025}. For $t\geq t_0$, decompose
\[
		{\mathcal L_{X_t}}-\mu_\ast
		=
		\left(
		{\mathcal L_{X_t}}-\mathcal L_{X_{t_0}}P_{t-t_0}^\ast
		\right)
		+
		\left(
		\mathcal L_{X_{t_0}}P_{t-t_0}^\ast-\mu_\ast
		\right).
		\]
	By Proposition \ref{Lp-drift-TV-stability} on the interval \([t_0,t]\), we have
	\begin{align*}
		&
		\left|
		\mathbb E f(X_t)
		-
		\mathbb E f\big(\bar X_{t-t_0}^\ast(X_{t_0})\big)
		\right|\leq 	\mathfrak C_{p,t_0}^{\nu_0}\|f\|_\infty
		\int_{t_0}^t
		K_{\alpha,\sigma,M_p,p}(t-s)
		\|G_{\cL_{X_s}}-G_{\mu_\ast}\|_{L^p}\dif s
		\nonumber\\
		&\qquad\leq
		 	\mathfrak C_{p,t_0}^{\nu_0}M_p\|f\|_\infty
		\int_{t_0}^t
		K_{\alpha,\sigma,M_p,p}(t-s)
		\|\mathcal L_{X_s}-\mu_\ast\|_{\rm TV}
		\dif s.
	\end{align*}
	 Taking the supremum over
\(\|f\|_\infty\leq1\), and using the frozen
exponential ergodicity \eqref{ouTV}, we obtain
\[
\begin{aligned}
\|\mathcal L_{X_t}-\mu_\ast\|_{\rm TV}
&\leq
C_0\e^{-\lambda(t-t_0)}
\\
&\quad+
\mathfrak C_{p,t_0}^{\nu_0}M_p
\int_{t_0}^t
K_{\alpha,\sigma,M_p,p}(t-s)
\|\mathcal L_{X_s}-\mu_\ast\|_{\rm TV}
\,\dif s .
\end{aligned}
\]
	Define
	\[
	\mathscr D_{t_0}(r)
	:=
	\|\mathcal L_{X_{t_0+r}}-\mu_\ast\|_{\rm TV},
	\qquad r\geq0.
	\]
	After the change of variables $s=t_0+u$, the preceding inequality becomes
	\begin{align*}
		\mathscr D_{t_0}(r)
		&\leq
		{ C_0\e^{-\lambda r}}
		+
		\mathfrak C_{p,t_0}^{\nu_0}M_p
		\int_0^r
		K_{\alpha,\sigma,M_p,p}(r-u)
		\mathscr D_{t_0}(u)\dif u.
	\end{align*}
	This is exactly the convolution inequality in Lemma \ref{gron}, with
	\[
	{\Gamma(r)}
	=
		\mathfrak C_{p,t_0}^{\nu_0} M_p
	K_{\alpha,\sigma,M_p,p}(r).
	\]
	If there exists $\theta\in(0,\lambda]$ such that
	\[
	\mathfrak C_{p,t_0}^{\nu_0} M_p
	\int_0^\infty
	\e^{\theta r}
	K_{\alpha,\sigma,M_p,p}(r)\dif r
	<1,
	\]
	then Lemma \ref{gron} gives
	\[
	\|{\mathcal L_{X_t}}-\mu_\ast\|_{\rm TV}
	\leq
	C_{\nu,t_0,\theta}\e^{-\theta t},
	\qquad t\geq t_0.
	\]
	This proves the quantitative ergodicity assertion in the total variation criterion.
	
	\vspace{1mm}
	\noindent  {\it (ii) ($W_1$ criterion)}. We now turn to the \(W_1\)-based uniqueness result. For the frozen drift
	$
	G_\ast(x)
	$,
	the definition of $\mathfrak L_x$ gives
	\[
	\langle G_\ast(x)-G_\ast(z),x-z\rangle
	\leq
	\mathfrak L_x|x-z|^2.
	\]
	Hence by Theorem \ref{ouk}, the frozen semigroup $P_t^\ast$ satisfies
	$$
	|\nabla_xP_t^\ast f|
	\leq
	\e^{-(\alpha-\mathfrak L_x)t}\operatorname{Lip}(f).
	$$
	Furthermore, by Kantorovich--Rubinstein duality and the Lipschitz continuity
	of $F$ in its second variable,
	\[
	\sup_{x\in\mathbb R^d}
	\left|
	\int_{\mathbb R^d}F(x,y)(\mu-\nu)(\dif y)
	\right|
	\leq
	\mathfrak L_yW_1(\mu,\nu).
	\]
	Thus the anchored perturbation kernel is
	\[
	\Gamma_{W_1}(t)
	=
	\mathfrak L_y \e^{-(\alpha-\mathfrak L_x)t},
	\]
	and
	\[
	\int_0^\infty\Gamma_{W_1}(t)\dif t
	=
	\frac{\mathfrak L_y}{\alpha-\mathfrak L_x}.
	\]
	Condition \eqref{sh} is precisely the requirement that this quantity be
	strictly less than one.
	The conclusion follows  from Corollary~\ref{cor:W1-unique}. This completes the proof of Theorem~\ref{ou}.
\end{proof}

\section{Dynamical Curie--Weiss model}

This section is devoted to the dynamical Curie--Weiss model \eqref{ex2}, i.e.,
 \begin{align*}
	\dif X_t
	=
	-(\gamma+c)X_t\dif t
	+
	\sqrt c\,\tanh\left(\sqrt c\,\mE X_t\right)\dif t
	+
	\dif W_t,
	\qquad X_0=\xi.
\end{align*}
Our purpose is threefold: to identify all invariant measures, to quantify the convergence in the uniqueness
and critical regimes, and to show how the local anchored principle, combined
with the closed mean equation, yields basin-dependent ergodicity when several
invariant measures coexist.  For simplicity,  for $\mu\in\sP(\mathbb{R})$, we write
$$
m_\mu:=\int_{\mR}x\mu(\dif x)
$$
whenever the integral is finite.

\subsection{Invariant measures and phase transition}

We first give:
\begin{proof}[{\bf Proof of Theorem \ref{tan1}}]
We divide the proof into four steps.

\smallskip
\noindent
{\it Step 1. Existence.}  Let $$
b(x,\mu)=-(\gamma+c)x+\sqrt c\,\tanh(\sqrt c\,m_\mu),
\qquad
\sigma(x,\mu)=1.
$$
We apply the existence criterion on moment balls in
\(\mathscr P_2(\mathbb R)\). Assumption $(\mathbf H_0)$ is immediate since the diffusion coefficient is a constant and the drift is locally bounded
in $x$.  Moreover,
$$
2xb(x,\mu)+1
=
-2(\gamma+c)x^2
+
2\sqrt c\,x\tanh(\sqrt c\,m_\mu)+1.
$$
By Young's inequality and the boundedness of \(\tanh\), there exists a constant \(C_1>0\) such that
$$
2xb(x,\mu)+1
\leq
-(\gamma+c)x^2+C_1.
$$
Thus the Lyapunov condition in $(\mathbf{\tilde H_1})$ holds with
$r_1=2$, $r_2=0$ and $r_4=0$.  Moreover, if \(\mu_n,\mu\in\mathscr P_2^M(\mathbb R)\) and \(\mu_n\Rightarrow\mu\), then
the uniform second-moment bound implies \(m_{\mu_n}\to m_\mu\). Hence
\[
b(\cdot,\mu_n)\to b(\cdot,\mu)
\]
locally uniformly, and \((\widetilde{\mathbf H}_2)\) holds. Therefore
Corollary \ref{cor1} yields the existence of an invariant measure.

\smallskip
\noindent
{\it Step 2. Identification of invariant measures.} Taking expectation in \eqref{ex2},  the mean
$$
m(t):=\mE X_t
$$
solves the closed ODE
\begin{align}\label{mean-ode-tanh}
m'(t)
=
F(m(t)),
\qquad
F(m):=-(\gamma+c)m+\sqrt c\,\tanh(\sqrt c\,m).
\end{align}
Let $\mu$ be an invariant  measure and set $m=m_\mu$.  The frozen
equation with parameter $\mu$ is
$$
\dif X_t
=
\Big[-(\gamma+c)X_t+\sqrt c\,\tanh(\sqrt c\,m)\Big]\dif t
+
\dif W_t.
$$
Its unique invariant measure is Gaussian with variance $1/[2(\gamma+c)]$ and mean
$$
\frac{\sqrt c}{\gamma+c}\tanh(\sqrt c\,m).
$$
Since the invariant measure must have mean $m$, we obtain
$$
m=\frac{\sqrt c}{\gamma+c}\tanh(\sqrt c\,m),
$$
which is precisely \eqref{self-con-tanh}.  This gives \eqref{inv-tanh}.
Conversely, any solution of \eqref{self-con-tanh} defines the Gaussian measure
\eqref{inv-tanh}, which is invariant for the frozen equation with parameter
equal to itself.  Hence it is invariant for \eqref{ex2}.

\smallskip
\noindent
{\it Step 3. Phase transition.}
To count the solutions of \eqref{self-con-tanh}, set
$$
y:=\sqrt c\,m,
\qquad
\beta:=\frac{\gamma+c}{c}.
$$
Then \eqref{self-con-tanh} is equivalent to
$$
\beta y=\tanh y.
$$

If $\gamma\geq0$, then $\beta\geq1$.  Since
$$
|\tanh y|<|y|,
\qquad y\neq0,
$$
the only solution is $y=0$, hence $m=0$.

If $-c<\gamma<0$, then $\beta\in(0,1)$.  The function
$$
y\mapsto \tanh y-\beta y
$$
is odd, has positive derivative at zero, and tends to $-\infty$ as
$y\to+\infty$.  Since $\tanh$ is strictly concave on $(0,\infty)$, there is
exactly one positive zero, and by symmetry  exactly one negative zero.
Thus  there are precisely   three invariant measures $\mu_{-}, \mu_{0}, \mu_{+}$, with first moments given by three solutions of \eqref{self-con-tanh}:
$$
m_-<0,
\qquad
m_0=0,
\qquad
m_+>0,
$$
and $m_-=-m_+$.

\smallskip
\noindent
{\it Step 4. Stability.}
For an invariant mean $m_i$,
$$
F'(m_i)
=
-(\gamma+c)+c\left(1-\tanh^2(\sqrt c\,m_i)\right),
$$
and hence
\begin{equation}\label{F'(m_i)}
-F'(m_i)
=
\gamma+c\tanh^2(\sqrt c\,m_i)=\lambda_i^{\rm mean}.
\end{equation}
For $m_0=0$, this gives $-F'(m_0)=\gamma<0$ in the regime $-c<\gamma<0$, and
therefore $m_0$ is unstable.

We now show that $m_-$ and $m_+$ are stable.  Let $y_i=\sqrt c\,m_i\neq0$.
Then
\begin{equation*}
	-F'(m_i)
	=
	\gamma+c\tanh^2(y_i).
\end{equation*}
It suffices to prove $\gamma+c\tanh^2(y_i)>0$, i.e.
$$
\beta>1-\tanh^2 y_i.
$$
Using $\beta y_i=\tanh y_i,$ this is equivalent to
$$
\frac{\tanh y_i}{y_i}>1-\tanh^2 y_i.
$$
Since both sides are even functions of $y_i$, it suffices to consider
$y_i>0$.  Define
$$
\Phi(y):=\tanh y-y\left(1-\tanh^2 y\right),
\qquad y>0.
$$
Then
\begin{align*}
\Phi'(y)
&=
\left(1-\tanh^2 y\right)
-
\left(1-\tanh^2 y\right)
+
2y\tanh y\left(1-\tanh^2 y\right)
\\
&=
2y\tanh y\left(1-\tanh^2 y\right)>0,
\qquad y>0.
\end{align*}
Moreover,
$$
\lim_{y\downarrow0}\Phi(y)=0.
$$
Thus $\Phi(y)>0$ for all $y>0$.  This proves the stability of the two non-zero equilibria.
\end{proof}

\subsection{Ergodicity in the uniqueness regime and critical slowing down}
We now prove the quantitative ergodic behavior in the uniqueness regime
$\gamma\geq0$, with particular attention to the critical case $\gamma=0$.

\begin{proof}[{\bf Proof of Theorem \ref{tan2}}]
	Recall from \eqref{mean-ode-tanh} that the mean
	$m(t):=\mE X_t$ satisfies
	\[
	m'(t)
	=
	F(m(t))
	=
	-(\gamma+c)m(t)
	+
	\sqrt c\,\tanh(\sqrt c\,m(t)).
	\]
We divide the proof into three steps.

\smallskip
\noindent
{\it Step 1. The case $\gamma>0$.} Since $\tanh x$ has the same sign as $x$ and
$|\tanh x|\leq |x|$, we have
\[
\frac{\dif}{\dif t}|m(t)|
\leq
-\gamma|m(t)|.
\]
Consequently,
\begin{equation}\label{mean-positive-gamma-tan2}
|m(t)|
\leq
|m(0)|\e^{-\gamma t},
\qquad t\geq0.
\end{equation}
Let $\bar X_t$ be the centered frozen Ornstein--Uhlenbeck process
started from the same initial distribution $\nu_0=\mathcal L_\xi$:
\[
\dif\bar X_t
=
-(\gamma+c)\bar X_t\,\dif t+\dif W_t.
\]
By Theorem~\ref{thm2}, the Lipschitz continuity of $\tanh$, and the
definition of $\sK_1$, for $t\geq2$,
\begin{align}\label{tan2-frozen-comparison}
{\bf d}_V(\mathcal L_{X_t},\nu_0\bar P^0_t)
&\leq
C_1\int_0^t
\sK_1(t-s)
\left|
\sqrt c\,\tanh(\sqrt c\,m(s))
\right|\dif s
\nonumber\\
&\leq
C_1\int_0^t\sK_1(t-s)|m(s)|\,\dif s
\nonumber\\
&\leq
C_1\int_{t-2}^t
\frac{|m(s)|}{\sqrt{t-s}}\,\dif s
+
C_1\int_0^{t-2}
\e^{-\lambda_V(t-s-1)}|m(s)|\,\dif s.
\end{align}
Fix $0<\lambda<\lambda_V\wedge\gamma$. By
\eqref{mean-positive-gamma-tan2}, we have
\begin{align*}
\int_{t-2}^t
\frac{|m(s)|}{\sqrt{t-s}}\,\dif s
&\leq
|m(0)|\e^{-\gamma t}
\int_0^2 r^{-1/2}\e^{\gamma r}\,\dif r
\\
&\leq
C_2\e^{-\gamma t}
\leq
C_2\e^{-\lambda t}.
\end{align*}
For the second term, using
\[
\e^{-\lambda_V(t-s-1)}\e^{-\gamma s}
=
\e^{\lambda_V}
\e^{-\lambda t}
\e^{-(\lambda_V-\lambda)(t-s)}
\e^{-(\gamma-\lambda)s},
\]
we obtain
\begin{align}\label{t2}
\int_0^{t-2}
\e^{-\lambda_V(t-s-1)}|m(s)|\,\dif s
&\leq
C_2\e^{-\lambda t}
\int_0^{t-2}
\e^{-(\lambda_V-\lambda)(t-s)}
\e^{-(\gamma-\lambda)s}\dif s
\no\\
&\leq
C_\lambda\e^{-\lambda t}.
\end{align}
Thus \eqref{tan2-frozen-comparison} gives
\[
{\bf d}_V(\mathcal L_{X_t},\nu_0\bar P^0_t)
\leq
C_\lambda\e^{-\lambda t}.
\]
Combining this with
(\ref{cw1})
and enlarging the constant to cover $0\leq t\leq2$, we get (\ref{cw-positive-gamma-exp}).

If $\lambda_V>\gamma$, then the preceding argument can be carried out
with $\lambda=\gamma$. Indeed,
\begin{align*}
\int_0^{t-2}
\e^{-\lambda_V(t-s-1)}\e^{-\gamma s}\dif s
&=
\e^{\lambda_V}\e^{-\gamma t}
\int_0^{t-2}
\e^{-(\lambda_V-\gamma)(t-s)}\dif s
\\
&\leq
C_3\e^{-\gamma t}.
\end{align*}
Hence the endpoint $\lambda=\gamma$ is also admissible.

\smallskip
\noindent
{\it Step 2. The  case $\gamma=0$.}
If $m(0)=0$, uniqueness for the mean equation implies
$m(t)\equiv0$. Therefore \eqref{ex2} reduces exactly to the centered
Ornstein--Uhlenbeck equation, and
\[
{\bf d}_V(\mathcal L_{X_t},\mu_0)
\leq
C_4\e^{-\lambda_Vt},
\qquad t\geq0.
\]
This proves \eqref{cw-critical-zero-mean-exp}.

Assume now that $m(0)\neq0$ and set
\[
y(t):=\sqrt c\,m(t).
\]
Then
\[
y'(t)
=
-c\big(y(t)-\tanh y(t)\big).
\]
Since $y-\tanh y$ has the same sign as $y$, the sign of $y(t)$ is
preserved and $|y(t)|$ decreases to zero. Moreover,
\[
\frac{y-\tanh y}{y^3}
\longrightarrow
\frac13,
\qquad y\to0.
\]
It follows that, for all sufficiently large $t$,
\[
\frac{\dif}{\dif t}\frac1{y(t)^2}
=
2c\frac{y(t)-\tanh y(t)}{y(t)^3}
\geq C_4.
\]
After integration and adjustment of the constant on bounded time
intervals, we obtain
\begin{equation*}
|m(t)|
\leq
C_4(1+t)^{-1/2},
\qquad t\geq0.
\end{equation*}
Applying \eqref{tan2-frozen-comparison} with $\gamma=0$, we deduce that
\begin{align*}
{\bf d}_V(\mathcal L_{X_t},\nu_0\bar P^0_t)
&\leq
C_4\int_{t-2}^t
\frac{(1+s)^{-1/2}}{\sqrt{t-s}}\dif s
\\
&\quad+
C_4\int_0^{t-2}
\e^{-\lambda_V(t-s-1)}(1+s)^{-1/2}\dif s.
\end{align*}
The first integral is bounded by $C_4(1+t)^{-1/2}$. Splitting the
second integral over $[0,t/2]$ and $[t/2,t-2]$ gives the same bound.
Therefore,
\[
{\bf d}_V(\mathcal L_{X_t},\nu_0\bar P^0_t)
\leq
C_5(1+t)^{-1/2}.
\]
This together with (\ref{cw1}) proves
(\ref{cw-critical-poly}).

\smallskip
\noindent
{\it Step 3. Optimality of the critical rate.}
Suppose that $\gamma=0$, $m(0)\neq0$ and $V\equiv1$. From the preceding asymptotics, we have
\[
\frac{\dif}{\dif t}\frac1{y(t)^2}
=
2c\frac{y(t)-\tanh y(t)}{y(t)^3}
\longrightarrow
\frac{2c}{3}.
\]
Consequently,
\[
\frac1{t\,y(t)^2}\longrightarrow\frac{2c}{3}.
\]
Since $y(t)=\sqrt c\,m(t)$, we obtain
\begin{equation}\label{critical-exact-mean-tan2}
t\,m(t)^2
\longrightarrow
\frac{3}{2c^2}.
\end{equation}
Let $\nu_t$ be the Gaussian measure with mean $m(t)$ and variance
$1/(2c)$. By centering the equation around $m(t)$ and using the
exponential convergence of the centered Ornstein--Uhlenbeck process, we have
\[
\|\mathcal L_{X_t}-\nu_t\|_{\rm TV}
\leq
C_6\e^{-ct}.
\]
Moreover, the total variation distance between two Gaussian measures
with the same variance $1/(2c)$ satisfies
\[
\|\nu_t-\mu_0\|_{\rm TV}
=
2\sqrt{\frac c\pi}\,|m(t)|
+
o(|m(t)|).
\]
Since $\e^{-ct}=o(t^{-1/2})$, \eqref{critical-exact-mean-tan2}
implies
\[
\begin{aligned}
\lim_{t\to\infty}
\sqrt t\,
\|\mathcal L_{X_t}-\mu_0\|_{\rm TV}
&=
2\sqrt{\frac c\pi}
\lim_{t\to\infty}\sqrt t\,|m(t)|
\\
&=
2\sqrt{\frac c\pi}
\sqrt{\frac{3}{2c^2}}
=
\sqrt{\frac{6}{\pi c}}.
\end{aligned}
\]
This proves \eqref{critical-exact-limit} and shows that the rate $(1+t)^{-1/2}$ is optimal.
\end{proof}

\subsection{Basin-dependent ergodicity in the phase-transition regime}
We finally enter the phase-transition regime \(-c<\gamma<0\), where
\(\mu_-\), \(\mu_0\), and \(\mu_+\) coexist. Since no global anchored
smallness can hold around any one of them, we combine the local anchored
principle with the closed mean equation: the latter determines the basin
through the sign of the initial mean, while the former transfers the frozen
Ornstein--Uhlenbeck decay around the selected stable equilibrium.
We give:

\begin{proof}[{\bf Proof of Theorem \ref{tan3}}]
	Recall that the mean $m(t)$ satisfies \eqref{mean-ode-tanh}.
In the regime $-c<\gamma<0$, the function $F$ has exactly three
zeros
\[
m_-<0=m_0<m_+.
\]
The scalar mean equation implies that
	positive initial means converge to $m_+$, negative initial means converge to
	$m_-$.
If $m(0)=0$, uniqueness for the solution gives
$m(t)\equiv0$. Thus the sign of the initial mean determines the
limiting equilibrium. Below, we divide the proof into three steps.

	\medskip
	\noindent\textit{Step 1: Exponential convergence of the mean.}
Assume $m(0)\neq0$, and let $m_i\in\{m_-,m_+\}$ be the equilibrium
determined by the sign of $m(0)$. Set
\(
u(t):=m(t)-m_i.
\)
Then
	$$u'(t) = m'(t) = F(m(t)) = F(m_i + u(t)).$$
Since $F(m_i)=0$, Taylor's formula gives
\[
u'(t)
=
F'(m_i)u(t)+O(u(t)^2)
=
-\lambda_i^{\rm mean}u(t)+O(u(t)^2),
\]
where, by \eqref{F'(m_i)},
\[
\lambda_i^{\rm mean}
=
\gamma+c\tanh^2(\sqrt c\,m_i)>0.
\]
By symmetry, this value is the same for $m_+$ and $m_-$:
$$
\lambda_+^{\rm mean}=\lambda_-^{\rm mean}=:\lambda^{\rm mean}.
$$
Since $u(t)\to0$, there exists $t_0>0$ such that
\[
\frac{\dif}{\dif t}|u(t)|
\leq
-\frac{\lambda^{\rm mean}}2|u(t)|,
\qquad t\geq t_0.
\]
Consequently,
\begin{equation}\label{preliminary-mean-rate-tan3}
|u(t)|
\leq
C_0\e^{-\lambda^{\rm mean}t/2},
\qquad t\geq0.
\end{equation}
Taylor's formula also gives
\[
u'(t)
=
-\lambda^{\rm mean}u(t)+r(t)u(t),
\qquad
|r(t)|\leq C_1|u(t)|.
\]
In view of \eqref{preliminary-mean-rate-tan3},
\[
|r(t)|
\leq
C_1\e^{-\lambda^{\rm mean}t/2},
\qquad
\int_0^\infty|r(t)|\dif t<\infty.
\]
If $u$ does not vanish, integration of this scalar equation gives
\[
|u(t)|
=
|u(0)|
\exp\left(
-\lambda^{\rm mean}t+\int_0^t r(s)\dif s
\right).
\]
If $u$ vanishes at some time, uniqueness of the mean equation implies
that it remains zero thereafter. Hence, in either case,
\begin{equation}\label{mean-rate-exact-tan3}
|m(t)-m_i|
\leq
C_2\e^{-\lambda^{\rm mean}t},
\qquad t\geq0.
\end{equation}

	\medskip
	\noindent\textit{Step 2: Comparison with the frozen Ornstein--Uhlenbeck process.}
	For $i\in\{-,+\}$, let $\bar X_t^i$ be the frozen process
\[
\dif\bar X_t^i
=
-(\gamma+c)(\bar X_t^i-m_i)\dif t+\dif W_t,
\]
started from the same initial distribution
$\nu_0=\mathcal L_\xi$. By Theorem~\ref{thm2} and the Lipschitz
continuity of $\tanh$,
\begin{align*}
{\bf d}_V(\mathcal L_{X_t},\nu_0\bar P_t^i)
&\leq
C_3\int_0^t
\sK_1(t-s)|m(s)-m_i|\dif s
\nonumber\\
&\leq
C_3\int_{t-2}^t
\frac{|m(s)-m_i|}{\sqrt{t-s}}\dif s
+
C_3\int_0^{t-2}
\e^{-\lambda_V(t-s-1)}
|m(s)-m_i|\dif s.
\end{align*}
Fix $0<\lambda<\lambda_V\wedge\lambda^{\rm mean}$. By
\eqref{mean-rate-exact-tan3}, we obtain
\begin{align*}
\int_{t-2}^t
\frac{|m(s)-m_i|}{\sqrt{t-s}}\dif s
&\leq
C_4\e^{-\lambda^{\rm mean}t}
\int_0^2r^{-1/2}\e^{\lambda^{\rm mean}r}\dif r
\\
&\leq
C_4 \e^{-\lambda t}.
\end{align*}
For the second part, similar to (\ref{t2}), we have
\begin{align*}
&\int_0^{t-2}
\e^{-\lambda_V(t-s-1)}
|m(s)-m_i|\dif s
\\
&\qquad\leq
C_5\e^{-\lambda t}
\int_0^{t-2}
\e^{-(\lambda_V-\lambda)(t-s)}
\e^{-(\lambda^{\rm mean}-\lambda)s}\dif s
\leq
C_{5}\e^{-\lambda t}.
\end{align*}
Thus, after adjusting the constant for $0\leq t\leq2$,
\[
{\bf d}_V(\mathcal L_{X_t},\nu_0\bar P_t^i)
\leq
C_{6}\e^{-\lambda t}.
\]
This together with \eqref{OU1},
and
the triangle inequality yields
\[
{\bf d}_V(\mathcal L_{X_t},\mu_i)
\leq
C_{7}\e^{-\lambda t}.
\]
Taking $i=+$ when $\mE\xi>0$ proves \eqref{r1}, while taking
$i=-$ when $\mE\xi<0$ proves \eqref{r2}.

If $\lambda_V>\lambda^{\rm mean}$, then
\begin{align*}
&\int_0^{t-2}
\e^{-\lambda_V(t-s-1)}
\e^{-\lambda^{\rm mean}s}\dif s
\\
&\qquad=
\e^{\lambda_V}\e^{-\lambda^{\rm mean}t}
\int_0^{t-2}
\e^{-(\lambda_V-\lambda^{\rm mean})(t-s)}\dif s
\leq
C_8\e^{-\lambda^{\rm mean}t}.
\end{align*}
The short-time part has the same decay rate by
\eqref{mean-rate-exact-tan3}. Therefore,
\[
{\bf d}_V(\mathcal L_{X_t},\mu_i)
\leq
C_8\e^{-\lambda^{\rm mean}t},
\]
so the endpoint $\lambda=\lambda^{\rm mean}$ is admissible in
\textup{(i)} and \textup{(ii)}.

\smallskip
\noindent
{\it Step 3. The centered initial.}
If $\mE\xi=0$, then $m(0)=0$, and uniqueness for the mean equation
implies
\(
m(t)\equiv0.
\)
Consequently, \eqref{ex2} reduces exactly to the centered
Ornstein--Uhlenbeck equation
\[
\dif X_t
=
-(\gamma+c)X_t\dif t+\dif W_t.
\]
Its invariant measure is $\mu_0$, and its exponential ergodicity
gives
\[
{\bf d}_V(\mathcal L_{X_t},\mu_0)
\leq
C_2\e^{-\lambda_Vt},
\qquad t\geq0.
\]
This proves \eqref{r0} and completes the whole proof.
\end{proof}

\bigskip

\noindent{\bf Declarations}

\medskip
\noindent
{\bf Funding.}
This work is supported by the National Key R\&D Program of China (No. 2023YFA1010103) and the National Natural Science Foundation of China (No. 12471140,  12401178).

\medskip
\noindent
{\bf Conflict of interest.}
The authors declared that they have no conflict of interest to this work.


\bigskip
\begin{thebibliography}{99}
\bibitem{Barbu2020} Barbu, V. and R{\"o}ckner, M.: From nonlinear Fokker-Planck equations to solutions of distribution dependent SDE. {\it Ann. Probab.} {\bf 48} (2020), 1902--1920.

\bibitem{BRS19} Bogachev, V.I., R{\"o}ckner, M. and Shaposhnikov, S.V.: Convergence in variation of solutions of nonlinear Fokker--Planck--Kolmogorov equations to stationary measures. {\it J. Funct. Anal.} {\bf 276} (2019), 3681--3713.



\bibitem{BGG13} Bolley, F., Gentil, I. and Guillin, A.: Uniform convergence to equilibrium for granular media. {\it Arch. Ration. Mech. Anal.} {\bf 208} (2013), 429--445.




\bibitem{CD18} Carmona, R. and Delarue, F.: {\it Probabilistic Theory of Mean Field Games with Applications I}. Springer, Cham (2018).

\bibitem{C2020} Carrillo, J.A., Gvalani, R.S., Pavliotis, G.A. and Schlichting, A.: Long-time behaviour and phase transitions for the McKean--Vlasov equation on the torus. {\it Arch. Ration. Mech. Anal.} {\bf 235} (2020), 635--690.



\bibitem{Carrillo2006} Carrillo, J.A., McCann, R.J. and Villani, C.: Contractions in the 2-Wasserstein length space and thermalization of granular media. {\it Arch. Ration. Mech. Anal.} {\bf 179} (2006), 217--263.

\bibitem{CGM08} Cattiaux, P., Guillin, A. and Malrieu, F.: Probabilistic approach for granular media equations in the non-uniformly convex case. {\it Probab. Theory Related Fields} {\bf 140} (2008), 19--40.

\bibitem{Cormier2025} Cormier, Q.: On the stability of the invariant probability measures of McKean--Vlasov equations. {\it Ann. Inst. Henri Poincar\'e Probab. Stat.} {\bf 61} (2025), 2405--2429.

\bibitem{Daw83} Dawson, D.A.: Critical dynamics and fluctuations for a mean-field model of cooperative behavior. {\it J. Stat. Phys.} {\bf 31} (1983), 29--85.

        \bibitem{Delgadino2021} Delgadino, M.G., Gvalani, R.S. and Pavliotis, G.A.: On the diffusive-mean field limit for weakly interacting diffusions exhibiting phase transitions. {\it Arch. Ration. Mech. Anal.} {\bf 241} (2021), 91--148.

\bibitem{Delgadino2023} Delgadino, M.G., Gvalani, R.S., Pavliotis, G.A. and Smith, S.A.: Phase transitions, logarithmic Sobolev inequalities, and uniform-in-time propagation of chaos for weakly interacting diffusions. {\it Comm. Math. Phys.} {\bf 401} (2023), 275--323.



\bibitem{Eberle2019} Eberle, A., Guillin, A. and Zimmer, R.: Quantitative Harris-type theorems for diffusions and McKean--Vlasov processes. {\it Trans. Amer. Math. Soc.} {\bf 371} (2019), 7135--7173.

\bibitem{GLM}Guillin, A., Le Bris, P. and Monmarch\'e, P.: Uniform in time propagation of chaos for the 2D vortex model and other singular stochastic systems. {\it J. Eur. Math. Soc.} {\bf 27} (2025),  2359--2386.

\bibitem{Guillin2022} Guillin, A., Liu, W., Wu, L. and Zhang, C.: Uniform Poincar\'e and logarithmic Sobolev inequalities for mean field particle systems. {\it Ann. Appl. Probab.} {\bf 32} (2022), 1590--1614.



\bibitem{Gvalani2020} Gvalani, R.S. and Schlichting, A.: Barriers of the McKean--Vlasov energy via a mountain pass theorem in the space of probability measures. {\it J. Funct. Anal.} {\bf 279} (2020), 108720.

\bibitem{Hammersley2021} Hammersley, W.R.P., Si{\'s}ka, D. and Szpruch, L.: McKean--Vlasov SDEs under measure dependent Lyapunov conditions. {\it Ann. Inst. Henri Poincar\'e Probab. Stat.} {\bf 57} (2021), 1032--1057.

\bibitem{Hao2024} Hao, Z., R{\"o}ckner, M. and Zhang, X.: Strong convergence of propagation of chaos for McKean--Vlasov SDEs with singular interactions. {\it SIAM J. Math. Anal.} {\bf 56} (2024), 2661--2713.

\bibitem{Huang2026} Huang, X., Kopfer, E. and Ren, P.: Log-Sobolev inequalities and exponential ergodicity for non-degenerate and degenerate McKean--Vlasov SDEs. {\it J. Funct. Anal.} {\bf 290} (2026), No. 111410.

\bibitem{Huang2025} Huang, X., Ren, P. and Wang, F.-Y.: Entropy-cost inequality for McKean--Vlasov SDEs with singular interactions. arXiv:2505.19787 (2025).


\bibitem{HuangWang2025}
Huang, X. and Wang, F.-Y.:
Log-Harnack inequality and Bismut formula for McKean--Vlasov SDEs with singularities in all variables.
{\it Math. Ann.} {\bf 393} (2025), 241--269.

\bibitem{JW} Jabin, P.E. and Wang, Z.: Quantitative estimates of propagation of chaos for stochastic systems with $W^{-1,\infty}$ kernels. {\it Invent. Math.} {\bf 214} (2018), 523--591.

\bibitem{Liang2021} Liang, M., Majka, M.B. and Wang, J.: Exponential ergodicity for SDEs and McKean--Vlasov processes with L\'evy noise. {\it Ann. Inst. Henri Poincar\'e Probab. Stat.} {\bf 57} (2021), 1665--1701.

\bibitem{Liu2021} Liu, W., Wu, L. and Zhang, C.: Long-time behaviors of mean-field interacting particle systems related to McKean--Vlasov equations. {\it Comm. Math. Phys.} {\bf 387} (2021), 179--214.

\bibitem{MT93} Meyn, S.P. and Tweedie, R.T.: Stability of Markovian processes III: Foster-Lyapunov criteria for continuous time processes. {\it Adv. Appl. Probab.} {\bf 25} (1993), 518--548.

\bibitem{MR} Monmarch\'e, P. and Reygner, J.: Local convergence rates for Wasserstein gradient flows and McKean--Vlasov equations with multiple stationary solutions. {\it Probab. Theory Related Fields} (2025), 1--59.

\bibitem{Ren2021} Ren, P. and Wang, F.-Y.: Exponential convergence in entropy and Wasserstein for McKean--Vlasov SDEs. {\it Nonlinear Anal.} {\bf 206} (2021), 112259.

\bibitem{RZ2021} R{\"o}ckner, M. and Zhang, X.: Well-posedness of distribution dependent SDEs with singular drifts. {\it Bernoulli} {\bf 27} (2021), 1131--1158.

\bibitem{SV} Stroock, D.W. and Srinivasa Varadhan, S.R.: {\it Multidimensional diffusion processes}. Springer-Verlag Berlin Heidelberg, 1997.

\bibitem{Sz91} Sznitman, A.-S.: Topics in propagation of chaos. In: {\it Ecole d'Et{\'e} de Probabilit{\'e}s de Saint-Flour XIX--1989}. Lecture Notes in Mathematics, vol. {\bf 1464}, Springer, Berlin (1991), 165--251.

\bibitem{T} Tugaut, J.: Convergence to the equilibria for self-stabilizing processes in double well landscape. {\it Ann. Probab.} {\bf 41} (2010), 1427--1460.



\bibitem{W18} Wang, F.-Y.: Distribution dependent SDEs for Landau type equations. {\it Stoch. Proc. Appl.} {\bf 128} (2018), 595--621.

\bibitem{Wang2023a} Wang, F.-Y.: Exponential ergodicity for non-dissipative McKean--Vlasov SDEs. {\it Bernoulli} {\bf 29} (2023), 1035--1062.

\bibitem{Wang2023b} Wang, F.-Y.: Exponential ergodicity for singular reflecting McKean--Vlasov SDEs. {\it Stoch. Proc. Appl.} {\bf 160} (2023), 265--293.

\bibitem{WangRen2025} Wang, F.-Y. and Ren, P.: {\it Distribution Dependent Stochastic Differential Equations}. World Scientific, Singapore, 2025.

\bibitem{XXZ26} Xia, P., Xie, L. and Zhang, X.: Gradient estimates for SDEs with singular and unbounded coefficients. arXiv:2604.00685 (2026).

\bibitem{XZ} Xie, L. and Zhang, X.: Ergodicity of stochastic differential equations with jumps and singular coefficients. {\it Ann. Inst. Henri Poincar\'e Probab. Stat.} {\bf 56} (2020), 175--229.

\bibitem{Zhang2023} Zhang, S.-Q.: Existence and non-uniqueness of stationary distributions for distribution dependent SDEs. {\it Electron. J. Probab.} {\bf 28} (2023),  1--34.

\bibitem{Zhang2025} Zhang, S.-Q.: A local bifurcation theorem for McKean--Vlasov diffusions. {\it J. Funct. Anal.} {\bf 289} (2025), No. 111144.

\bibitem{Zhang11} Zhang, X.: Stochastic differential equations with Sobolev diffusion and singular drift and their applications. {\it Ann. Appl. Probab.} {\bf 21} (2011), 1803--1830.
\end{thebibliography}
\end{document}